\documentclass[reqno]{amsart}

\usepackage[english]{babel}
\usepackage{amsmath,amssymb,amsthm}
\usepackage{hyperref}
\hypersetup{hidelinks}
\usepackage[centering]{geometry}
\allowdisplaybreaks[2]
\usepackage{xcolor}

\newtheorem{theorem}{Theorem}[section]   
\newtheorem*{theorem*}{Theorem}          
\newtheorem{lemma}[theorem]{Lemma}
\newtheorem{proposition}[theorem]{Proposition}

\theoremstyle{definition}

\newtheorem{example}[theorem]{Example}

\newtheorem{remark}{Remark}[section]

\numberwithin{equation}{section}

\title[Hypersurfaces of $\mathbb{H}^2\times\mathbb{H}^2$ with constant sectional curvature]
{Hypersurfaces of $\mathbb{H}^2\times\mathbb{H}^2$ with constant sectional curvature}
\thanks{}

\author{Haizhong Li, Luc Vrancken, Xianfeng Wang \and Zeke Yao}


\date{}
\begin{document}

\begin{abstract}
In this paper, we classify the hypersurfaces 
of $\mathbb{H}^2\times\mathbb{H}^2$ with constant sectional curvature. 
In contrast to the case of $\mathbb{S}^2\times\mathbb{S}^2$, the resulting examples for $\mathbb{H}^2\times\mathbb{H}^2$ exhibit  greater diversity, and we construct a special example with non-constant product angle function. For $\mathbb{S}^2\times\mathbb{S}^2$, however, the product angle function of any constant sectional curvature hypersurface is identically zero. 
As a byproduct, we also classify the hypersurfaces of  $\mathbb{H}^2\times\mathbb{H}^2$ with constant product angle function and constant mean curvature (or constant scalar curvature). 
\end{abstract}

\maketitle

\section{Introduction}\label{sect:1}

From the point of view of Riemannian geometry, the classification problem of submanifolds with constant sectional curvature in different  ambient spaces is an important and attractive topic in differential geometry. 
When the ambient space is a real space form, the problem has a long history and has been well understood. It is remarkable that explicit examples of such 
submanifolds were constructed in \cite{DT2002} by developing the classical Ribaucour transformation. We refer to the monograph  \cite{DT2019} for more details on the  classification of constant sectional curvature hypersurfaces in real space forms. 
In real space forms, non-flat complex space forms, and some other irreducible symmetric spaces, there are also some important classification results for hypersurfaces with constant sectional curvature, Einstein hypersurfaces, or conformally flat hypersurfaces (see, e.g.,  \cite{dDM,CR1982,Chen,CDFV,I-R,K-R,Kon,L,M,N-P,R} and the references therein). 
Additionally, hypersurfaces with constant sectional curvature in cylinders over space forms have also been classified (see, e.g., \cite{AEG,CT2025,M-T} and the references therein).

However, in general, if the curvature tensor of the ambient space is not simple, it is a very challenging problem to obtain 
the classification of hypersurfaces with 
constant sectional curvature, if one uses the Gauss and Codazzi 
equations directly. In this aspect, the first three authors introduced a new approach  in 2013, 
the so-called {\it Tsinghua principle}, which applies the Codazzi equation 
and the Ricci identity in an indirect way to obtain some nice linear equations involving 
the components of the second fundamental form. This new approach has been applied successfully to classify submanifolds with constant sectional curvature in different ambient spaces (see, e.g., \cite{A-L-V-W,LMVVW,LVWY}).

 In recent years, the study of the submanifolds of $\mathbb{S}^2\times\mathbb{S}^2$ and 
 $\mathbb{H}^2\times\mathbb{H}^2$ has become very active. 
Both  $\mathbb{S}^2\times\mathbb{S}^2$ and 
$\mathbb{H}^2\times\mathbb{H}^2$ are very typical $4$-dimensional Hermitian symmetric manifolds. 
When the codimension is $2$, there are many interesting  results on totally geodesic surfaces (cf. \cite{CN-1,CN-2,Kle}), 
minimal surfaces (cf. \cite{C-U,GVWX,LMVVW,Tor,T-U2,VW}, etc.), and surfaces with parallel mean curvature vector (cf. \cite{T-U1}, etc.) in $\mathbb{S}^2\times\mathbb{S}^2$ or 
$\mathbb{H}^2\times\mathbb{H}^2$. 
 When the codimension is $1$, it is remarkable that the homogeneous hypersurfaces and the isoparametric hypersurfaces of $\mathbb{S}^2\times\mathbb{S}^2$ and $\mathbb{H}^2\times\mathbb{H}^2$ were classified in \cite{Ur} and \cite{GMY}, respectively. There have also been some classification results on Hopf hypersurfaces and the hypersurfaces with isometric Reeb flow of  $\mathbb{S}^2\times\mathbb{S}^2$ and $\mathbb{H}^2\times\mathbb{H}^2$ (see \cite{GHMY} and \cite{ZGHY}, etc.). In particular, the  authors of the present paper classified the hypersurfaces of $\mathbb{S}^2\times\mathbb{S}^2$ 
 with constant sectional curvature in \cite{LVWY}. It turns out that any such hypersurface is a parallel hypersurface of a minimal hypersurface in $\mathbb{S}^2\times\mathbb{S}^2$, and the authors established a one-to-one correspondence between such minimal hypersurfaces and solutions to the famous ``sinh-Gordon equation'' 
 $
 (\frac{\partial^2}{\partial u^2}+\frac{\partial^2}{\partial v^2})h
 =-\tfrac{1}{\sqrt{2}}\sinh(\sqrt{2}h).
 $ 
 The main purpose of this paper is to classify the hypersurfaces of $\mathbb{H}^2\times\mathbb{H}^2$ with constant sectional curvature. 
 Owing to the curvature properties and noncompactness of  $\mathbb{H}^2\times\mathbb{H}^2$, 
in contrast to $\mathbb{S}^2\times\mathbb{S}^2$, the resulting examples for $\mathbb{H}^2\times\mathbb{H}^2$ exhibit more diversity.

Before stating our main results, we first recall some basic properties of the hypersurfaces in $\mathbb{H}^2\times\mathbb{H}^2$. Note that there is a natural  product structure $P$ on $\mathbb{H}^2\times\mathbb{H}^2$ defined by $$P(v_1,v_2)=(v_1,-v_2)$$ for 
any tangent vector fields $v_1, v_2$ on $\mathbb{H}^2$. 
For any orientable hypersurface $M$ of $\mathbb{H}^2\times\mathbb{H}^2$ with 
a unit normal vector field $N$, we can introduce an important function $C$ defined by 
$$
C=g(PN, N), 
$$ 
where $g$ denotes the standard product metric 
on $\mathbb{H}^2\times\mathbb{H}^2$. 
The geometry of hypersurfaces in $\mathbb{H}^2\times\mathbb{H}^2$ is closely related 
to this function $C$. Hereafter, we always assume that $M$ is an orientable hypersurface of $\mathbb{H}^2\times\mathbb{H}^2$, 
and we call $C$ the {\it product angle function} of $M$. 

As the first main result in this paper, we classify the hypersurfaces of $\mathbb{H}^2\times\mathbb{H}^2$ with constant sectional curvature.
\begin{theorem}\label{thm:1.1}
Let $M$ be a hypersurface of $\mathbb{H}^2\times \mathbb{H}^2$ with constant sectional curvature $\kappa$. 
Then $\kappa=-1/2$ and $M$ is either 
\begin{enumerate}
\item[(1)]
an open part of Example \ref{exam:4.1ss1ss11ab1}; or

\item[(2)]
an open part of $M_{k_1,1/k_1}^{1/2}$ given in Example \ref{exam:3.4fg1} for some constant $0<k_1\leq1$; or

\item[(3)]
an open part of one of the hypersurfaces described by Theorems \ref{thm:5.10a1}--\ref{thm:5.13a1}.

\end{enumerate}
\end{theorem}

\begin{remark}\label{rem:1.1}
(i) Compared with $\mathbb{S}^2\times \mathbb{S}^2$, there is a hypersurface of $\mathbb{H}^2\times \mathbb{H}^2$ with constant sectional curvature and non-constant product angle function, which is Case (1) in Theorem \ref{thm:1.1}. The examples in 
Case (2) and Case (3) of Theorem \ref{thm:1.1} have constant product angle function $C=0$.

(ii) In Theorem \ref{thm:5.10a1} (or Theorem \ref{thm:5.11a1}), for any function depending on one variable that is not identically zero, 
one can construct a hypersurface of $\mathbb{H}^2\times \mathbb{H}^2$ with constant sectional curvature 
$-1/2$. More precisely, for any function $\beta_1:I\to\mathbb{R}$ that is not identically zero (respectively, any function $\beta_2:I\to\mathbb{R}$ that is not identically zero), we construct a metric $g$ and two $(1,1)$-tensors $A,\tilde{P}$ as in \eqref{eqn:3.1ff3}--\eqref{eqn:3.1ff5} (respectively, as in \eqref{eqn:3.1ffff3}--\eqref{eqn:3.1ffff5}). Then, by the existence and uniqueness theorems in \cite{K,L-T-V}, up to an isometry of $\mathbb{H}^2\times\mathbb{H}^2$, there exists a unique hypersurface in $\mathbb{H}^2\times\mathbb{H}^2$ with constant sectional curvature $-1/2$ such that $C=0$, and its induced metric and shape operator are given by $g$ and $A$, respectively. 

(iii) In Theorem \ref{thm:5.12a1}, for any function $h_3:(u,v)\to\mathbb{R}$ satisfying
the ``cosh-Gordon equation''
\begin{equation*}
\frac{\partial^2 h_3}{\partial u \partial v}
=4\sqrt{2}\cosh(h_3),
\end{equation*}
we construct a metric $g$ and two 
$(1,1)$-tensors $A,~\tilde{P}$ (see \eqref{eqn:3.1fffff3}, \eqref{eqn:3.1fffff4} and \eqref{eqn:3.1fffff5}). Then by the existence and uniqueness theorems in \cite{K,L-T-V}, up to an isometry of $\mathbb{H}^2\times\mathbb{H}^2$, there exists a unique hypersurface in $\mathbb{H}^2\times\mathbb{H}^2$ with constant sectional curvature
such that $C=0$ and its induced metric and shape operator are given by $g$ and $A$,
respectively.

(iv) In Theorem \ref{thm:5.13a1}, for any smooth function $h_4:(u,v)\mapsto\mathbb{R}$ that is not identically zero and satisfies the ``sinh-Gordon equation''
\begin{equation*}
\frac{\partial^2 h_4}{\partial u \partial v}
=4\sqrt{2}\sinh(h_4),
\end{equation*}
we construct a metric $g$ and two
$(1,1)$-tensors $A,~\tilde{P}$ (see \eqref{eqn:3.1ffffff3}, \eqref{eqn:3.1ffffff4} and \eqref{eqn:3.1ffffff5}). Then, by the existence and uniqueness theorems in \cite{K,L-T-V}, up to an isometry of $\mathbb{H}^2\times\mathbb{H}^2$, there exists a unique hypersurface in $\mathbb{H}^2\times\mathbb{H}^2$ with constant sectional curvature such that $C=0$ and its induced metric and shape operator are given by $g$ and $A$, respectively.

(v) In \cite{LVWY}, it was proved that any hypersurface of $\mathbb{S}^2\times\mathbb{S}^2$ 
with constant sectional curvature must be a parallel hypersurface of a minimal hypersurface with $C=0$, and there is 
a one-to-one correspondence between such minimal hypersurface and the solution to the
``sinh-Gordon equation''
$
(\frac{\partial^2}{\partial u^2}+\frac{\partial^2}{\partial v^2})h
=-\tfrac{1}{\sqrt{2}}\sinh(\sqrt{2}h).
$
However, this phenomenon is no longer true in $\mathbb{H}^2\times\mathbb{H}^2$. 
\end{remark} 

As the second main result, we obtain a complete classification of 
the hypersurfaces of $\mathbb{H}^2\times\mathbb{H}^2$
with constant product angle function and constant mean curvature. 
\begin{theorem}\label{thm:1.2}
Let $M$ be a connected hypersurface of $\mathbb{H}^2\times \mathbb{H}^2$ with constant product angle function $C$. Then
$M$ has constant mean curvature if and only if either
\begin{enumerate}
\item[(1)]
$M$ is an open part of $M_\Gamma$ given in Example \ref{exam:3.1f1f2}, where $\Gamma$ is a curve of $\mathbb{H}^2$ with constant curvature; or

\item[(2)]
$M$ is an open part of $M_\tau$ given in Example \ref{exam:3.8f1f2} for some $\tau<-1$; or

\item[(3)]
$M$ is an open part of $M_{1,1}^c$ given in Example \ref{exam:3.4fg1} for some $c\in(0,1)$; or

\item[(4)]
$M$ is an open part of $M_{1,-1}^c$ given in Example \ref{exam:3.4fg1} for some $c\in(0,1)$; or

\item[(5)]
$M$ is an open part of $M_{k_1,k_1}^{1/2}$ given in Example \ref{exam:3.4fg1} for some constant $k_1\neq1$; or

\item[(6)]
$M$ is an open part of a minimal hypersurface described by Theorem \ref{thm:3.1}.
\end{enumerate}
\end{theorem}

\begin{remark}\label{rem:1.2}
In Theorem \ref{thm:3.1}, 
for any non-constant function $h:(u,v)\rightarrow \mathbb{R}$ satisfying
the ``sine-Gordon equation'' 
\begin{equation*}
	(\frac{\partial^2}{\partial u^2}+\frac{\partial^2}{\partial v^2})h
	=\tfrac{1}{\sqrt{2}}\sin(\sqrt{2}h),
\end{equation*}
we construct a metric $g$ and two
$(1,1)$-tensors $A,~\tilde{P}$ (see \eqref{eq:2.ab17}, \eqref{eqn:A} and \eqref{defP}), then by the existence and uniqueness theorems in  \cite{K,L-T-V}, up to an isometry of $\mathbb{H}^2\times\mathbb{H}^2$, there exists a unique minimal hypersurface in $\mathbb{H}^2\times\mathbb{H}^2$
such that $C=0$ and its induced metric and shape operator  are given by $g$ and $A$,
respectively.
\end{remark}

\begin{remark}\label{rem:1.3}
The conditions of constant product angle function and constant mean curvature in Theorem \ref{thm:1.2} are necessary.

(i) In fact, we can define the map $\Phi:\Omega\subset \mathbb{R}^3\hookrightarrow\mathbb{H}^2\times\mathbb{H}^2$ as follows:
\begin{equation}\label{eqn:1.1a1a1}
\begin{aligned}
&p(t,r)=\cosh r\,(\cosh t,\sinh t,0)+\sinh r\,(0,0,1),\\
&q(t,s)=\cosh s\,(\cosh t,\sinh t,0)+\sinh s\,(0,0,1).
\end{aligned}
\end{equation}
Then hypersurface $\Phi(\Omega)$ is minimal, but its product angle function is $C=\frac{\cosh^2 s-\cosh^2 r}{\cosh^2 s+\cosh^2 r}$, which is not constant.

(ii) In Example \ref{exam:3.4fg1}, when $k_1(r)$ and $k_2(s)$ are non-constant,
then the hypersurface $M_{k_1,k_2}^c$ has constant product angle function $C=1-2c$,
but it does not have constant mean curvature.
\end{remark}

As the third main result, we obtain a complete classification of
the hypersurfaces of $\mathbb{H}^2\times\mathbb{H}^2$
with constant product angle function and constant scalar curvature.

\begin{theorem}\label{thm:1.3}
Let $M$ be a connected hypersurface of $\mathbb{H}^2\times \mathbb{H}^2$ with constant product angle function $C$. Then 
$M$ has constant scalar curvature if and only if either 
\begin{enumerate}
\item[(1)]
$M$ is an open part of $M_\Gamma$ given in Example \ref{exam:3.1f1f2}; or 

\item[(2)]
$M$ is an open part of $M_\tau$ given in Example \ref{exam:3.8f1f2} for some $\tau<-1$; or

\item[(3)]
$M$ is an open part of $M_{1,1}^c$ given in Example \ref{exam:3.4fg1} for some $c\in(0,1)$; or

\item[(4)]
$M$ is an open part of $M_{1,-1}^c$ given in Example \ref{exam:3.4fg1} for some $c\in(0,1)$; or

\item[(5)]
$M$ is an open part of $M_{k_1,1/k_1}^{1/2}$ given in Example \ref{exam:3.4fg1} for some constant $0<k_1<1$; or

\item[(6)]
$M$ is an open part of one of the hypersurfaces described by Theorems \ref{thm:5.10a1}--\ref{thm:5.13a1}.

\end{enumerate}
\end{theorem}

\begin{remark}\label{rem:1.4}
The conditions of constant product angle function and constant scalar curvature in Theorem \ref{thm:1.3} are necessary.

(i) In fact, Example \ref{exam:4.1ss1ss11ab1} has constant scalar curvature, but its product angle function $C$ is not constant.

(ii) In Example \ref{exam:3.4fg1}, when $k_1(r)$ and $k_2(s)$ are non-constant, then the
hypersurface $M_{k_1,k_2}^c$ has constant product angle function $C=1-2c$, but 
it does not have constant scalar curvature. 
\end{remark}

The paper is organized as follows. In Sect. \ref{sect:2}, we collect some basic properties of $\mathbb{H}^2\times \mathbb{H}^2$ and some preliminaries of the geometry of the hypersurfaces of $\mathbb{H}^2\times \mathbb{H}^2$. In Sect. \ref{sect:3}, we introduce some canonical examples of hypersurfaces in
$\mathbb{H}^2\times\mathbb{H}^2$. In particular, we construct a special example with constant sectional curvature and non-constant product angle function (see Example \ref{exam:4.1ss1ss11ab1}). In Sect. \ref{sect:4}, we give a complete classification of the minimal hypersurfaces in $\mathbb{H}^2\times \mathbb{H}^2$ with $C=0$ (see Theorem \ref{thm:3.2aa}). 
In Sect. \ref{sect:5}, we present the proof of Theorem \ref{thm:1.1}. One of the key ingredients is to prove that the constant sectional curvature can only be $-\frac{1}{2}$, which is achieved by applying the so-called Tsinghua principle 
(see Proposition \ref{prop:4.5}). In Sect. \ref{sect:6}, we present the proofs of Theorems \ref{thm:1.2} 
and \ref{thm:1.3}. In Sect. \ref{sect:7}, we give a complete classification of the hypersurfaces 
in $\mathbb{S}^m\times \mathbb{S}^n$ ($m\geq3$, $n\geq2$), 
$\mathbb{H}^m\times \mathbb{H}^n$ ($m\geq3$, $n\geq2$), and $\mathbb{S}^m\times \mathbb{H}^n$ ($m,n\geq2$) with constant sectional curvature.

\vskip 2mm

\textbf{Acknowledgments:} 
H. Li was supported by NSFC Grant No. 12471047 and NSFC Grant No. W2521103. 
L. Vrancken was supported by NSFC-FWO mobility project VS05726N. 
X. Wang was supported by Young Scientific and Technological Talents (Level Two) in Tianjin, Natural Science Foundation of Tianjin, China (Grant No. 25JCZDJC01110) and the Fundamental Research Funds for the Central Universities (Grant Nos. 050-63261176, 050-63253091).
Z. Yao was supported by NSFC Grant No. 12401061.

\section{Preliminaries} \label{sect:2}
\subsection{The geometric structure on $\mathbb{H}^{2}\times \mathbb{H}^{2}$}~

Let $\mathbb{R}_{1}^{3}$ be the three-dimensional Minkowski space with the
Lorentzian metric $\langle\cdot,\cdot\rangle $.
The hyperbolic plane of curvature $-1$ can be defined as the following subset
of $\mathbb{R}_{1}^{3}$: 
$$
\mathbb{H}^{2}=\left\{\left(x_{1}, x_{2}, x_{3}\right) \in \mathbb{R}_{1}^{3}
\mid-x_{1}^{2}+x_{2}^{2}+x_{3}^{2}=-1, x_{1}>0\right\}.
$$
The standard complex structure $J$ on $\mathbb{H}^{2}$ is defined by
$$J_xu=x\boxtimes u,$$
for all $x\in \mathbb{H}^{2}$ and all $u\in T_x\mathbb{H}^{2}$, where $\boxtimes$
is the Lorentzian cross product defined by
$$
(a_1,a_2,a_3)\boxtimes (b_1,b_2,b_3)=(a_3b_2-a_2b_3, a_3b_1-a_1b_3, a_1b_2-a_2b_1).
$$

Throughout the paper we will consider $\mathbb{H}^2 \times  \mathbb{H}^2$
embedded isometrically in $\mathbb{R}_{1}^{3} \times \mathbb{R}_{1}^{3} \cong \mathbb{R}_{2}^{6}$,
with the induced Riemannian product metric which we denote by $g$.
We define two complex structures on $\mathbb{H}^2\times \mathbb{H}^2$ by
$$
J_1=(J,J),\quad J_2=(J,-J),
$$
which endow $\mathbb{H}^2 \times  \mathbb{H}^2$ with two K\"ahler structures. 
The isometry group of $\mathbb{H}^{2}\times \mathbb{H}^{2}$ is
$$
\operatorname{Iso}\left(\mathbb{H}^{2}\times \mathbb{H}^{2}\right)=\left\{\left(\begin{array}{cc}
A_{1} & 0 \\
0 & A_{2}
\end{array}\right),\left(\begin{array}{cc}
0 & B_1 \\
B_{2} & 0
\end{array}\right) \mid A_{1}, A_{2}, B_1, B_{2} \in \mathrm{O}^{+}(1,2)\right\},
$$
where $\mathrm{O}^{+}(1,2)$ denotes the orthochronous Lorentz group.

The product structure $P$ on $\mathbb{H}^2\times\mathbb{H}^2$ is defined by
$P: T(\mathbb{H}^2\times\mathbb{H}^2)\rightarrow T(\mathbb{H}^2\times\mathbb{H}^2)$ such that
$$
P(v_1,v_2)=(v_1,-v_2), \quad \forall\, v_1, v_2\in T\mathbb{H}^2.
$$
Obviously, we have $P=-J_1J_2=-J_2J_1$, $P^2=\mathrm{Id}$ and
\begin{equation}\label{eqn:2.1}
g(PY,Z)=g(Y,PZ), \quad \forall\,Y,Z\in T(\mathbb{H}^2\times\mathbb{H}^2).
\end{equation}
Moreover, $\bar{\nabla} P=0$, where $\bar{\nabla}$ is the Levi-Civita connection on
$\mathbb{H}^2\times\mathbb{H}^2$.

The curvature tensor $\bar{R}$ of $\mathbb{H}^2\times\mathbb{H}^2$
with the Riemannian product metric is given by
\begin{align*}
g(\bar{R}(U,Y)Z,W)=-\tfrac{1}{2}&\big\{g(Y,Z)g(U,W)
-g(U,Z)g(Y,W)\\
&+g(PY,Z)g(PU,W)-g(PU,Z)g(PY,W)\big\},
\end{align*}
where $U, Y, Z, W\in T(\mathbb{H}^2\times\mathbb{H}^2)$. Thus, $\mathbb{H}^2\times\mathbb{H}^2$
is an Einstein manifold with scalar curvature $-4$ and nonpositive sectional curvature.

\subsection{Hypersurfaces of $\mathbb{H}^{2}\times \mathbb{H}^{2}$}~

Let $M$ be an orientable hypersurface of $\mathbb{H}^2\times\mathbb{H}^2$ with
$N$ a unit normal vector field and still $g$ the induced metric on $M$.
Then, with respect to the product structure $P$, we define by
\begin{align*}
C&:=g(PN,N)=g(J_1 N,J_2 N),\\
X&:=PN-CN,
\end{align*}
the product angle function $C: M\rightarrow\mathbb{R}$ and a vector field
$X$ tangent to $M$. It is clear that $-1\leq C\leq 1$ and $\|X\|^2:=g(X,X)=1-C^2$.

For any tangential vector field $Y$ of $M$, acting by the product structure $P$,
we have the following decomposition
$$
PY=T Y+\mu(Y)N,
$$
where $T Y$ and $\mu(Y)N$ are the tangential
and normal parts of $PY$. Thus $T$ is a 
tensorial field of type $(1,1)$, and $\mu$ is a $1$-form over $M$.
Moreover, $\mu(Y)=g(PY, N)$.

Let $\nabla$ be the Levi-Civita connection of the induced metric $g$ on $M$.
The Gauss and Weingarten formulae are given by 
\begin{align*}
\bar{\nabla}_Y Z=\nabla_Y Z+g(AY,Z)N, \quad \bar{\nabla}_YN=-AY,
\end{align*}
where $A$ is the shape operator of $M$.

Now, the Gauss and Codazzi equations of $M$ are given by
\begin{equation}\label{eqn:2.2}
\begin{aligned}
\operatorname{R}(U, Y)Z=&-\tfrac{1}{2}\big[g(Y,Z)U-g(U,Z)Y+g(TY,Z)TU
-g(TU,Z)TY\big]\\
&+g(AY,Z)AU-g(AU,Z)AY,
\end{aligned}
\end{equation}
\begin{equation}\label{eqn:2.3}
(\nabla_YA)Z-(\nabla_Z A)Y=-\tfrac{1}{2}\big[g(Y,X)TZ-g(Z,X)TY\big],
\end{equation}
where $U, Y, Z\in TM$, and $R$ denotes the curvature tensor of $M$ with respect
to the metric $g$.

It follows from \eqref{eqn:2.2} that the Ricci curvature tensor of $M$ is given by
\begin{equation}\label{eqn:2.4a11}
\begin{aligned}
g({\rm Ric}(Y),Z)
=-\tfrac{1}{2}\big[g(Y,Z)-Cg(TY,Z)
+g(Y,X)g(Z,X)\big]+3Hg(AY,Z)-g(A^2Y,Z),
\end{aligned}
\end{equation}
where $H=\frac{1}{3}{\rm tr}A$ denotes the mean curvature of $M$. 
It follows that the scalar curvature $\rho$ of $M$ is given by
\begin{equation}\label{eqn:2.6a11}
\rho=-2+9H^{2}-\|A\|^{2}.
\end{equation}

Let $\nabla^2 A$ denote the second covariant derivative of $A$, i.e.,
$$
(\nabla^{2}A)(U,Y,Z):=\nabla_U[(\nabla_Y A)Z]-(\nabla_{\nabla_U Y}
A)Z-(\nabla_Y A){\nabla_U Z}.
$$
Then the Ricci identity states that
\begin{equation}\label{eqn:ric}
g((\nabla^{2}A)(U,Y,Z),W)-g((\nabla^{2}A)(Y,U,Z),W)
=-g( R(U,Y)Z,AW)-g( R(U,Y)W,AZ).
\end{equation}

Now, by an argument similar to the proof of Lemma 2.1 of \cite{LVWY}, 
we can obtain the following Lemma \ref{lemma:2.1} which describes some properties 
of the function $C$, the vector field $X$ and the covariant derivatives of $T$ and $\mu$.

\begin{lemma}\label{lemma:2.1}
Let $M$ be an orientable hypersurface of $\mathbb{H}^2\times\mathbb{H}^2$ 
with $A$ the shape operator associated 
to the unit normal field $N$. Then the gradient of $C$, the
covariant derivatives of $X$, $T$ and $\mu$ are given by
\begin{equation}\label{eqn:2.4}
\begin{split}
&\nabla C=-2 AX,\ \ \nabla_{Y}X=C AY -T A Y,\\
&\left(\nabla_Y T\right) Z=g( AY,Z)X+\mu(Z) A Y,\\
&\left(\nabla_Y \mu\right) Z=C g( AY,Z)-g( TZ,AY).
\end{split}
\end{equation}
\end{lemma}
\begin{proof}
The formulae of the gradient of $C$ and the covariant derivative of $X$ can be directly found
in Lemma 2.1 of \cite{GMY}. For the remaining two formulae, by the definitions of $T$, $\mu$
and the Gauss and Weingarten formulae, we get
\begin{align*}
\left(\nabla_{Y} T\right) Z &=(\bar{\nabla}_{Y} (T Z))^{\top}-T (\nabla_{Y} Z) 
=[\bar{\nabla}_{Y}(P Z-\mu(Z) N)]^{\top}-T (\nabla_{Y} Z) \\
&=[P \bar{\nabla}_{Y} Z-\mu(Z) \bar{\nabla}_{Y} N]^{\top}-T (\nabla_{Y} Z) 
=g(AY,Z)(P N)^{\top}+\mu(Z) A Y\\
&=g(AY,Z)X+\mu(Z) A Y,
\end{align*}
\begin{align*}
\left(\nabla_{Y} \mu\right) Z&=Y \mu(Z)-\mu\left(\nabla_{Y} Z\right) \\
&=g(\bar{\nabla}_{Y} PZ,N)+g(PZ,\bar{\nabla}_{Y} N)-g(P\bar{\nabla}_{Y} Z-g(AY,Z) PN,N)\\
&=-g(TZ,AY)+Cg(AY,Z),
\end{align*}
where $\cdot^\top$ denotes the tangential part.
\end{proof}

\begin{lemma}\label{lemma:2.1aaa}
Let $M$ be an orientable hypersurface of $\mathbb{H}^2\times\mathbb{H}^2$
with $A$ the shape operator associated
to the unit normal field $N$. Then the Hessian of $C$,
the Laplacian of $C$ and the divergence of $X$ are given by
\begin{equation}\label{eqn:2.4aaa}
({\rm Hess}(C))(Y,Z)=-2g((\nabla_Y A)X,Z)-2Cg(A^2Y,Z)+2g(TAY,AZ).
\end{equation}
\begin{equation}\label{eqn:2.5aaa}
\triangle C=-6g(\nabla H,X)-2C\|A\|^2+2{\rm tr}(TA^2).
\end{equation}
\begin{equation}\label{eqn:2.6aaa}
{\rm div}X=3CH-{\rm tr}(TA).
\end{equation}
\end{lemma}
\begin{proof}
\setlength{\abovedisplayskip}{3pt}
\setlength{\belowdisplayskip}{3pt}
\setlength{\abovedisplayshortskip}{1pt}
\setlength{\belowdisplayshortskip}{1pt}
By using Lemma \ref{lemma:2.1}, the Gauss and Weingarten formulae, we get
\begin{align*}
({\rm Hess}(C))(Y,Z)&=Y(ZC)-(\nabla_YZ)C=-2Yg(AX,Z)+2g(AX,\nabla_YZ)\\
&=-2g(\nabla_Y{AX},Z)=-2g((\nabla_YA)X,Z)-2g(\nabla_Y X,AZ)\\
&=-2g((\nabla_Y A)X,Z)-2Cg(A^2Y,Z)+2g(TAY,AZ).
\end{align*}
\begin{align*}
\triangle C&=\sum_i({\rm Hess}(C))(e_i,e_i)
=\sum_i\{-2g((\nabla_{e_i} A)X,{e_i})-2Cg(A^2{e_i},{e_i})+2g(TA{e_i},A{e_i})\}\\
&=-2\sum_i\{g((\nabla_X A){e_i},{e_i})\}-2C\|A\|^2+2{\rm tr}(TA^2)
=-6g(\nabla H,X)-2C\|A\|^2+2{\rm tr}(TA^2).
\end{align*}
\begin{align*}
{\rm div}X&=\sum_ig(\nabla_{e_i}X,e_i)=\sum_ig(CAe_i-TAe_i,e_i)
=3CH-{\rm tr}(TA).
\end{align*}
\end{proof}

In \cite{GMY}, Gao, Ma and the fourth author classified the hypersurfaces with $C^2\equiv 1$ as follows.
\begin{lemma}[cf. Lemma 3.3 of \cite{GMY}]\label{lem:2.2}
Let $M$ be a hypersurface of $\mathbb{H}^2\times\mathbb{H}^2$ with $C^2\equiv 1$.
Then, up to isometries of $\mathbb{H}^2\times\mathbb{H}^2$, $M$ is locally the
product of a curve $\Gamma$ in $\mathbb{H}^2$ and an open
subset of $\mathbb{H}^2$.
\end{lemma}

\subsection{A canonical frame related to hypersurfaces of
$\mathbb{H}^2\times\mathbb{H}^2$}\label{sect:2.3}~

In this subsection, we assume that $|C|< 1$ holds on $M$,
and choose an appropriate local orthonormal frame field $\{E_1, E_2, E_3\}$
of $M$ as follows:
\begin{equation}\label{eqn:2.5}
E_1=\frac{J_1N+J_2 N}{\sqrt{2(1+C)}},\
\ E_2=\frac{J_1N-J_2 N}{\sqrt{2(1-C)}},\ \ E_3=\frac{X}{\sqrt{1-C^{2}}}.
\end{equation}
This frame has the following properties:
\begin{equation}\label{eqn:2.6}
\left\{
\begin{aligned}
&P E_1=E_1,\ \ P E_2=-E_2,\ \ P E_3=\sqrt{1-C^{2}}N-C E_3,\ \ PN=CN+\sqrt{1-C^{2}}E_3.\\
&TE_1=E_1,\ \ TE_2=-E_2,\ \ TE_3=-CE_3,\ \ \mu(E_1)=\mu(E_2)=0,\ \ \mu(E_3)=\sqrt{1-C^{2}}.
\end{aligned}\right.
\end{equation}

We assume that
\begin{equation}\label{eqn:2.7}
\begin{aligned}	
&A E_1=b_1 E_1+b_2 E_2+b_3 E_3, \\
&A E_2=b_2 E_1+b_4 E_2+b_5 E_3,\\
&A E_3=b_3 E_1+b_5E_2+b_6 E_3,
\end{aligned}
\end{equation}
where $b_1$, $b_2$, $b_3$, $b_4$, $b_5$ and $b_6$ are smooth
functions on $M$. 
By \eqref{eqn:2.4a11}, under the frame $\{E_1,E_2,E_3\}$, the Ricci tensor is expressed by
\begin{equation}\label{eqn:2.8a1212}
\begin{aligned}
&{\rm Ric}(E_1)=(\tfrac{-1+C}{2}
+b_1b_4+b_1b_6-b_2^{2}-b_3^{2}) E_1+
(b_2b_6-b_3 b_5) E_2+(b_3b_4-b_2b_5) E_3, \\
&{\rm Ric}(E_2)=(b_2b_6-b_3 b_5) E_1
+(\tfrac{-1-C}{2}+b_1b_4+b_4b_6-b_2^{2}-b_5^{2}) E_2+(b_1b_5-b_2b_3) E_3,\\
&{\rm Ric}(E_3)=(b_3b_4-b_2b_5) E_1+(b_1b_5-b_2b_3) E_2+(-1+b_1b_6+b_4b_6-b_3^{2}-b_5^{2}) E_3.
\end{aligned}
\end{equation}

When $C$ is constant and $C$ is not equal to $\pm1$ on $M$, we have the following lemma:

\begin{lemma}\label{lemma:2.3}
Let $M$ be a hypersurface of $\mathbb{H}^2\times \mathbb{H}^2$ with constant $C\neq\pm1$.
Then, under the local orthonormal frame field $\{E_1, E_2, E_3\}$,
the following relations hold:
\begin{equation}\label{eqn:2.25}
\begin{aligned}
&b_3=b_5=b_6=0, \ \ A E_1=b_1 E_1+b_2 E_2,\ A E_2=b_2 E_1+b_4 E_2,\ A E_3=0,\\		
&\nabla_{E_1}E_1=b_1\sqrt{\tfrac{1-C}{1+C}}E_3,\
\nabla_{E_1}E_2=-b_2\sqrt{\tfrac{1+C}{1-C}}E_3,\ \nabla_{E_1}E_3=-b_1\sqrt{\tfrac{1-C}{1+C}}E_1
+b_2\sqrt{\tfrac{1+C}{1-C}}E_2,\\
&\nabla_{E_2} E_1=b_2\sqrt{\tfrac{1-C}{1+C}}E_3,\
\nabla_{E_2} E_2=-b_4\sqrt{\tfrac{1+C}{1-C}}E_3,\ \nabla_{E_2}E_3=-b_2\sqrt{\tfrac{1-C}{1+C}}E_1
+b_4\sqrt{\tfrac{1+C}{1-C}}E_2,\\
&\nabla_{E_3}E_1=\nabla_{E_3}E_2=\nabla_{E_3}E_3=0.
\end{aligned}
\end{equation}
Moreover, the Gauss and Codazzi equations are equivalent to the following equations:
\begin{equation}\label{eqn:2.26}
E_3b_1=-\tfrac{\sqrt{1-C^2}}{2}+b_1^{2}\sqrt{\tfrac{1-C}{1+C}}-b_2^2\sqrt{\tfrac{1+C}{1-C}},
\end{equation}
\begin{equation}\label{eqn:2.27}
E_3b_2=b_2\big(b_1\sqrt{\tfrac{1-C}{1+C}}-b_4\sqrt{\tfrac{1+C}{1-C}}\big),
\end{equation}
\begin{equation}\label{eqn:2.28}
E_3b_4=\tfrac{\sqrt{1-C^2}}{2}+b_2^{2}\sqrt{\tfrac{1-C}{1+C}}-b_4^2\sqrt{\tfrac{1+C}{1-C}},
\end{equation}
\begin{equation}\label{eqn:2.29}
E_1b_2-E_2b_1=0,	
\end{equation}
\begin{equation}\label{eqn:2.30}
E_1b_4-E_2b_2=0.
\end{equation}
\end{lemma}
\begin{proof}
By the assumption that $C$ is constant and $\nabla C=-2 AX$, we have $b_3=b_5=b_6=0$.

The product structures and the almost complex structures in
$\mathbb{S}^2\times\mathbb{S}^2$ and $\mathbb{H}^2\times\mathbb{H}^2$ share the same
properties, i.e., the product structures are parallel, and are symmetric with respect to the
product metric, and the almost complex structures are parallel. Then the connections in \eqref{eqn:2.25}
can be obtained in exactly the same way as (2.11) of \cite{LVWY}.

Because the Codazzi equation for a hypersurface of $\mathbb{H}^2\times\mathbb{H}^2$ 
is exactly opposite to the Codazzi equation for a hypersurface of $\mathbb{S}^2\times\mathbb{S}^2$, 
\eqref{eqn:2.26}--\eqref{eqn:2.30} can be obtained in a way analogous to
(2.12)--(2.16) of \cite{LVWY}.
\end{proof}

\section{Some canonical examples}\label{sect:3}

In this section, we introduce some canonical examples of hypersurfaces in
$\mathbb{H}^2\times\mathbb{H}^2$, where Example \ref{exam:3.1f1f2}, Example \ref{exam:3.8f1f2}
and Example \ref{exam:3.4fg1} have appeared in \cite{GMY}.
\begin{example}\label{exam:3.1f1f2}
For any smooth regular curve $\Gamma$ of $\mathbb{H}^2$, one can define a hypersurface
in $\mathbb{H}^2\times\mathbb{H}^2$ by
$$
M_\Gamma:=\left\{(x,y)\in \mathbb{H}^2\times\mathbb{H}^2
~|~x\in \Gamma,\ y\in\mathbb{H}^2\right\}.
$$
\end{example}

Let $k_{\Gamma}$ be the curvature of $\Gamma$ in $\mathbb{H}^2$. 
Then, hypersurface $M_\Gamma$ has constant product angle function $C=1$, and its principal curvatures are $k_{\Gamma}$, $0$ and $0$.

\begin{lemma}[cf. Lemma 3.2 of \cite{GMY}]\label{lemma:3.1ff1}
For any smooth regular curve $\Gamma$ of $\mathbb{H}^2$, $M_\Gamma$ has constant principal curvatures
if and only if $\Gamma$ is an open part of a curve in $\mathbb{H}^2$ with constant curvature.
\end{lemma}

\begin{example}\label{exam:3.8f1f2}
For any given $\tau<-1$, we define $M_\tau=\{(p,q)\in
\mathbb{H}^2\times\mathbb{H}^2|\ \langle p,q\rangle=\tau\}$.
\end{example}

According to the calculations in Example 3.10 of \cite{GMY}, we know that
$M_\tau$ has three distinct constant principal curvatures
$\sqrt{\tfrac{\tau-1}{2(\tau+1)}}$, $\sqrt{\tfrac{\tau+1}{2(\tau-1)}}$ and $0$.
Its product angle function is $C=0$.
For any $\tau<-1$, $M_\tau$ cannot be a minimal hypersurface or a hypersurface of $\mathbb{H}^2\times\mathbb{H}^2$ with constant sectional curvature.

In the following, we define a class of hypersurfaces which generalizes the Example 3.4 of \cite{GMY}.

\begin{example}\label{exam:4.1}
For any two non-constant functions $f_1(t)$, $f_2(t)$ and any smooth curves $\gamma_1(r)$ and $\gamma_2(s)$
of $\mathbb{H}^2$ with $r,s$ being their arc length parameters, a hypersurface in
$\mathbb{H}^2\times\mathbb{H}^2$ can be defined by the immersion
$\Psi:\Omega\subset\mathbb{R}^3\rightarrow\mathbb{H}^2\times\mathbb{H}^2$:
$(t,r,s)\rightarrow(p(t,r),q(t,s))$
\begin{equation}\label{eqn:3.1ss1}
\begin{aligned}
&p(t,r)=\cosh(f_1(t))\gamma_1(r)+\sinh(f_1(t))N_1(r),\\
&q(t,s)=\cosh(f_2(t))\gamma_2(s)+\sinh(f_2(t))N_2(s),
\end{aligned}	
\end{equation}
where $N_1(r)$ and $N_2(s)$ are unit normal vector fields of $\gamma_1(r)$ and
$\gamma_2(s)$ in $\mathbb{H}^2$, respectively.
Let $k_1(r)$ and $k_2(s)$ be the curvatures of
curves $\gamma_1(r)$ and $\gamma_2(s)$ in $\mathbb{H}^2$, respectively. 
We restrict $\Omega$ to an open set on which
$$
(f_1')^2+(f_2')^2>0,\qquad
(\cosh f_1-k_1\sinh f_1)(\cosh f_2-k_2\sinh f_2)\neq0.
$$
These are exactly the immersion conditions for this parametrization.

Since $\{\gamma_1(r), \frac{d\gamma_1(r)}{dr}, N_1(r)\}$ and $\{\gamma_2(s),
\frac{d\gamma_2(s)}{ds}, N_2(s)\}$ are two orthonormal frames
of $\mathbb{R}_1^{3}$ along curves $\gamma_1(r)$ and $\gamma_2(s)$ respectively, it follows that
\begin{equation}\label{eqn:3.2ss1}
\begin{aligned}
&\frac{d^2\gamma_1(r)}{dr^2}=\gamma_1(r)+k_1(r)N_1(r),\ \
\frac{dN_1(r)}{dr}=-k_1(r)\frac{d\gamma_1(r)}{dr},\\
&\frac{d^2\gamma_2(s)}{ds^2}=\gamma_2(s)+k_2(s)N_2(s),\ \
\frac{dN_2(s)}{ds}=-k_2(s)\frac{d\gamma_2(s)}{ds}.
\end{aligned} 	
\end{equation}%
Now, we have an orthonormal frame field $\left\{E_{1}, E_{2}, E_{3}\right\}$
of hypersurface $\Psi(t,r,s)$ defined as

\begin{equation*}
\begin{aligned}
E_1=&\frac{1}{\cosh(f_1(t))-\sinh(f_1(t))k_1(r)}
\frac{\partial}{\partial r}(p(t,r),q(t,s))
=(\frac{d\gamma_1(r)}{dr},0),\\
E_2=&\frac{1}{\cosh(f_2(t))-\sinh(f_2(t))k_2(s)}
\frac{\partial }{\partial s}(p(t,r),q(t,s))
=(0,\frac{d\gamma_2(s)}{ds}),
\end{aligned}	
\end{equation*}
\begin{equation*}
\begin{aligned}
E_3=&\frac{1}{\sqrt{(f_1'(t))^2+(f_2'(t))^2}}\frac{\partial}{\partial t}(p(t,r),q(t,s))\\
=&\frac{1}{\sqrt{(f_1'(t))^2+(f_2'(t))^2}}(f_1'(t)\sinh(f_1(t))\gamma_1(r)
+f_1'(t)\cosh(f_1(t))N_1(r),\\
&\qquad\qquad  f_2'(t)\sinh(f_2(t))\gamma_2(s)
+f_2'(t)\cosh(f_2(t))N_2(s)).\\
\end{aligned}	
\end{equation*}
The unit normal vector field $N$ of $\Psi(t,r,s)$ is given by
\begin{equation}\label{eqn:3.3ss1}
\begin{aligned}
N&=\frac{1}{\sqrt{(f_1'(t))^2+(f_2'(t))^2}}(f_2'(t)\sinh(f_1(t))\gamma_1(r)
+f_2'(t)\cosh(f_1(t))N_1(r),\\
&\qquad \qquad -f_1'(t)\sinh(f_2(t))\gamma_2(s)
-f_1'(t)\cosh(f_2(t))N_2(s)).
\end{aligned}
\end{equation}
It follows that $\Psi(t,r,s)$ has $C=g(PN,N)=\frac{(f_2'(t))^2-(f_1'(t))^2}{(f_1'(t))^2+(f_2'(t))^2}$.
By direct calculations, we get
\begin{equation}\label{eqn:3.4ss1}
\begin{aligned}
&AE_1=-\frac{f_2'(t)}{\sqrt{(f_1'(t))^2+(f_2'(t))^2}}\frac{\sinh(f_1(t))-k_1(r)\cosh(f_1(t))}
{\cosh(f_1(t))-k_1(r)\sinh(f_1(t))}E_1,\\
&AE_2=\frac{f_1'(t)}{\sqrt{(f_1'(t))^2+(f_2'(t))^2}}\frac{\sinh(f_2(t))-k_2(s)\cosh(f_2(t))}
{\cosh(f_2(t))-k_2(s)\sinh(f_2(t))}E_2,\\
&AE_3=-\frac{f_1'(t)f_2''(t)-f_1''(t)f_2'(t)}{((f_1'(t))^2+(f_2'(t))^2)^{\frac{3}{2}}}E_3.
\end{aligned}
\end{equation}
\end{example}

\begin{remark}\label{rem:3.5fg1}
Hypersurfaces constructed by \eqref{eqn:3.1ss1} are curvature-adapted
hypersurfaces of $\mathbb{H}^2\times \mathbb{H}^2$ (see Proposition 3.1 of \cite{HLYZ}).
When taking $k_1(r)=k_2(s)=1$,
hypersurfaces constructed by \eqref{eqn:3.1ss1} are minimal if and only if the functions $f_1(t)$ and
$f_2(t)$ satisfy
$$
(f_2'(t)-f_1'(t))\big((f_1'(t))^2+(f_2'(t))^2\big)=f_1'(t)f_2''(t)-f_1''(t)f_2'(t).
$$
When taking $k_1(r)=-k_2(s)=1$,
hypersurfaces constructed by \eqref{eqn:3.1ss1} are minimal if and only if the functions $f_1(t)$ and
$f_2(t)$ satisfy
$$
(f_2'(t)+f_1'(t))\big((f_1'(t))^2+(f_2'(t))^2\big)=f_1'(t)f_2''(t)-f_1''(t)f_2'(t).
$$

\end{remark}

Now, we choose specific $f_1(t),f_2(t)$, $\gamma_1(r),N_1(r)$ and $\gamma_2(s),N_2(s)$ in \eqref{eqn:3.1ss1}, so that we can get a hypersurface of $\mathbb{H}^2\times \mathbb{H}^2$ with constant sectional curvature $-\frac{1}{2}$.

\begin{example}\label{exam:4.1ss1ss11ab1}
In \eqref{eqn:3.1ss1}, we take 
$$
\begin{aligned}
&f_1(t)=\ln\Big(\sqrt{2}\cosh(\tfrac{t}{\sqrt{2}})+\sqrt{\cosh(\sqrt{2}t)}\Big),\ \
f_2(t)={\rm arcsinh}\Big(-\sqrt{2}\sinh(\tfrac{t}{\sqrt{2}})\Big)-{\rm arcsinh}(1),\\
&\gamma_1(r)=(\cosh(r),\sinh(r),0), \ \
\gamma_2(s)=(\sqrt{2},\cos(s),\sin(s)),\\
&N_1(r)=(0,0,1),\ \ N_2(s)=(1,\sqrt{2}\cos(s),\sqrt{2}\sin(s)).
\end{aligned}	
$$
Then we have the immersion:
$\Phi_1:\Omega\subset\mathbb{R}^3\rightarrow\mathbb{H}^2\times\mathbb{H}^2$:
$(t,r,s)\rightarrow(p(t,r),q(t,s))$, where 
\begin{equation}\label{eqn:jr}
\begin{aligned}
&p(t,r)=\Big(\sqrt{2}\cosh(\tfrac{t}{\sqrt{2}})\cosh(r),\sqrt{2}\cosh(\tfrac{t}{\sqrt{2}})\sinh(r),\sqrt{\cosh(\sqrt{2}t)}\Big),\\
&q(t,s)=\Big(\sqrt{\cosh(\sqrt{2}t)},-\sqrt{2}\sinh(\tfrac{t}{\sqrt{2}})\cos(s),-\sqrt{2}\sinh(\tfrac{t}{\sqrt{2}})\sin(s)\Big).
\end{aligned}	
\end{equation}
By \eqref{eqn:2.2}, \eqref{eqn:3.3ss1} and \eqref{eqn:3.4ss1}, one can verify that such hypersurface constructed by \eqref{eqn:jr} has constant sectional curvature $-\frac{1}{2}$, and
$C={\rm sech}(\sqrt{2}t)$.
\end{example}

\begin{remark}\label{rem:axisregular-verified} 
The polar parametrization \eqref{eqn:jr} is an immersion if and only if $t\neq0$. 
Indeed, $\partial_s\Phi_1=0$ at $t=0$.  This is a coordinate
singularity and not a geometric singularity.  Define $R^2=x^2+y^2$,
$a=\sqrt{2+R^2}$ and $z=\sqrt{1+R^2}$, and put
\begin{equation}\label{eq:axis-extension-verified}
\widehat\Phi(r,x,y)=\big((a\cosh(r),a\sinh(r),z),(z,x,y)\big),
\qquad (r,x,y)\in\mathbb R^3.
\end{equation}
For $R>0$ this agrees with \eqref{eqn:jr} on the branch $t<0$ via
$R=-\sqrt2\sinh(t/\sqrt2)$, $x=R\cos(s)$, $y=R\sin(s)$. 
The induced metric is
\begin{equation}\label{eq:axis-metric-verified}
\widehat\Phi^*g=(2+R^2)\,dr^2+dx^2+dy^2
-\frac{(x\,dx+y\,dy)^2}{2+R^2}. 
\end{equation} 
Comparing directly with the hyperboloid
$\{F\in\mathbb{R}_1^4:\langle F,F\rangle=-2,\ F_0>0\}$ under the parametrization
$F=(a\cosh(r),a\sinh(r),x,y)$, we see that \eqref{eq:axis-metric-verified} is
exactly the induced metric. Furthermore, $\widehat\Phi$ is a smooth complete
isometric embedding of $\mathbb{H}^3(-1/2)$ into $\mathbb{H}^2\times\mathbb{H}^2$. 
Consequently, the original polar chart is incomplete once the axis is removed,
whereas the extended image is complete. 
	
For $R>0$, the formula already computed in Example~\ref{exam:4.1ss1ss11ab1}
gives
$$
C=\operatorname{sech}(\sqrt2t)=\frac{1}{1+R^2}.
$$
The right-hand side is smooth on $\mathbb R^3$, and the angle function
is smooth for the globally defined normal of the extended embedding.
Thus $C=(1+R^2)^{-1}$ also at $R=0$, where $C=1$.
The canonical frame \eqref{eqn:2.5} is undefined on that axis, even
though the hypersurface itself is regular.  All arguments using that
frame must therefore first be made on $|C|<1$ and then extended
geometrically.
\end{remark}

As another special case of Example \ref{exam:4.1} by choosing $f_1(t)=\sqrt{c}t$ and $f_2(t)=\sqrt{1-c}t$
for some constant $0<c<1$, we have the following example.
\begin{example}\label{exam:3.4fg1}
For any constant $0<c<1$, and any smooth curves $\gamma_1(r)$ and $\gamma_2(s)$
of $\mathbb{H}^2$ with $r,s$ being their arc length parameters, we define a hypersurface $M_{k_1,k_2}^c$ in
$\mathbb{H}^2\times\mathbb{H}^2$ by the following immersion
$\Phi:\Omega\subset\mathbb{R}^3\rightarrow\mathbb{H}^2\times\mathbb{H}^2$:
$(t,r,s)\rightarrow(p(t,r),q(t,s))$,
\begin{equation}\label{eqn:3.1fg12}
\begin{aligned}
&p(t,r)=\cosh(\sqrt{c}t)\gamma_1(r)+\sinh(\sqrt{c}t)N_1(r),\\
&q(t,s)=\cosh(\sqrt{1-c}t)\gamma_2(s)+\sinh(\sqrt{1-c}t)N_2(s),
\end{aligned}	
\end{equation}
where $N_1(r)$ and $N_2(s)$ are unit normal vector fields of $\gamma_1(r)$ and
$\gamma_2(s)$ in $\mathbb{H}^2$, respectively.
Let $k_1(r)$ and $k_2(s)$ be the curvatures of
curves $\gamma_1(r)$ and $\gamma_2(s)$ in $\mathbb{H}^2$, respectively.
The domain is restricted by
$$
(\cosh(\sqrt c t)-k_1(r)\sinh(\sqrt c t))
(\cosh(\sqrt{1-c}t)-k_2(s)\sinh(\sqrt{1-c}t))\ne0.
$$

It follows from \eqref{eqn:3.4ss1} that $M_{k_1,k_2}^c$ has constant
$C=g(PN,N)=1-2c$,
and its principal curvatures are
$$
\begin{aligned}
0,\ \ -\sqrt{1-c}\frac{\sinh(\sqrt{c}t)-\cosh(\sqrt{c}t)k_1(r)}
{\cosh(\sqrt{c}t)-\sinh(\sqrt{c}t)k_1(r)},\ \
\sqrt{c}\frac{\sinh(\sqrt{1-c}t)-\cosh(\sqrt{1-c}t)k_2(s)}
{\cosh(\sqrt{1-c}t)-\sinh(\sqrt{1-c}t)k_2(s)}.
\end{aligned}
$$

\begin{lemma}[cf. Lemma 3.5 and Remark 3.1 of \cite{GMY}]\label{lemma:3.5fg12}
Let $M$ be an open part of hypersurface $M_{k_1,k_2}^c$ which is constructed by 
\eqref{eqn:3.1fg12}. Then 
\begin{enumerate}
\item[(1)]
$M$ is minimal if and only if $c=\frac{1}{2}$, $k_1(r)$ and $k_2(s)$
are the same constant functions.

\item[(2)]
$M$ has constant sectional curvature if and only if $c=\frac{1}{2}$, $k_1(r)$ and
$k_2(s)$ are constant functions which satisfy $k_1(r) k_2(s)=1$.
Moreover, the sectional curvature of $M$ is $-\frac{1}{2}$.

\item[(3)]
$M$ has constant principal curvatures if and only if $k_1(r)^2=k_2(s)^2=1$. Up to ambient congruence, it suffices to consider the representatives $M_{1,1}^c$ and $M_{1,-1}^c$. 
\end{enumerate}
\end{lemma}
\end{example}

\section{Minimal hypersurfaces with $C=0$}\label{sect:4}
In this section, we give the classification of the minimal hypersurfaces with $C=0$, which
plays an important role in the proof of Theorem \ref{thm:1.2}. 

\begin{lemma}\label{lemma:2.1aaa1}
Let $M$ be a minimal hypersurface of $\mathbb{H}^2\times \mathbb{H}^2$ with $C=0$.
Then, under the frame $\{E_1,E_2,E_3\}$ defined by \eqref{eqn:2.5}, we have two cases
for $b_1,b_2$ and $b_4$ on $M$ as follows:

\noindent{\bf Case 1}:
$$
b_2=0,
$$

\noindent{\bf Case 2}:
\begin{equation}\label{eq:j}
\begin{aligned}
b_1&=-b_4=-\tfrac{1}{\sqrt{2}}\frac{\sinh(\frac{t-l}{\sqrt{2}})\cosh(\frac{t-l}{\sqrt{2}})}
{\sinh^2(\frac{t-l}{\sqrt{2}})+\cos^2(\frac{h}{\sqrt{2}})},\\
b_2&=-\tfrac{1}{\sqrt{2}}\frac{\sin(\frac{h}{\sqrt{2}})\cos(\frac{h}{\sqrt{2}})}
{\sinh^2(\frac{t-l}{\sqrt{2}})+\cos^2(\frac{h}{\sqrt{2}})},
\end{aligned}	
\end{equation}
where $t$ satisfies $E_3t=1$, and $h,l:M\to\mathbb{R}$ satisfy $E_3h=E_3l=0$. Moreover, $h$ is not identically equal to $\frac{k\pi}{\sqrt{2}}$ for any $k\in\mathbb{Z}$. 

\end{lemma}
\begin{proof}
By $H=\frac{1}{3}(b_1+b_4)=0$ and $C=0$, equations \eqref{eqn:2.26} and \eqref{eqn:2.27} become
\begin{equation*}
\begin{split}
E_3b_1&=-\tfrac{1}{2}+b_1^2-b_2^2,\\
E_3b_2&=2b_1b_2.
\end{split}	
\end{equation*}

When $b_2$ is constant on $M$, we have $b_1b_2=0$. Furthermore,
if $b_1=0$, then we get a contradiction from the first equation.
Thus, if $b_2$ is constant on $M$, then it holds that $b_2=0$, which is the {\bf Case 1}.

When $b_2$ is not constant on $M$, let $F=b_1+\mathbf{i}b_2$, we get
$$
E_3F=E_3(b_1+\mathbf{i}b_2)=-\frac{1}{2}+(b_1+\mathbf{i}b_2)^2=-\frac{1}{2}+F^2.
$$
This equation can be directly integrated along the integral curves of $E_3$
to obtain $F=-\frac{1}{\sqrt{2}}\tanh(\frac{t+z_0}{\sqrt{2}})$,
where $z_0=-l+\mathbf{i}h$, $h$ and $l$ are functions on $M$ which satisfy $E_3h=E_3l=0$.
By $F=b_1+\mathbf{i}b_2$, it follows that $b_1$ and $b_2$ are its real and imaginary parts,
which corresponds to the {\bf Case 2}.
\end{proof}

In the following, we deal with the {\bf Case 2} of Lemma \ref{lemma:2.1aaa1} 
by finding appropriate local coordinates. 
We assume that
\begin{equation}\label{UVT-H2}
\left(
  \begin{array}{c}
    U \\
    V \\
    T \\
  \end{array}
\right)
=
\left(
  \begin{array}{ccc}
    a_1 & a_2 & a_1d_1+a_2d_2 \\
    -a_2 & a_1 & -a_2d_1+a_1d_2 \\
    0 & 0 & 1 \\
  \end{array}
\right)
\left(
  \begin{array}{c}
    E_1 \\
    E_2 \\
    E_3 \\
  \end{array}
\right),
\end{equation}
where $a_1,a_2,d_1,d_2$ are smooth functions on $M$.

\begin{lemma}\label{lemma:2.1a1}
$\{U,V,T\}$ is a local coordinate frame if and only if
the following system of equations is satisfied:
\begin{equation}\label{eq:2.ab1}
E_3a_1=-a_1b_1-a_2b_2,\ \ E_3a_2=a_1b_2-a_2b_1,
\end{equation}
\begin{equation}\label{eq:2.ab2}
E_1a_2+E_2a_1=0,\ \ E_1a_1-E_2a_2=0,
\end{equation}
\begin{equation}\label{eq:2.ab3}
E_3d_1=b_1d_1-b_2d_2,\ \ E_3d_2=d_1b_2+d_2b_1,
\end{equation}
\begin{equation}\label{eq:2.ab4}
E_1d_2-E_2d_1=2b_2.
\end{equation}

\end{lemma}
\begin{proof}
We just need to show that our system of equations is equivalent to the vanishing of all the commutators $[T,U]$, $[T,V]$ and $[U,V]$.
Using $[E_i,E_j]=\nabla_{E_i} {E_j}-\nabla_{E_j} {E_i}$, $C=0$
and \eqref{eqn:2.25}, we get
$$
[E_1,E_2]=-2b_2E_3,\ [E_1,E_3]=-b_1E_1+b_2E_2,\ [E_2,E_3]=-b_2E_1+b_4E_2=-b_2E_1-b_1E_2.
$$
Since \eqref{UVT-H2} is the same as (3.2) of \cite{LVWY}, the proof of Lemma \ref{lemma:2.1a1} is exactly the same as the proof of Lemma 3.2 of \cite{LVWY}. 
\end{proof}

We now give a special solution to the system of equations \eqref{eq:2.ab1}--\eqref{eq:2.ab4}, which provides us with  adapted coordinates.

\begin{lemma}\label{lemma:2.1a2}
Assume that $b_1,b_2$ and $b_4$ are given by \eqref{eq:j} on $M$,
and let $f$ be any function satisfying $E_3f=-1$. Then the functions $a_1,a_2,d_1,d_2$ defined by
\begin{equation}\label{eq:2.ab7}
\begin{aligned}
(a_1+\mathbf{i}a_2)^2=\frac{1}{1-2(b_1-\mathbf{i}b_2)^2}, \ \ d_1=E_1f,\ d_2=E_2f,
\end{aligned}
\end{equation}
are a solution to the system of differential equations \eqref{eq:2.ab1}--\eqref{eq:2.ab4}.
\end{lemma}
\begin{proof}
Observe that the system of equations splits into two independent
sub-systems, \eqref{eq:2.ab1}--\eqref{eq:2.ab2} and \eqref{eq:2.ab3}--\eqref{eq:2.ab4}.

Assume now that $a_1,a_2$ are functions of $b_1,b_2$. Using the equations
\eqref{eqn:2.26}--\eqref{eqn:2.30} and $C=0$, we re-express the first sub-system with respect to
the $b_1,b_2$ to obtain that
\begin{equation}\label{eq:2.ab8}
(b_1^2-b_2^2-\frac{1}{2})\frac{\partial a_1}{\partial b_1}+2b_1b_2\frac{\partial a_1}{\partial b_2}
=-a_1b_1-a_2b_2,
\end{equation}
\begin{equation}\label{eq:2.ab9}
(b_1^2-b_2^2-\frac{1}{2})\frac{\partial a_2}{\partial b_1}+2b_1b_2\frac{\partial a_2}{\partial b_2}
=a_1b_2-a_2b_1,
\end{equation}
\begin{equation}\label{eq:2.ab10}
\frac{\partial a_2}{\partial b_1}E_1b_1+\frac{\partial a_2}{\partial b_2}E_1b_2
+\frac{\partial a_1}{\partial b_1}E_2b_1+\frac{\partial a_1}{\partial b_2}E_2b_2=0,
\end{equation}
\begin{equation}\label{eq:2.ab11}
\frac{\partial a_1}{\partial b_1}E_1b_1+\frac{\partial a_1}{\partial b_2}E_1b_2
-\frac{\partial a_2}{\partial b_1}E_2b_1-\frac{\partial a_2}{\partial b_2}E_2b_2=0.
\end{equation}
Equations \eqref{eq:2.ab10} and \eqref{eq:2.ab11} are satisfied if $(a_1,a_2)$ satisfy the Cauchy-Riemann equations with respect to $b_1,b_2$, i.e.,
$$
\frac{\partial a_2}{\partial b_1}=\frac{\partial a_1}{\partial b_2},\ \
\frac{\partial a_2}{\partial b_2}=-\frac{\partial a_1}{\partial b_1}.
$$

Therefore, we assume that $a_1(z)+\mathbf{i}a_2(z)$ is a holomorphic function with respect 
to $z=b_1-\mathbf{i}b_2$. Then we have
$$
(z^2-\tfrac{1}{2})\frac{\partial}{\partial z}(a_1+\mathbf{i}a_2)=
(b_1^2-b_2^2-\tfrac{1}{2})\frac{\partial a_1}{\partial b_1}+2b_1b_2\frac{\partial a_1}{\partial b_2}
+\mathbf{i}\Big((b_1^2-b_2^2-\tfrac{1}{2})\frac{\partial a_2}{\partial b_1}+2b_1b_2\frac{\partial a_2}{\partial b_2}\Big),
$$
and
$$
-z(a_1+\mathbf{i}a_2)=(-a_1b_1-a_2b_2)+\mathbf{i}(-b_1a_2+b_2a_1).
$$
Then, \eqref{eq:2.ab8} and \eqref{eq:2.ab9} are equivalent to
$$
(z^2-\tfrac{1}{2})\frac{\partial}{\partial z}(a_1+\mathbf{i}a_2)=-z(a_1+\mathbf{i}a_2).
$$
Thus, we get a special solution to the equation system \eqref{eq:2.ab8}--\eqref{eq:2.ab11}: $(a_1+\mathbf{i}a_2)^2=\frac{1}{1-2(b_1-\mathbf{i}b_2)^2}$.

Let $f$ be any function satisfying that $E_3f=-1$. We denote $d_1=E_1f, d_2=E_2f$, by using \eqref{eqn:2.25}, we obtain the following integrability conditions for $f$.
$$
\begin{aligned}
&0=[E_1,E_3]f+b_1E_1(f)-b_2E_2(f)=-E_3d_1+b_1d_1-b_2d_2,\\
&0=[E_2,E_3]f+b_2E_1(f)+b_1E_2(f)=-E_3d_2+b_2d_1+b_1d_2,\\
&0=[E_1,E_2]f+2b_2E_3(f)=E_1d_2-E_2d_1-2b_2,
\end{aligned}
$$
which implies that  $d_1=E_1f, d_2=E_2f$ is a solution to the
second subsystem \eqref{eq:2.ab3}--\eqref{eq:2.ab4}.
\end{proof}

We denote the coordinates obtained from the solution given in Lemma \ref{lemma:2.1a2} by $(u,v,t)$, i.e.
$\partial_u=U,\ \partial_v=V,\ \partial_t=T=E_3$.
We will further make a special choice for $f$ to simplify \eqref{eq:j}.

\begin{lemma}\label{lemma:2.1a3}
There exists a function $f$ with $E_3f=-1$ and a constant $\tilde{c}$ such that
replacing $t$ by $t-\tilde{c}$ one has $l=0$. Moreover, $f$ is given by
$$
f=-\tfrac{1}{\sqrt{2}}{\rm arcsinh}\Big(\tfrac{-2\sqrt{2}b_1}
{\sqrt{(2b_1^2-2b_2^2-1)^2+16b_1^2b_2^2}}\Big).
$$
The variable $t$ is given by $t=-f+$const.

\end{lemma}
\begin{proof}
A short calculation implies that our choice of $f$ satisfies $E_3f=-1$.
By \eqref{UVT-H2}, we get
\begin{equation}\label{eq:2.ab12}
E_1=\frac{a_1}{a_1^2+a_2^2}\partial_u-\frac{a_2}{a_1^2+a_2^2}\partial_v-d_1\partial_t,
\end{equation}
\begin{equation}\label{eq:2.ab13}
E_2=\frac{a_2}{a_1^2+a_2^2}\partial_u+\frac{a_1}{a_1^2+a_2^2}\partial_v-d_2\partial_t,
\end{equation}
\begin{equation}\label{eq:2.ab14}
E_3=\partial_t.
\end{equation}
From $E_1f=d_1, E_2f=d_2, E_3f=-1$, we have
$$
\frac{\partial}{\partial u}f=\frac{\partial}{\partial v}f=0.
$$
We insert \eqref{eq:j} into the definition of $f$ and obtain $\sinh(-\sqrt{2}f)=\sinh(\sqrt{2}(t-l))$,
hence $\frac{\partial}{\partial u}l=\frac{\partial}{\partial v}l=0$.
Recalling that $E_3l=0$ (see Lemma \ref{lemma:2.1aaa1}), we obtain that $l$ must be a constant.
\end{proof}

\begin{theorem}\label{thm:3.1}
Let $M$ be a minimal hypersurface of $\mathbb{H}^2\times \mathbb{H}^2$ with product angle function $C=0$.
Assume that with respect to the frame field $\{E_1,E_2,E_3\}$ (see \eqref{eqn:2.5}), the function $b_2$ defined in \eqref{eqn:2.7} is not constant on $M$.
Then $M$ admits local coordinates $(u,v,t)$ and there is a non-constant function $h:(u,v)\mapsto \mathbb{R}$
such that the functions $b_1,b_2,b_4$ defined in \eqref{eqn:2.7} are given by
\eqref{eq:2.ab15}, where $h$ satisfies the following ``sine-Gordon equation'' 
\begin{equation}\label{eq:2.ab16}
(\frac{\partial^2}{\partial u^2}+\frac{\partial^2}{\partial v^2})h
=\tfrac{1}{\sqrt{2}}\sin(\sqrt{2}h).
\end{equation}

Conversely, let $\mathcal{D}$ be an open subset in $\mathbb{R}^2$ and  $h:\mathcal{D}\to\mathbb{R}, (u,v)\mapsto h(u,v)\in\mathbb{R}$
be a non-constant function which satisfies the differential equation \eqref{eq:2.ab16}.
Let $\Omega_0=\{(u,v,t)\in \mathcal{D}\times \mathbb{R}|\sqrt{2}h(u,v)=(2p+1)\pi,t=0, p\in\mathbb{Z}\}$, $\Omega=\{(u,v,t)\in \mathcal{D}\times \mathbb{R}\}-\Omega_0\subseteq \mathbb{R}^3$, we define a metric $g$ on $\Omega$ by \eqref{eq:2.ab17} and a $(1,1)$-tensor field $A$ on $T\Omega$ by
\eqref{eqn:A}. Let $\nu$ be a vector bundle over  $\Omega$ of rank $1$ with metric $\tilde{g}$, $\nabla^{\perp}$ a connection on $\nu$ compatible with the metric $\tilde{g}$, and $N$ a unit vector field in $\nu$.   We define a $(1,1)$-tensor field $\tilde{P}$ on $T\Omega\oplus \nu$ by
\eqref{defP}.
Then the integrability conditions are satisfied on $\Omega$ and hence on any simply
connected open subset $\Omega_1$ of $\Omega$, up to an isometry of  $\mathbb{H}^2\times \mathbb{H}^2$, there is a unique isometric immersion $\Psi$ from $(\Omega_1,~g)$ into $\mathbb{H}^2\times \mathbb{H}^2$ such that $A$ is the shape operator of  $\Omega_1$ into $\mathbb{H}^2\times \mathbb{H}^2$, $\nu$ is isometric to the normal bundle of $\Psi(\Omega_1)$ in
$\mathbb{H}^2\times \mathbb{H}^2$ by an isomorphism $\tilde{\Psi}:\nu\to T^{\perp}\Psi(\Omega_1)$,
and for any $Y\in T\Omega_1$, we have
\begin{equation}
P(\Psi_{*}Y)=\Psi_{*}((\tilde{P}Y)^{\top})+\tilde{\Psi}((\tilde{P}Y)^{\perp}),~
P(\tilde{\Psi}N)=\Psi_*(\partial_t),
\end{equation}
where $P$ is the product structure of $\mathbb{H}^2\times \mathbb{H}^2$, $(\tilde{P}Y)^{\top}$ and $(\tilde{P}Y)^{\perp}$ denote the projections of  $\tilde{P}Y$ onto $T\Omega_1$ and $\nu$, respectively.
Moreover, $\Psi(\Omega_1)$ is a minimal hypersurface with $C=0$.
\end{theorem}
\begin{proof}
Let $M$ be a minimal hypersurface of $\mathbb{H}^2\times \mathbb{H}^2$ with $C=0$.
By Lemmas \ref{lemma:2.1aaa1}--\ref{lemma:2.1a3}, $M$ admits local coordinates $(u,v,t)$ and
there is a function $h:(u,v)\mapsto \mathbb{R}$
such that $b_1,b_2,b_4$ are given by
\begin{equation}\label{eq:2.ab15}
b_1=-b_4=-\tfrac{1}{\sqrt{2}}\frac{\sinh(\frac{t}{\sqrt{2}})\cosh(\frac{t}{\sqrt{2}})}
{\sinh^2(\frac{t}{\sqrt{2}})+\cos^2(\frac{h}{\sqrt{2}})},\ \
b_2=-\tfrac{1}{\sqrt{2}}\frac{\sin(\frac{h}{\sqrt{2}})\cos(\frac{h}{\sqrt{2}})}
{\sinh^2(\frac{t}{\sqrt{2}})+\cos^2(\frac{h}{\sqrt{2}})},
\end{equation}
Now, by Lemma \ref{lemma:2.1a2}, we can obtain that
\begin{equation}\label{eq:2.ab18}
a_1=\cosh(\tfrac{t}{\sqrt{2}})\cos(\tfrac{h}{\sqrt{2}}),\
a_2=-\sinh(\tfrac{t}{\sqrt{2}})\sin(\tfrac{h}{\sqrt{2}}).
\end{equation}

Using \eqref{eq:2.ab12}--\eqref{eq:2.ab14} and \eqref{eq:2.ab15},
then \eqref{eqn:2.29} and \eqref{eqn:2.30} are equivalent to
$$
d_1=\tfrac{2\big(\cosh(\frac{t}{\sqrt{2}})\cos(\frac{h}{\sqrt{2}})h_v
-\sinh(\frac{t}{\sqrt{2}})\sin(\frac{h}{\sqrt{2}})h_u\big)}{\cosh(\sqrt{2}t)+\cos(\sqrt{2}h)},\ \
d_2=-\tfrac{2\big(\sinh(\frac{t}{\sqrt{2}})\sin(\frac{h}{\sqrt{2}})h_v
+\cosh(\frac{t}{\sqrt{2}})\cos(\frac{h}{\sqrt{2}})h_u\big)}{\cosh(\sqrt{2}t)+\cos(\sqrt{2}h)}.
$$

Now, we can check the compatibility of functions $b_1,b_2,b_4$ with respect to the Lie bracket $[E_i,E_j]$,
$1\leq i<j\leq3$, and obtain that
$$
(\frac{\partial^2}{\partial u^2}+\frac{\partial^2}{\partial v^2})h
=\tfrac{1}{\sqrt{2}}\sin(\sqrt{2}h).
$$
In fact, the  compatibility conditions of  $b_1,b_2,b_4$  with respect to the Lie brackets $[E_1,E_3]$ and $[E_2,E_3]$ are automatically satisfied and the compatibility conditions of  $b_1,b_2,b_4$  with respect to the Lie bracket $[E_1,E_2]$  imply the above differential equation.

Conversely, let $\mathcal{D}$ be an open subset in $\mathbb{R}^2$ and  $h:\mathcal{D}\to\mathbb{R}, (u,v)\mapsto h(u,v)\in\mathbb{R}$
be a non-constant function which satisfies the differential equation \eqref{eq:2.ab16}.
We assume that the metric $g$ on $\Omega$ is defined by
\begin{equation}\label{eq:2.ab17}
\begin{aligned}
g&=\Big(\tfrac{1}{2}(\cosh(\sqrt{2}t)+\cos(\sqrt{2}h))+h_v^2\Big)du^2
+\Big(\tfrac{1}{2}(\cosh(\sqrt{2}t)+\cos(\sqrt{2}h))+h_u^2\Big)dv^2+dt^2\\
&\ \ -2h_uh_vdudv+2h_vdudt-2h_udvdt,
\end{aligned}
\end{equation}
the $(1,1)$-tensor field $A$ on $T\Omega$ is defined by
\begin{equation}\label{eqn:A}
\begin{aligned}
A\partial_u&=-\tfrac{\cos(\sqrt{2}h)\sinh(\sqrt{2}t)}{\sqrt{2}(\cosh(\sqrt{2}t)+\cos(\sqrt{2}h))}
\partial_u
-\tfrac{\cosh(\sqrt{2}t)\sin(\sqrt{2}h)}{\sqrt{2}(\cosh(\sqrt{2}t)+\cos(\sqrt{2}h))}\partial_v\\
&\ \ +\tfrac{\cos(\sqrt{2}h)\sinh(\sqrt{2}t)h_v-\cosh(\sqrt{2}t)\sin(\sqrt{2}h)h_u}
{\sqrt{2}(\cosh(\sqrt{2}t)+\cos(\sqrt{2}h))}\partial_t,\\
A\partial_v&=-\tfrac{\cosh(\sqrt{2}t)\sin(\sqrt{2}h)}{\sqrt{2}(\cosh(\sqrt{2}t)+\cos(\sqrt{2}h))}
\partial_u
+\tfrac{\cos(\sqrt{2}h)\sinh(\sqrt{2}t)}{\sqrt{2}(\cosh(\sqrt{2}t)+\cos(\sqrt{2}h))}\partial_v\\
&\ \ +\tfrac{\sin(\sqrt{2}h)\cosh(\sqrt{2}t)h_v+\sinh(\sqrt{2}t)\cos(\sqrt{2}h)h_u}
{\sqrt{2}(\cosh(\sqrt{2}t)+\cos(\sqrt{2}h))}\partial_t,\\
A\partial_t&=0,
\end{aligned}
\end{equation}
and the $(1,1)$-tensor field $\tilde{P}$ on $T\Omega\oplus \nu$ is defined by 
\begin{equation}\label{defP}
\begin{aligned}
	\tilde{P}\partial_u&=\tfrac{1+\cos(\sqrt{2}h)\cosh(\sqrt{2}t)}{\cos(\sqrt{2}h)+\cosh(\sqrt{2}t)}
	\partial_u
	+\tfrac{\sin(\sqrt{2}h)\sinh(\sqrt{2}t)}{\cos(\sqrt{2}h)+\cosh(\sqrt{2}t)}\partial_v\\
	&\ \ -
\tfrac{h_v+h_v\cos(\sqrt{2}h)\cosh(\sqrt{2}t)-h_u\sin(\sqrt{2}h)\sinh(\sqrt{2}t)}
{{\cos(\sqrt{2}h)+\cosh(\sqrt{2}t)}}\partial_t+h_vN,\\
	\tilde{P}\partial_v&=\tfrac{\sin(\sqrt{2}h)\sinh(\sqrt{2}t)}{\cos(\sqrt{2}h)+\cosh(\sqrt{2}t)}
	\partial_u-
	\tfrac{1+\cos(\sqrt{2}h)\cosh(\sqrt{2}t)}{\cos(\sqrt{2}h)+\cosh(\sqrt{2}t)}\partial_v\\
	&\ \  -\tfrac{h_u+h_u\cos(\sqrt{2}h)\cosh(\sqrt{2}t)+h_v\sin(\sqrt{2}h)\sinh(\sqrt{2}t)}
{{\cos(\sqrt{2}h)+\cosh(\sqrt{2}t)}}\partial_t-h_uN,\\
	\tilde{P}\partial_t&=N,\\
	\tilde{P}N&=\partial_t.\\	
\end{aligned}
\end{equation}
It can be checked  directly that all the integrability conditions are satisfied.
Therefore, the conclusions can be obtained by applying  the existence and uniqueness theorems in  \cite{K,L-T-V} directly.
\end{proof}

Before we deal with the {\bf Case 1} of Lemma \ref{lemma:2.1aaa1},
we shall prove the following proposition.

\begin{proposition}\label{prop:3.61}
Let $M$ be a hypersurface of $\mathbb{H}^2\times\mathbb{H}^2$ with constant $C\neq\pm1$.
Then, under the frame $\{E_1,E_2,E_3\}$ defined by \eqref{eqn:2.5}, $b_2$ is constant on $M$ if and only if either
\begin{enumerate}
\item[(1)]
$b_2\neq 0$ and $M$ is an open part of $M_\tau$ for some $\tau<-1$; or

\item[(2)]
$b_2=0$ and $M$ is an open part of $M_{k_1,k_2}^c$ given in Example \ref{exam:3.4fg1}. 

\end{enumerate}
\end{proposition}

\begin{proof}
To verify the ``only if'' part, we assume that $b_2$ is constant on $M$.

If $b_2$ is a non-zero constant on $M$, then by \eqref{eqn:2.27}, it holds
$$
b_1\sqrt{\tfrac{1-C}{1+C}}-b_4\sqrt{\tfrac{1+C}{1-C}}=0,
$$
which implies that $b_4=\frac{1-C}{1+C}b_1$.
Substituting $b_4=\frac{1-C}{1+C}b_1$ into \eqref{eqn:2.26} and \eqref{eqn:2.28}, we have
$$
\begin{aligned}
E_3b_1&=\tfrac{1+C}{2}\sqrt{\tfrac{1+C}{1-C}}+b_2^2\sqrt{\tfrac{1+C}{1-C}}
-b_1^2\sqrt{\tfrac{1-C}{1+C}},\\
E_3b_1&=-\tfrac{\sqrt{1-C^2}}{2}+b_1^2\sqrt{\tfrac{1-C}{1+C}}-b_2^2\sqrt{\tfrac{1+C}{1-C}}.
\end{aligned}
$$
From the above two equations, it follows that $b_1$ and $b_4$ are also constant on $M$.
Now, we know that $M$ has constant principal curvatures and constant $C$.
According to Theorem 1.1 of \cite{GMY}, it can be directly checked that
$M$ is an open part of $M_\tau$ for some $\tau<-1$.

If $b_2=0$ on $M$, then by Theorem 3.11 of \cite{GMY}, we know that $M$ is an open part of $M_{k_1,k_2}^c$ given in Example \ref{exam:3.4fg1}.

For the ``if" part, it is a direct consequence of the verification that the hypersurfaces $M_\tau$ and
$M_{k_1,k_2}^c$ have constant $C$ and constant $b_2$.
\end{proof}

Now, combining Lemma \ref{lemma:2.1aaa1}, Theorem \ref{thm:3.1} and Proposition \ref{prop:3.61},
we can obtain the following classification result.
\begin{theorem}\label{thm:3.2aa}
Let $M$ be a minimal hypersurface of $\mathbb{H}^2\times \mathbb{H}^2$ with constant $C=0$.
Then, either
\begin{enumerate}
\item[(1)]
$M$ is an open part of $M_{k_1,k_1}^{1/2}$ given in Example \ref{exam:3.4fg1} for some constant $k_1$; or

\item[(2)]
$M$ is an open part of a minimal hypersurface described by Theorem \ref{thm:3.1}.
\end{enumerate}
\end{theorem}
\begin{proof}
By Lemma \ref{lemma:2.1aaa1}, under the frame $\{E_1,E_2,E_3\}$ defined by \eqref{eqn:2.5}, there are two cases
for $b_1,b_2,b_4$ on $M$, and the {\bf Case 2} has been classified by Theorem \ref{thm:3.1}.

For {\bf Case 1}, by Proposition \ref{prop:3.61}, $M$ is an open part of $M_{k_1,k_2}^c$ given in Example \ref{exam:3.4fg1}.
Then, by Lemma \ref{lemma:3.5fg12}, we know that $c=\frac{1}{2}$, $k_1(r)$ and $k_2(s)$
are equal to the same constant. 
\end{proof}

\section{Proof of Theorem \ref{thm:1.1}}\label{sect:5}
In this section, we prove Theorem \ref{thm:1.1}.
Firstly, by using the so-called Tsinghua principle, we establish a very useful identity
on any hypersurface $M$ of $\mathbb{H}^2\times\mathbb{H}^2$ with constant sectional curvature.

\begin{lemma}\label{lemma:3.2aa}
Let $M$ be a hypersurface of $\mathbb{H}^2\times \mathbb{H}^2$ with constant sectional curvature $\kappa$.
Then, for any tangent vector fields $W,U,Y,Z\in TM$, the following equation holds:
\begin{equation}\label{eqn:3.1}
\mathop\mathfrak{S}\limits_{W,U,Y}\mathbf{I}(W,U,Y,Z)=0,
\end{equation}
where the symbol $\mathfrak{S}$ stands for the cyclic summation,
and $\mathbf{I}(W,U,Y,Z)$ is defined by
\begin{equation*}
\mathbf{I}(W,U,Y,Z):=-\tfrac{1}{2}\{-g(TY,Z)g(AW,TU)+g(TU,Z)g(AW,TY)\}.
\end{equation*}
\end{lemma}
\begin{proof}

In order to prove Lemma \ref{lemma:3.2aa}, we calculate the expression of the cyclic
summation
\begin{equation}\label{eqn:3.2}
\mathfrak{A}:=\mathop\mathfrak{S}\limits_{W,U,Y} \big\{g((\nabla^2 A)(W,U,Y),Z)-
g((\nabla^2 A)(W,Y,U),Z)\big\}
\end{equation}
in two different ways. On the one hand, we take the covariant derivative
of the Codazzi equation \eqref{eqn:2.3}, we have
\begin{equation*}
\begin{split}
&g((\nabla^2 A)(W,U,Y),Z)-g((\nabla^2 A)(W,Y,U),Z)\\
&=-\tfrac{1}{2}\{g((\nabla_W T)Y,Z)\mu(U)+g(TY,Z)(\nabla_W \mu)U\\
&\ \ -g((\nabla_W T)U,Z)\mu(Y)-g(TU,Z)(\nabla_W \mu)Y\}.
\end{split}
\end{equation*}
By direct calculations, with the use of Lemma \ref{lemma:2.1},
we can have
\begin{equation*}
\begin{split}
g&((\nabla^2 A)(W,U,Y),Z)-g((\nabla^2 A)(W,Y,U),Z)\\
&=-\tfrac{1}{2}\{g(AW,Y)\mu(Z)\mu(U)-g(TY,Z)g(AW,TU)+Cg(TY,Z)g(AW,U)\\
&\ \ -g(AW,U)\mu(Z)\mu(Y)+g(TU,Z)g(AW,TY)-Cg(TU,Z)g(AW,Y)\}.
\end{split}
\end{equation*}
Thus, it holds that
\begin{equation}\label{eqn:3.3aa}
\mathfrak{A}=\mathop\mathfrak{S}\limits_{W,U,Y}\mathbf{I}(W,U,Y,Z).
\end{equation}

On the other hand, the cyclic sum $\mathfrak{A}$ can be rewritten as
\begin{equation}\label{eqn:3.3}
\mathfrak{A}=\mathop\mathfrak{S}\limits_{W,U,Y}\big\{g((\nabla^{2}A)(W,U,Y),Z)
-g((\nabla^{2}A)(U,W,Y),Z)\big\}.
\end{equation}
Then by applying the Ricci identity \eqref{eqn:ric}, it follows from \eqref{eqn:3.3} that
\begin{equation}\label{eqn:3.4}
\mathfrak{A}=-\mathop\mathfrak{S}\limits_{W,U,Y} \big\{g(R(W,U)Y,AZ)
+g(R(W,U)Z,AY)\big\}.
\end{equation}

Since $M$ has constant sectional curvature $\kappa$, we have
\begin{equation}\label{eqn:3.6aa}
g(R(W,U)Y,Z)=\kappa\{g(U,Y)g(W,Z)-g(W,Y)g(U,Z)\},
\end{equation}
it follows directly that
$$
\mathfrak{A}=-\mathop\mathfrak{S}\limits_{W,U,Y} \big\{g(R(W,U)Y,AZ)
+g(R(W,U)Z,AY)\big\}=0.
$$
Thus, from \eqref{eqn:3.3aa}, we have completed the proof of Lemma \ref{lemma:3.2aa}.
\end{proof}

Next, we prove that the sectional curvature of $M$ must be
$\kappa=-\tfrac{1}{2}$.
\begin{proposition}\label{prop:4.5}
Let $M$ be a hypersurface of $\mathbb{H}^2\times\mathbb{H}^2$ with constant sectional curvature $\kappa$.
Then, $\kappa=-\tfrac{1}{2}$.
\end{proposition}
\begin{proof}
Suppose that $M$ has constant sectional curvature $\kappa$. If $C^2\equiv1$ on the connected hypersurface, then by Lemma \ref{lem:2.2},  we know that $M$ is an open part of 
$\Gamma\times \mathbb{H}^2$, where $\Gamma$ 
is a curve of $\mathbb{H}^2$. It follows that the tangent bundle of $M$ has
a decomposition $TM=T\Gamma\oplus T\mathbb{H}^2$,
and we have $A|_{T\mathbb{H}^2}=0$.
Now, we take an orthonormal basis $\{v_1,v_2,v_3\}$ with $v_1\in T\Gamma$ and $v_2,v_3\in T\mathbb{H}^2$. 
By direct calculations, we have $g(R(v_1,v_2)v_2,v_1)=0$ and $g(R(v_2,v_3)v_3,v_2)=-1$. This contradicts the assumption that 
$M$ has constant sectional curvature. Hence the open set $\tilde \Omega=\{|C|<1\}$ is nonempty. 

Now, we calculate on $\tilde \Omega$. Since $\kappa$ is constant, proving its value on $\tilde \Omega$ proves it on $M$. An isolated or lower-dimensional set with $|C|=1$ is not excluded by this argument, and is treated via smooth extension.

On $\tilde \Omega$, we can take the local orthonormal frame field $\{E_1,E_2,E_3\}$
defined by \eqref{eqn:2.5}.
Then, the shape operator $A$, $(1,1)$-tensor $T$ and $1$-form $\mu$
are given by \eqref{eqn:2.6}--\eqref{eqn:2.7}.
By using $\bar{\nabla}P=0$, \eqref{eqn:2.6} and Lemma \ref{lemma:2.1},
we can obtain the connections with respect to $\{E_1,E_2,E_3\}$ as follows:
\begin{equation}\label{eqn:LL}
\left\{
\begin{aligned}
&E_1C=-2b_3\sqrt{1-C^2},\ E_2C=-2b_5\sqrt{1-C^2},\ E_3C=-2b_6\sqrt{1-C^2},\\ 		
&\nabla_{E_1}E_1=b_1\sqrt{\tfrac{1-C}{1+C}}E_3,\
\nabla_{E_1}E_2=-b_2\sqrt{\tfrac{1+C}{1-C}}E_3,\ \nabla_{E_1}E_3=-b_1\sqrt{\tfrac{1-C}{1+C}}E_1
+b_2\sqrt{\tfrac{1+C}{1-C}}E_2,\\
&\nabla_{E_2} E_1=b_2\sqrt{\tfrac{1-C}{1+C}}E_3,\
\nabla_{E_2} E_2=-b_4\sqrt{\tfrac{1+C}{1-C}}E_3,\ \nabla_{E_2}E_3=-b_2\sqrt{\tfrac{1-C}{1+C}}E_1
+b_4\sqrt{\tfrac{1+C}{1-C}}E_2,\\
&\nabla_{E_3}E_1=b_3\sqrt{\tfrac{1-C}{1+C}}E_3,\
\nabla_{E_3}E_2=-b_5\sqrt{\tfrac{1+C}{1-C}}E_3,\ \nabla_{E_3}E_3=-b_3\sqrt{\tfrac{1-C}{1+C}}E_1
+b_5\sqrt{\tfrac{1+C}{1-C}}E_2.
\end{aligned}\right.
\end{equation}

Taking in \eqref{eqn:3.1}, respectively, 
$$
(W,U,Y,Z)=(E_1,E_2,E_3,E_1), \ (E_1,E_2,E_3,E_2), \ (E_1,E_2,E_3,E_3),
$$
we obtain $b_3(1+C)=b_5(1-C)=b_2C=0$.
Since $C\neq\pm1$, we have $b_3=b_5=0$. 
In the following, we divide the argument into two subcases depending on whether $C$ is $0$ or not.

{\bf Case I}: $C\neq0$.

In this subcase, we have $b_3=b_5=b_2=0$. Then, by direct calculations, we get
$$
g(R(E_1,E_2)E_2,E_1)=b_1b_4,\
g(R(E_1,E_3)E_3,E_1)=\tfrac{C}{2}-\tfrac{1}{2}+b_1b_6,\
g(R(E_2,E_3)E_3,E_2)=b_4b_6-\tfrac{1}{2}-\tfrac{C}{2}.
$$
By $C\neq\pm1$, it follows that $\kappa\neq0$, $b_6=\tfrac{2\kappa+1-C}{2b_1}=\tfrac{2\kappa+1+C}{2b_4}$ 
and $b_6^2=\tfrac{(2\kappa+1)^2-C^2}{4\kappa}$. If $\kappa>0$, without loss of generality, we have
\begin{equation}\label{eqn:6.2}
b_6=-\tfrac{\sqrt{(2\kappa+1)^2-C^2}}{2\sqrt{\kappa}},\
b_1=-\tfrac{\sqrt{\kappa}(2\kappa+1-C)}{\sqrt{(2\kappa+1)^2-C^2}},
\ b_4=-\tfrac{\sqrt{\kappa}(2\kappa+1+C)}{\sqrt{(2\kappa+1)^2-C^2}}.
\end{equation}

Next, taking $(Y,Z)=(E_3,E_1)$ and $(Y,Z)=(E_3,E_2)$ into Codazzi equation \eqref{eqn:2.3},
with the use of $b_2=b_3=b_5=0$, \eqref{eqn:2.6} and \eqref{eqn:LL},
we compare the components of $E_1$ and $E_2$, respectively, then we can have
\begin{equation}\label{eqn:6.3}
E_3b_1=-\tfrac{\sqrt{1-C}(1+C-2b_1(b_1-b_6))}{2\sqrt{1+C}},\
E_3b_4=\tfrac{\sqrt{1+C}(1-C-2b_4(b_4-b_6))}{2\sqrt{1-C}}.
\end{equation}
In the following, we take the derivatives of $b_1$ and $b_4$ with respect to $E_3$,
with the use of \eqref{eqn:6.2} and $E_3C=-2b_6\sqrt{1-C^2}$, we obtain
\begin{equation}\label{eqn:6.4}
E_3b_1=\tfrac{\sqrt{1-C^2}(2\kappa+1)}{1+C+2\kappa},\
E_3b_4=\tfrac{\sqrt{1-C^2}(2\kappa+1)}{-1+C-2\kappa}.
\end{equation}
By \eqref{eqn:6.2}, \eqref{eqn:6.3} and \eqref{eqn:6.4}, we have $\sqrt{1-C^2}(2\kappa+1)=0$,
which implies that $\kappa=-\tfrac{1}{2}$.
But, this contradicts $\kappa>0$. Thus, it holds $\kappa<0$.

If $\kappa<0$, without loss of generality, we have
\begin{equation}\label{eqn:6.5}
b_6=-\tfrac{\sqrt{-(2\kappa+1)^2+C^2}}{2\sqrt{-\kappa}},\
b_1=-\tfrac{\sqrt{-\kappa}(2\kappa+1-C)}{\sqrt{-(2\kappa+1)^2+C^2}},
\ b_4=-\tfrac{\sqrt{-\kappa}(2\kappa+1+C)}{\sqrt{-(2\kappa+1)^2+C^2}}.
\end{equation}

By taking $(Y,Z)=(E_3,E_1)$ and $(Y,Z)=(E_3,E_2)$ into Codazzi equation \eqref{eqn:2.3},
we compare the components of $E_1$ and $E_2$ respectively, then we still have \eqref{eqn:6.3}.
Taking the derivatives of $b_1$ and $b_4$ in \eqref{eqn:6.5} with respect to $E_3$,
with the use of \eqref{eqn:6.5} and $E_3C=-2b_6\sqrt{1-C^2}$, then comparing with \eqref{eqn:6.3}, we can obtain
$$
\kappa=-\frac{1}{2},\ b_1=\frac{C}{\sqrt{2C^2}},\ b_4=-\frac{C}{\sqrt{2C^2}},\ b_6=-\frac{\sqrt{C^2}}{\sqrt{2}}.
$$

Notice that, the isometry map
$$
\mathbb{F}:\mathbb{H}^2\times \mathbb{H}^2\rightarrow\mathbb{H}^2\times \mathbb{H}^2,\ \
\mathbb{F}(p,q)=(q,p),
$$
satisfies $d \mathbb{F}\circ P=-P \circ d \mathbb{F}$. If $C<0$ on an open subset $\Omega$ of $M$, then using the isometry $\mathbb{F}$, we can obtain another hypersurface $\mathbb{F}(M)$ with constant sectional curvature and $C>0$ on $\mathbb{F}(\Omega)$. 
Thus, we can assume that $C>0$ on an open subset $\Omega$ of $M$, and have
$$
b_1=\frac{1}{\sqrt{2}},\ b_4=-\frac{1}{\sqrt{2}},\ b_6=-\frac{C}{\sqrt{2}},\ b_2=b_3=b_5=0.
$$
By directly integrating $E_3C=C\sqrt{2(1-C^2)}$ along the integral curves of $E_3$ and $C>0$ on $\Omega$, we have
$$
C=\frac{1}{\cosh(\sqrt{2}\tilde{t}+\tilde{f})},
$$
where $\tilde{t},\tilde{f}$ are functions on $\Omega$ satisfying $E_3\tilde{t}=1$, $E_3\tilde{f}=0$ and $\sqrt{2}\tilde{t}+\tilde{f}<0$. 
By \eqref{eqn:2.4} and $AE_3=-\frac{C}{\sqrt{2}}E_3$, we get  
$|\nabla C|^2=2C^2(1-C^2)$. Integration of $|\nabla C|\le\sqrt2|C|$ along piecewise smooth paths proves that either 
$C\equiv0$ or $C$ is nowhere zero with constant sign.

{\bf Case II}: $C=0$.

In this subcase, by Lemma \ref{lemma:2.1}, we have $b_3=b_5=b_6=0$.
By Gauss equation \eqref{eqn:2.2}, we have
$$
g(R(E_1,E_2)E_2,E_1)=b_1b_4-b_2^2, \ \ g(R(E_1,E_3)E_3,E_1)=g(R(E_2,E_3)E_3,E_2)=-\tfrac{1}{2}.
$$
It implies that the sectional curvature of $M$ is $\kappa=-\tfrac{1}{2}$
and it holds $b_1b_4-b_2^2=-\tfrac{1}{2}$. 
\end{proof}

In the following, we discuss the {\bf Case I} of Proposition \ref{prop:4.5}.

\begin{theorem}\label{thm:4.1abc}
Let $M$ be a hypersurface of $\mathbb{H}^2\times \mathbb{H}^2$ with constant sectional curvature $-\frac{1}{2}$.
If under the frame field $\{E_1,E_2,E_3\}$ defined by \eqref{eqn:2.5}, it holds
$$
C>0,\ b_1=\frac{1}{\sqrt{2}},\ b_4=-\frac{1}{\sqrt{2}},\ b_6=-\frac{C}{\sqrt{2}},\ b_2=b_3=b_5=0.
$$
Then, $M$ is an open part of Example \ref{exam:4.1ss1ss11ab1}.
\end{theorem}
\begin{proof}
By $b_3=b_5=0$, $b_6=-\frac{C}{\sqrt{2}}$ and $\nabla C=-2AX$, we have that the product angle function satisfies
$$
E_1C=E_2C=0,\ E_3C=C\sqrt{2(1-C^2)}.
$$
Now, we choose two appropriate non-zero functions $\rho_1$ and $\rho_2$, such that
the new frame
$$
\{X_1=\rho_1E_1,\ X_2=\rho_2E_2,\ X_3=E_3\}
$$
satisfies that $[X_1,X_2]=[X_1,X_3]=[X_2,X_3]=0$. Specifically, by using
$[E_i,E_j]=\nabla_{E_i} {E_j}-\nabla_{E_j} {E_i}$, \eqref{eqn:LL} and $b_2=b_3=b_5=0$, we get
$$
[E_1,E_2]=0,\ [E_1,E_3]=-\sqrt{\tfrac{1-C}{2(1+C)}}E_1,\ [E_2,E_3]=-\sqrt{\tfrac{1+C}{2(1-C)}}E_2.
$$
Therefore
$$
\begin{aligned}
[X_1,X_3]&=-(\sqrt{\tfrac{1-C}{2(1+C)}}\rho_1+E_3\rho_1)E_1.
\end{aligned}
$$
Then $[X_1,X_3]=0$ is equivalent to
\begin{equation}\label{eqn:m1}
E_3\rho_1=-\sqrt{\tfrac{1-C}{2(1+C)}}\rho_1.
\end{equation}
Similarly, we get
$$
[X_2,X_3]=(-\sqrt{\tfrac{1+C}{2(1-C)}}\rho_2-E_3\rho_2)E_2,\ \ 
[X_1,X_2]=\rho_1(E_1\rho_2)E_2-\rho_2(E_2\rho_1)E_1.
$$
Then the equations $[X_2,X_3]=[X_1,X_2]=0$ are equivalent to
\begin{equation}\label{eqn:m2}
E_3\rho_2=-\sqrt{\tfrac{1+C}{2(1-C)}}\rho_2,
\end{equation}
\begin{equation}\label{eqn:dm1dm4}
\rho_1 E_1\rho_2=\rho_2E_2\rho_1=0.
\end{equation}

Let $\rho_1=\sqrt{\frac{1+C}{2C}}$, $\rho_2=\sqrt{\frac{1-C}{2C}}$, then by $E_1C=E_2C=0$ and $E_3C=C\sqrt{2(1-C^2)}$,
we can check that \eqref{eqn:m1}--\eqref{eqn:dm1dm4} hold, which implies that $[X_1,X_2]=[X_1,X_3]=[X_2,X_3]=0$.

By the definition of the frame $\{X_1,X_2,X_3\}$ and using $[X_1,X_2]=[X_1,X_3]=[X_2,X_3]=0$,
we can identify $M$ with an open subset $\Omega_1$ of $\mathbb{R}^3$ locally and express the hypersurface $M$ by an immersion
$$
\begin{aligned}
\Phi_1:\Omega_1
\longrightarrow\mathbb{H}^2\times\mathbb{H}^2,\ \ \
(t_1,r,s)\mapsto (p(t_1,r,s),q(t_1,r,s)),
\end{aligned}
$$
such that $(\frac{\partial p}{\partial t_1},\frac{\partial q}{\partial t_1})=E_3$,
$(\frac{\partial p}{\partial r},\frac{\partial q}{\partial r})=\rho_1E_1$ and
$(\frac{\partial p}{\partial s},\frac{\partial q}{\partial s})=\rho_2E_2$.

By $C>0$, we integrate $E_3C=C\sqrt{2(1-C^2)}$ along the integral curves of $E_3$, and obtain
$$
C=\frac{1}{\cosh(\sqrt{2}t_1+\tilde{h})},
$$
where $\tilde{h}$ is a function on $M$ satisfying $E_3\tilde{h}=0$. This is a local formula on a branch where $0<C<1$: the equation
$E_3C=\sqrt2 C\sqrt{1-C^2}>0$ forces 
$\sqrt2 t_1+\tilde h<0$.  The value $C=1$ is not part of
this canonical-frame chart and is treated by
Remark~\ref{rem:axisregular-verified}.

By $E_1C=E_2C=0$, $(\frac{\partial p}{\partial r},\frac{\partial q}{\partial r})=\rho_1E_1$ and
$(\frac{\partial p}{\partial s},\frac{\partial q}{\partial s})=\rho_2E_2$, we have that
$$
\frac{\partial C}{\partial r}=\frac{\partial C}{\partial s}=0\ \  {\rm and}\ \
\frac{\partial \tilde{h}}{\partial r}=\frac{\partial \tilde{h}}{\partial s}=0.
$$
It follows from $E_3\tilde{h}=0$ that $\tilde{h}$ is a constant function. Thus, without loss of generality,
we can identify $M$ with an open subset $\Omega$ of $\mathbb{R}^3$ locally and express the hypersurface $M$ by an immersion
$$
\begin{aligned}
\Phi:\Omega
\longrightarrow\mathbb{H}^2\times\mathbb{H}^2,\ \ \
(t,r,s)\mapsto (p(t,r,s),q(t,r,s)),
\end{aligned}
$$
such that $(\frac{\partial p}{\partial t},\frac{\partial q}{\partial t})=E_3$,
$(\frac{\partial p}{\partial r},\frac{\partial q}{\partial r})=\rho_1E_1$,
$(\frac{\partial p}{\partial s},\frac{\partial q}{\partial s})=\rho_2E_2$ and
$C=\frac{1}{\cosh(\sqrt{2}t)}$, where $t=t_1+\frac{\tilde{h}}{\sqrt{2}}<0$. By $E_3=(\frac{\partial p}{\partial t},\frac{\partial q}{\partial t})$ and
$PN=CN+\sqrt{1-C^2}E_3$, we have $N=(\sqrt{\frac{1+C}{1-C}}\frac{\partial p}{\partial t},-\sqrt{\frac{1-C}{1+C}}\frac{\partial q}{\partial t})$.

From the definition of $P$  and using
$$
PE_1=E_1,\ \ PE_2=-E_2,
$$
we obtain that $dp,dq:T(\Omega)\rightarrow
T\mathbb{H}^2$ satisfy
\begin{equation}\label{eqn:5.49}
\left\{
\begin{aligned}
(dp(\tfrac{\partial}{\partial r}),0)&=\tfrac{1}{2}(d\Phi(\tfrac{\partial}{\partial r})+Pd\Phi(\tfrac{\partial}{\partial r}))=d\Phi(\tfrac{\partial}{\partial r}),\\
(0,dq(\tfrac{\partial}{\partial r}))&=\tfrac{1}{2}(d\Phi(\tfrac{\partial}{\partial r})-Pd\Phi(\tfrac{\partial}{\partial r}))=0,
\end{aligned}
\right.
\end{equation}
and
\begin{equation}\label{eqn:5.50}
\left\{
\begin{aligned}
(dp(\tfrac{\partial}{\partial s}),0)&=\tfrac{1}{2}(d\Phi(\tfrac{\partial}{\partial s})+Pd\Phi(\tfrac{\partial}{\partial s}))=0,\\
(0,dq(\tfrac{\partial}{\partial s}))&=\tfrac{1}{2}(d\Phi(\tfrac{\partial}{\partial s})-Pd\Phi(\tfrac{\partial}{\partial s}))=d\Phi(\tfrac{\partial}{\partial s}).
\end{aligned}
\right.
\end{equation}

The first equation of \eqref{eqn:5.50} shows that $p$ depends only on $(t,r)$.
Similarly, from the second equation in \eqref{eqn:5.49} we derive that $q$ depends only on
the $(t,s)$. By $g(PE_3,E_3)=-C$, $AE_3=-\frac{C}{\sqrt{2}}E_3$ and $\nabla_{E_3}{E_3}=0$,
we have that
\begin{equation}\label{eqn:4.mm1}
\begin{aligned}
(\tfrac{\partial^2 p}{\partial t^2},\tfrac{\partial^2 q}{\partial t^2})&
=\bar{\nabla}_{E_3}{E_3}-g( D_{E_3}{E_3},\tfrac{(p,q)}{\sqrt{2}})\tfrac{(p,q)}{\sqrt{2}}
-g( D_{E_3}{E_3},\tfrac{(p,-q)}{\sqrt{2}})\tfrac{(p,-q)}{\sqrt{2}}\\
&=-\tfrac{C}{\sqrt{2}}N+\tfrac{1}{2}(p,q)-\tfrac{C}{2}(p,-q)\\
&=(-C\sqrt{\tfrac{1+C}{2(1-C)}}\tfrac{\partial p}{\partial t},C\sqrt{\tfrac{1-C}{2(1+C)}}\tfrac{\partial q}{\partial t})+(\tfrac{1-C}{2}p,\tfrac{1+C}{2}q),
\end{aligned}
\end{equation}
where $g$ also denotes the pseudo-Euclidean inner product 
on $\mathbb{R}_2^6$, and $D$ is the canonical pseudo-Euclidean connection on $\mathbb{R}_2^6$.
From $\frac{dC}{dt}=E_3C=C\sqrt{2(1-C^2)}$, we know that the solution of
\eqref{eqn:4.mm1} is given by
\begin{equation}\label{eqn:pq}
\begin{aligned}
p(t,r)=\sqrt{\tfrac{1+C}{C}}V_1(r)+\tfrac{1}{\sqrt{C}}V_2(r),\ \
q(t,s)=\tfrac{1}{\sqrt{C}}W_1(s)+\sqrt{\tfrac{1-C}{C}}W_2(s),
\end{aligned}
\end{equation}
where $V_1(r),V_2(r),W_1(s),W_2(s)\in \mathbb{R}_1^3$ are vector-valued functions depending on variables $r$ and $s$. Thus, we have
$$
\begin{aligned}
E_1&=\tfrac{1}{\rho_1}(\tfrac{\partial p}{\partial r},\tfrac{\partial q}{\partial r})
=(\sqrt{2}\tfrac{d V_1(r)}{dr}+\sqrt{\tfrac{2}{1+C}}\tfrac{d V_2(r)}{dr},0),\\
E_2&=\tfrac{1}{\rho_2}(\tfrac{\partial p}{\partial s},\tfrac{\partial q}{\partial s})
=(0,\sqrt{\tfrac{2}{1-C}}\tfrac{d W_1(s)}{ds}+\sqrt{2}\tfrac{d W_2(s)}{ds}),\\
E_3&=(\tfrac{\partial p}{\partial t},\tfrac{\partial q}{\partial t})
=(-\sqrt{\tfrac{1-C}{2C}}V_1(r)-\sqrt{\tfrac{1-C^2}{2C}}V_2(r),
-\sqrt{\tfrac{1-C^2}{2C}}W_1(s)-\sqrt{\tfrac{1+C}{2C}}W_2(s)),
\end{aligned}
$$
and
$$
N=(\sqrt{\tfrac{1+C}{1-C}}\tfrac{\partial p}{\partial t},-\sqrt{\tfrac{1-C}{1+C}}\tfrac{\partial q}{\partial t})
=(-\sqrt{\tfrac{1+C}{2C}}V_1(r)-\tfrac{1+C}{\sqrt{2C}}V_2(r),
\tfrac{1-C}{\sqrt{2C}}W_1(s)+\sqrt{\tfrac{1-C}{2C}}W_2(s)).
$$

From the fact that $\langle p(t,r),p(t,r)\rangle=\langle q(t,s),q(t,s)\rangle=-1$, we have
$$
\begin{aligned}
&1+\langle V_1(r),V_1(r)\rangle+\tfrac{1}{C}(\langle V_1(r),V_1(r)\rangle+\langle V_2(r),V_2(r)\rangle)
+\tfrac{2\sqrt{1+C}}{C}\langle V_1(r),V_2(r)\rangle=0,\\
&1-\langle W_2(s),W_2(s)\rangle+\tfrac{1}{C}(\langle W_1(s),W_1(s)\rangle+\langle W_2(s),W_2(s)\rangle)
+\tfrac{2\sqrt{1-C}}{C}\langle W_1(s),W_2(s)\rangle=0.
\end{aligned}
$$
According to the independence of parameters $t,r$ and $s$, it follows from the above two equations that
$$
\begin{aligned}
&\langle V_1(r),V_1(r)\rangle=-1,\ \langle V_1(r),V_2(r)\rangle=0,\ \langle V_2(r),V_2(r)\rangle=1.\\
&\langle W_1(s),W_1(s)\rangle=-1,\ \langle W_1(s),W_2(s)\rangle=0,\ \langle W_2(s),W_2(s)\rangle=1.
\end{aligned}
$$
Then, by $g(E_1,E_3)=g(E_2,E_3)=0$,
with the use of $\langle \frac{dV_1(r)}{dr},V_2(r)\rangle+\langle V_1(r),\frac{dV_2(r)}{dr}\rangle=0$
and $\langle \frac{dW_1(s)}{ds},W_2(s)\rangle+\langle W_1(s),\frac{dW_2(s)}{ds}\rangle=0$,
we can have
$$
\begin{aligned}
&\langle \frac{d V_1(r)}{dr},V_2(r)\rangle=\langle \frac{d W_1(s)}{ds},W_2(s)\rangle=0.
\end{aligned}
$$

Suppose first that $V_1'(r_0)=0$ at some point
$r_0$. Since $E_1$ is a unit vector, the formula for $E_1$ implies
$V_2'(r_0)\ne0$. At points with $r=r_0$ we then have 
$$
\begin{aligned}
E_1&=\sqrt{\tfrac{2}{1+C}}(\tfrac{d V_2(r)}{d r},0),\\
D_{E_1}N&=\sqrt{\tfrac{2C}{1+C}}\tfrac{\partial N}{\partial r}=
\sqrt{\tfrac{2C}{1+C}}(-\tfrac{1+C}{\sqrt{2C}}\tfrac{d V_2(r)}{d r},0)=-\tfrac{1+C}{\sqrt{2}}E_1,
\end{aligned}
$$
which implies that $b_1=-\langle D_{E_1}N,E_1\rangle=\frac{1+C}{\sqrt{2}}$.
By the assumption that $b_1=\frac{1}{\sqrt{2}}$, we have $C=0$. It is a contradiction.
Thus, $V_1(r)$ is a regular curve of $\mathbb{H}^2$,
and $V_2(r)$ is a unit normal vector field of $V_1(r)\hookrightarrow\mathbb{H}^2$. 
Let $k(r)$ be its signed geodesic curvature with respect to  $V_2(r)$, we have
$\frac{dV_2(r)}{dr}=-k(r)\frac{dV_1(r)}{dr}$. 
Now, we can have
$$
\begin{aligned}
E_1&=\Big((\sqrt{2}-k(r)\sqrt{\tfrac{2}{1+C}})\tfrac{d V_1(r)}{dr},0\Big),\\
D_{E_1}N&=\sqrt{\tfrac{2C}{1+C}}\tfrac{\partial N}{\partial r}=
((-1+k(r)\sqrt{1+C})\tfrac{d V_1(r)}{d r},0)=\tfrac{-1+k(r)\sqrt{1+C}}{\sqrt{2}-k(r)\sqrt{\tfrac{2}{1+C}}}E_1,
\end{aligned}
$$
which implies that $b_1=-\langle D_{E_1}N,E_1\rangle=\frac{1-k(r)\sqrt{1+C}}{\sqrt{2}-k(r)\sqrt{\frac{2}{1+C}}}$.
By the assumption that $b_1=\frac{1}{\sqrt{2}}$ and $C\neq0$, we have $k(r)=0$,
i.e., $V_1(r)$ is a geodesic of $\mathbb{H}^2$.

Suppose next that $W_1'(s_0)\ne0$ at some point $s_0$.
On a smaller interval, $W_1$ is then a regular curve in $\mathbb H^2$
with unit normal $W_2$. Denote its signed geodesic curvature by $\tilde k(s)$. We have
$\frac{dW_2(s)}{ds}=-\tilde{k}(s)\frac{dW_1(s)}{ds}$. 
So, it follows that 
$$
\begin{aligned}
E_2&=\Big(0,(\sqrt{\tfrac{2}{1-C}}-\sqrt{2}\tilde{k}(s))\tfrac{d W_1(s)}{ds}\Big),\\
D_{E_2}N&=\sqrt{\tfrac{2C}{1-C}}\tfrac{\partial N}{\partial s}=
(0,(\sqrt{1-C}-\tilde{k}(s))\tfrac{d W_1(s)}{d s})=\tfrac{\sqrt{1-C}-\tilde{k}(s)}{\sqrt{\tfrac{2}{1-C}}-\sqrt{2}\tilde{k}(s)}E_2,
\end{aligned}
$$
which implies that $b_4=-\langle D_{E_2}N,E_2\rangle=\frac{-\sqrt{1-C}+\tilde{k}(s)}{\sqrt{\frac{2}{1-C}}-\sqrt{2}\tilde{k}(s)}$.
By the assumption that $b_4=-\frac{1}{\sqrt{2}}$, we have $\sqrt{1-C}-\sqrt{\frac{1}{1-C}}=0$. 
This contradicts $C>0$. Consequently $W'_1(s)=0$ at every point, so $W_1(s)$ is a constant vector, and we have
$\frac{d W_2(s)}{ds}\neq0$. 
Then, we get
$$
\begin{aligned}
E_2&=(0,\sqrt{2}\tfrac{d W_2(s)}{d s}),\\
D_{E_2}N&=\sqrt{\tfrac{2C}{1-C}}\tfrac{\partial N}{\partial s}=
(0,\tfrac{d W_2(s)}{d s})=\tfrac{1}{\sqrt{2}}E_2,
\end{aligned}
$$
which implies that $b_4=-\langle D_{E_2}N,E_2\rangle=-\frac{1}{\sqrt{2}}$.

In conclusion, we know that $V_1(r)$ is a geodesic of $\mathbb{H}^2$ and $V_2(r)$ is a unit normal
vector field of $V_1(r)\hookrightarrow\mathbb{H}^2$;
$W_1(s)$ is a constant vector in $\mathbb{H}^2$ and $W_2(s)\in T^1_{W_1(s)}\mathbb{H}^2$.
By taking reparametrizations of $r$ and $s$, and up to an isometry of $\mathbb{H}^2\times \mathbb{H}^2$, we can assume that 
$$
V_2(r)=(0,0,1),\quad W_1(s)=(1,0,0), 
$$
while $V_1(r)$ and $W_2(s)$ are open arcs of the curves $(\cosh(r),\sinh(r),0)$ and $(0,\cos(s),\sin(s))$, respectively. 
Thus, combining \eqref{eqn:pq} and $C=\frac{1}{\cosh(\sqrt{2}t)}$, we know that 
$M$ is an open part of Example \ref{exam:4.1ss1ss11ab1}. 
\end{proof}

In the following, we discuss the hypersurface of $\mathbb{H}^2\times \mathbb{H}^2$ with constant sectional curvature $-\frac{1}{2}$ and $C=0$.
By Lemma \ref{lemma:2.3}, {\bf Case II} of Proposition \ref{prop:4.5} and $C=0$, we have the following lemma.

\begin{lemma}\label{lemma:4.7}
Let $M$ be a hypersurface of $\mathbb{H}^2\times \mathbb{H}^2$ with constant sectional curvature $\kappa=-\frac{1}{2}$ and $C=0$.
Then, under the frame $\{E_1,E_2,E_3\}$ defined by \eqref{eqn:2.5},
the connections satisfy
\begin{equation}\label{eqn:2.25a11}
\begin{aligned}
&A E_1=b_1 E_1+b_2 E_2,\ A E_2=b_2 E_1+b_4 E_2,\ A E_3=0,\\		
&\nabla_{E_1}E_1=b_1E_3,\
\nabla_{E_1}E_2=-b_2E_3,\ \nabla_{E_1}E_3=-b_1E_1
+b_2E_2,\\
&\nabla_{E_2} E_1=b_2E_3,\
\nabla_{E_2} E_2=-b_4E_3,\ \nabla_{E_2}E_3=-b_2E_1
+b_4E_2,\\
&\nabla_{E_3}E_1=\nabla_{E_3}E_2=\nabla_{E_3}E_3=0.
\end{aligned}
\end{equation}
Moreover, the following six equations hold:
\begin{equation}\label{eqn:4.26}
E_3b_1=-\tfrac{1}{2}+b_1^{2}-b_2^2,
\end{equation}
\begin{equation}\label{eqn:4.27}
E_3b_2=b_2\big(b_1-b_4\big),
\end{equation}
\begin{equation}\label{eqn:4.28}
E_3b_4=\tfrac{1}{2}+b_2^{2}-b_4^2,
\end{equation}
\begin{equation}\label{eqn:4.29}
E_1b_2-E_2b_1=0,	
\end{equation}
\begin{equation}\label{eqn:4.30}
E_1b_4-E_2b_2=0,
\end{equation}
\begin{equation}\label{eqn:4.31asas}
b_1b_4-b_2^2=-\tfrac{1}{2}.
\end{equation}
\end{lemma}

Then, by using \eqref{eqn:4.26}--\eqref{eqn:4.28} and \eqref{eqn:4.31asas},
we have the following lemma.

\begin{lemma}\label{lemma:4.1aaa1}
Let $M$ be a hypersurface of $\mathbb{H}^2\times \mathbb{H}^2$ with constant sectional curvature $\kappa=-\frac{1}{2}$ and $C=0$.
Then, under the frame $\{E_1,E_2,E_3\}$ defined by \eqref{eqn:2.5}, we have the following six cases
of $b_1,b_2$ and $b_4$ on $M$:

\noindent{\bf Case 1}:
$$
b_1=-\tfrac{1}{\sqrt{2}},\ \ b_2=0,\ \
b_4=\tfrac{1}{\sqrt{2}}.
$$

\noindent{\bf Case 2}:
$$
b_1=-\tfrac{1}{\sqrt{2}}\tanh(\tfrac{t+k}{\sqrt{2}}),\ \ b_2=0,\ \
b_4=\tfrac{1}{\sqrt{2}}\coth(\tfrac{t+k}{\sqrt{2}}).
$$

\noindent{\bf Case 3}:
\begin{equation}\label{eq:jj1}
\begin{aligned}
&b_1=-\tfrac{1}{\sqrt{2}}+b_2,\ b_4=\tfrac{1}{\sqrt{2}}+b_2,\ b_2=h_1e^{-\sqrt{2}(t)}.\\
\end{aligned}	
\end{equation}

\noindent{\bf Case 4}:
\begin{equation}\label{eq:jj2}
\begin{aligned}
&b_1=-\tfrac{1}{\sqrt{2}}-b_2,\ b_4=\tfrac{1}{\sqrt{2}}-b_2,\ b_2=h_2e^{-\sqrt{2}(t)}.\\
\end{aligned}	
\end{equation}

\noindent{\bf Case 5}:
\begin{equation}\label{eq:jj3}
\begin{aligned}
b_1&=\tfrac{1}{\sqrt{2}}{\rm sech}(\sqrt{2}(t+l_1))\Big(\sinh (h_3)-\sinh(\sqrt{2}(t+l_1))\Big),\\
b_2&=\tfrac{1}{\sqrt{2}}{\rm sech}(\sqrt{2}(t+l_1))\cosh (h_3) ,\\
b_4&=\tfrac{1}{\sqrt{2}}{\rm sech}(\sqrt{2}(t+l_1))\Big(\sinh (h_3)+\sinh(\sqrt{2}(t+l_1))\Big).
\end{aligned}	
\end{equation}

\noindent{\bf Case 6}: 
\begin{equation}\label{eq:jj4}
\begin{aligned}
b_1&=\tfrac{1}{\sqrt{2}}{\rm csch}(\sqrt{2}(t+l_2))\Big(\cosh(h_4)-\cosh(\sqrt{2}(t+l_2))\Big),\\
b_2&=\tfrac{1}{\sqrt{2}}{\rm csch}(\sqrt{2}(t+l_2))\sinh(h_4),\\
b_4&=\tfrac{1}{\sqrt{2}}{\rm csch}(\sqrt{2}(t+l_2))\Big(\cosh(h_4)+\cosh(\sqrt{2}(t+l_2))\Big).
\end{aligned}	
\end{equation}
Here, the function $t$ satisfies $E_3t=1$ and the functions $h_1,h_2,h_3,h_4, k, l_1, l_2:M\rightarrow \mathbb{R}$
satisfy $E_3h_1=E_3h_2=E_3h_3=E_3h_4=E_3k=E_3l_1=E_3l_2=0$. Moreover $h_1,h_2,h_4$ are not identically equal to $0$. 
\end{lemma}
\begin{proof}
Assume that the hypersurface $M$ has constant sectional curvature $-\frac{1}{2}$ and $C=0$.
We will discuss two cases depending on the function $b_2$ on $M$.

{\bf Case I}: $b_2$ is constant.

When $b_2$ is constant on $M$, by \eqref{eqn:4.27}, we have $b_2(b_1-b_4)=0$.
If $b_2$ is a non-zero constant, then it holds $b_1=b_4$ and $b_1^2-b_2^2=-\frac{1}{2}$. Now,
$$
E_3b_1=-\tfrac{1}{2}+b_1^{2}-b_2^2=-1,\ \ E_3b_4=\tfrac{1}{2}+b_2^{2}-b_4^2=1,
$$
which contradicts $b_1=b_4$.

If $b_2=0$, then $b_1$ and $b_4$ satisfy $E_3b_1=-\tfrac{1}{2}+b_1^2$ and $E_3b_4=\tfrac{1}{2}-b_4^2$.
By directly integrating these equations along the integral curves of $E_3$,
and combining $b_1b_4=-\tfrac{1}{2}$, up to changing the sign of the unit normal vector field,
we have
$$
\begin{aligned}
(b_1,b_4)=(-\tfrac{1}{\sqrt{2}},\tfrac{1}{\sqrt{2}}),\ &{\rm or}\
(b_1,b_4)=(-\tfrac{1}{\sqrt{2}}\tanh(\tfrac{t+k}{\sqrt{2}}),
\tfrac{1}{\sqrt{2}}\coth(\tfrac{t+k}{\sqrt{2}})),\\
&  {\rm or}\ (b_1,b_4)=(-\tfrac{1}{\sqrt{2}}\coth(\tfrac{t+k}{\sqrt{2}}),
\tfrac{1}{\sqrt{2}}\tanh(\tfrac{t+k}{\sqrt{2}})),
\end{aligned}
$$
where $k$ is a function on $M$ satisfying $E_3k=0$.
In fact, if $(b_1,b_4)=(-\tfrac{1}{\sqrt{2}}\coth(\tfrac{t+k}{\sqrt{2}}),
\tfrac{1}{\sqrt{2}}\tanh(\tfrac{t+k}{\sqrt{2}}))$, then by using the isometry
$$
\mathbb{F}:\mathbb{H}^2\times \mathbb{H}^2\rightarrow\mathbb{H}^2\times \mathbb{H}^2,\ \
\mathbb{F}(p,q)=(q,p),
$$
we can have another hypersurface $\mathbb{F}(M)$ of $\mathbb{H}^2\times \mathbb{H}^2$
with constant sectional curvature $-\frac{1}{2}$ and $C=0$. The unit normal vector field
of $\mathbb{F}(M)$ is $\tilde{N}=d\mathbb{F}(N)$.
Now, define the frame $\{\tilde{E}_1,\tilde{E}_2,\tilde{E}_3\}$ as \eqref{eqn:2.5} on $\mathbb{F}(M)$,
then the corresponding components of the shape operator on $\mathbb{F}(M)$ are given by:
$$
\tilde{b}_1=b_4,\ \tilde{b}_2=b_2=0,\ \tilde{b}_4=b_1.
$$
Thus, if $b_2=0$, up to an isometry of $\mathbb{H}^2\times \mathbb{H}^2$
or up to changing the sign of the unit normal vector field, we obtain {\bf Case 1} and {\bf Case 2}.

\vskip 2mm

{\bf Case II}: $b_2$ is not constant.

When $b_2$ is not constant on $M$, let $F=b_1-b_4$, we get
$$
E_3F=E_3(b_1-b_4)=-2+(b_1-b_4)^2=-2+F^2.
$$
This equation can be directly integrated along the integral curves of $E_3$
to obtain
\begin{equation}\label{eqn:5.1a}
b_1-b_4=\pm\sqrt{2},\ {\rm or}\ b_1-b_4=-\sqrt{2}\tanh(\sqrt{2}(t+l)),
\ {\rm or}\ b_1-b_4=-\sqrt{2}\coth(\sqrt{2}(t+l)),
\end{equation}
where $l$ is a function on $M$ satisfying $E_3l=0$.
Note that, for the case of $b_1-b_4=\pm\sqrt{2}$, up to the isometry $\mathbb{F}$ of $\mathbb{H}^2\times \mathbb{H}^2$, and up to changing the sign of the unit normal vector field, 
we may orient the normal and interchange factors to choose $b_1-b_4=-\sqrt{2}$. 

Then by $E_3b_2=b_2\big(b_1-b_4\big)=b_2F$, with the use of \eqref{eqn:5.1a}, we integrate along the integral curves of $E_3$
to obtain
\begin{equation}\label{eqn:5.2a}
b_2=he^{-\sqrt{2}t}, \ {\rm or}\ b_2=h\, {\rm sech}(\sqrt{2}(t+l)),
\ {\rm or}\ b_2=h\, {\rm csch}(\sqrt{2}(t+l)),
\end{equation}
where $h$ is a smooth function on $M$ satisfying $E_3h=0$. 

Now, let $\tilde{F}=b_1+b_4$, by \eqref{eqn:4.26} and \eqref{eqn:4.28}, we get
$$
E_3\tilde{F}=E_3(b_1+b_4)=(b_1+b_4)(b_1-b_4)=F\tilde{F}.
$$
Using \eqref{eqn:5.1a}, we integrate along the integral curves of $E_3$
to obtain
$$
b_1+b_4=c_3e^{-\sqrt{2}t}, \ {\rm or}\ \ b_1+b_4=c_3{\rm sech}(\sqrt{2}(t+l)),
\ {\rm or}\ \ b_1+b_4=c_3{\rm csch}(\sqrt{2}(t+l)),
$$
where $c_3$ is a function on $M$ satisfying $E_3c_3=0$. Thus, we have the following three situations:
$$
\begin{aligned}
&b_1=-\tfrac{1}{\sqrt{2}}+\tfrac{c_3}{2}e^{-\sqrt{2}t},\ b_4=\tfrac{1}{\sqrt{2}}+\tfrac{c_3}{2}e^{-\sqrt{2}t};\\
&b_1=\tfrac{1}{2}{\rm sech}(\sqrt{2}(t+l))\Big(c_3-\sqrt{2}\sinh(\sqrt{2}(t+l))\Big),\
b_4=\tfrac{1}{2}{\rm sech}(\sqrt{2}(t+l))\Big(c_3+\sqrt{2}\sinh(\sqrt{2}(t+l))\Big);\\
&b_1=\tfrac{1}{2}{\rm csch}(\sqrt{2}(t+l))\Big(c_3-\sqrt{2}\cosh(\sqrt{2}(t+l))\Big),\
b_4=\tfrac{1}{2}{\rm csch}(\sqrt{2}(t+l))\Big(c_3+\sqrt{2}\cosh(\sqrt{2}(t+l))\Big).
\end{aligned}
$$
From the fact that $b_1b_4-b_2^2=-\frac{1}{2}$ and using \eqref{eqn:5.2a},
corresponding to the above three situations, we have
$$
\begin{aligned}
&c_3^2=4h^2;\ \ \ {\rm or}\ \ 4h^2-c_3^2=2;\ \ \ {\rm or}\ \ 4h^2-c_3^2=-2.
\end{aligned}
$$

For the first situation that $c_3^2=4h^2$, we have $c_3=\pm2 h$.
Now, we have two subcases as follows:
\begin{equation*}
\left\{
\begin{aligned}
&(1)\ b_1=-\tfrac{1}{\sqrt{2}}+b_2,\ b_2=h e^{-\sqrt{2}(t)},\ b_4=\tfrac{1}{\sqrt{2}}+b_2.\\
&(2)\ b_1=-\tfrac{1}{\sqrt{2}}-b_2,\ b_2=h e^{-\sqrt{2}(t)},\ b_4=\tfrac{1}{\sqrt{2}}-b_2.
\end{aligned}\right.
\end{equation*}
Thus, we have {\bf Case 3} and {\bf Case 4}.

For the second situation that $4h^2-c_3^2=2$, up to an isometry of $\mathbb{H}^2\times \mathbb{H}^2$,
we assume $h=\frac{1}{\sqrt{2}}\cosh h_3$ and $c_3=\sqrt{2}\sinh h_3$, where $h_3$ is a function on $M$ satisfying $E_3h_3=0$. Thus we obtain {\bf Case 5}.

For the third situation that $4h^2-c_3^2=-2$, up to an isometry of $\mathbb{H}^2\times \mathbb{H}^2$,
we assume $h=\frac{1}{\sqrt{2}}\sinh h_4$ and $c_3=\sqrt{2}\cosh h_4$, where $h_4$ is a function on $M$ satisfying $E_3h_4=0$. Thus we obtain {\bf Case 6}.
\end{proof}

For {\bf Case 1} and {\bf Case 2} of Lemma \ref{lemma:4.1aaa1}, we have the following proposition.

\begin{proposition}\label{prop:4.1a1}
Let $M$ be a hypersurface of $\mathbb{H}^2\times \mathbb{H}^2$ with constant sectional curvature $-\frac{1}{2}$ and $C=0$.
Under the frame $\{E_1,E_2,E_3\}$ defined by  \eqref{eqn:2.5}, if $b_1,b_2$ and $b_4$ are
given by {\bf Case 1} or {\bf Case 2} of Lemma \ref{lemma:4.1aaa1}, then 
$M$ is an open part of $M_{k_1,\frac{1}{k_1}}^{1/2}$ given in Example \ref{exam:3.4fg1} for some constant $0<k_1\leq1$.
More specifically, {\bf Case 1} corresponds to $M_{1,1}^{1/2}$, and
{\bf Case 2} corresponds to $M_{k_1,\frac{1}{k_1}}^{1/2}$ for some constant $0<k_1<1$.
\end{proposition}
\begin{proof}
Recall that $b_2=0$ in {\bf Case 1} and {\bf Case 2} of Lemma \ref{lemma:4.1aaa1}, and we have classified the hypersurfaces of $\mathbb{H}^2\times \mathbb{H}^2$ with constant $C$
and constant $b_2$ in Proposition \ref{prop:3.61}. Then, Proposition \ref{prop:4.1a1} can be obtained directly by using (2) of Lemma \ref{lemma:3.5fg12}.
\end{proof}

Now, we begin to deal with {\bf Case 3} and {\bf Case 4} of Lemma \ref{lemma:4.1aaa1}. 
For clarity we first prove the coordinate construction, including the vanishing of all three
coordinate Lie brackets. 

\begin{lemma}\label{lemma:5.1aasa1}
Let $M$ be a hypersurface of $\mathbb{H}^2\times \mathbb{H}^2$ with constant sectional curvature $-\frac{1}{2}$ and $C=0$. Let $\sigma\in\{1,-1\}$. Under the frame $\{E_1,E_2,E_3\}$ defined by \eqref{eqn:2.5}, if $b_1,b_2$ and $b_4$ satisfy
\begin{equation}\label{eq:of-branch}
b_1=-\frac{1}{\sqrt{2}}+\sigma b_2,\qquad b_4=\frac{1}{\sqrt{2}}+\sigma b_2
\end{equation}
on an open neighbourhood, then around each point there are local
coordinates $(u,v,t)$ and just one smooth real-valued function
$\gamma(v)$ such that, with
\begin{equation}\label{eq:of-LH}
b_2=\gamma(v)e^{-\sqrt{2}t},\qquad
L_1=\sqrt{2}\sigma\gamma(v)(u^2-e^{-\sqrt{2}t}),\qquad
L_2=-4\sigma\gamma(v)u,
\end{equation}
the canonical frame satisfies
\begin{equation}\label{eq:of-frame}
\begin{aligned}
E_1-\sigma E_2=e^{-t/\sqrt{2}}\partial_u,\quad 
E_1+\sigma E_2=e^{-t/\sqrt{2}}
(L_1\partial_u+\partial_v+L_2\partial_t),\quad 
E_3=\partial_t. 
\end{aligned}
\end{equation} 
The function $\gamma(v)$ may change sign, vanish at points, or vanish
identically. Neither the coordinate map nor the frame degenerates at
such zeros.
\end{lemma}

\begin{proof}
Write
$$
U=E_1-\sigma E_2,\qquad V=E_1+\sigma E_2,
\qquad T=E_3.
$$
The connection formulas \eqref{eqn:2.25a11} and
\eqref{eq:of-branch} give
\begin{equation}\label{eq:of-brackets}
[U,T]=\frac{1}{\sqrt{2}}U,\qquad [V,T]=\frac{1}{\sqrt{2}}V-2\sigma b_2U,
\qquad[U,V]=-4\sigma b_2T.
\end{equation}
By \eqref{eqn:4.27}, $T(b_2)=-\sqrt{2}b_2$. By
\eqref{eqn:4.29}--\eqref{eqn:4.30}, $U(b_2)=0$. 
Since $[U,T]=U/\sqrt2$, Frobenius' theorem, followed by
straightening $U$ within the leaves, gives local coordinates
$(r,v,s)$ such that
$$
U=\partial_r,\qquad T=a\,\partial_r+b\,\partial_s,\qquad b\ne0.
$$
Here $b\ne0$ follows from the linear independence of $U,T$.
Comparing coefficients in the bracket identity gives
$\partial_r a=1/\sqrt2$ and $\partial_r b=0$, so $b=b(v,s)$.
Define $\hat t(v,s)=\int_{s_0}^{s}\frac{d\tau}{b(v,\tau)}$ 
for a fixed $s_0$ in the coordinate interval. Since
$\partial_s\hat t=1/b\ne0$, the functions
$(r,v,\hat t)$ are local coordinates after shrinking the
neighborhood. They satisfy $U\hat t=0$,
$T\hat t=1$ and $Uv=Tv=0$. In these coordinates,
$$
U=\partial_r,\qquad
T=\left(\frac r{\sqrt2}+\alpha(v,\hat t)\right)\partial_r
+\partial_{\hat t}
$$
for a smooth function $\alpha$. Set 
$u=e^{-\hat t/\sqrt2}r
-\int_0^{\hat t}e^{-\tau/\sqrt2}\alpha(v,\tau)\,d\tau$. 
Direct differentiation gives
$$
Uu=e^{-\hat t/\sqrt2},\qquad
Tu=e^{-\hat t/\sqrt2}
\left(\frac r{\sqrt2}+\alpha-\frac r{\sqrt2}-\alpha\right)=0.
$$
Moreover,
$$
\det\frac{\partial(u,v,\hat t)}
{\partial(r,v,\hat t)}=e^{-\hat t/\sqrt2}\ne0.
$$
Thus $(u,v,\hat t)$ is a local coordinate system.
Writing $t=\hat t$, the action of $U,T$ on these three
coordinate functions yields
\begin{equation}\label{eq:of-initial}
	T=\partial_t,\qquad U=e^{-t/\sqrt2}\partial_u,
	\qquad Tt=1,\quad Ut=0,\quad Uv=Tv=0.
\end{equation}
Integration of the two equations for $b_2$ then yields
\begin{equation}\label{eq:of-h}
b_2=h(v)e^{-\sqrt{2}t}
\end{equation}
for a smooth real-valued $h$.

To determine the third field, write
$V=e^{-t/\sqrt{2}}(F\partial_u+B\partial_v+L\partial_t)$.
Taking the coefficients of $\partial_u,\partial_v,\partial_t$
in the last two brackets of \eqref{eq:of-brackets} gives
\begin{equation}\label{eq:of-PDE}
\begin{aligned}
&F_t=2\sigma h(v)e^{-\sqrt{2}t},\quad B_t=L_t=0,\\
&F_u=-\frac{1}{\sqrt{2}}L,\quad B_u=0,\quad L_u=-4\sigma h(v).
\end{aligned}
\end{equation}
Consequently, for some smooth functions $B(v),c(v),\beta(v)$,
\begin{equation}\label{eq:of-raw}
\begin{aligned}
V=e^{-t/\sqrt{2}}\Big[&\big(\sqrt{2}\sigma hu^2-\frac{1}{\sqrt{2}}c(v)u+\beta(v)
-\sqrt{2}\sigma he^{-\sqrt{2}t}\big)\partial_u\\
&+B(v)\partial_v+
(c(v)-4\sigma hu)\partial_t\Big].
\end{aligned}
\end{equation}
Since $U,V,T$ are linearly independent, $B(v)\neq0$. Replacing 
$v$ by a primitive of $1/B(v)$ gives $B=1$. In the remainder of
this proof all functions $h,c,\beta$ are those after this
reparameterization.

We now remove both $c(v)$ and $\beta(v)$, rather than 
keeping $\beta$ as an additional arbitrary function. Solve the
following initial-value problem for $\phi(v)$ and $k(v)$:
\begin{equation}\label{eq:of-ODE}
\left\{
\begin{aligned}
\phi'&=-c(v)-4\sigma h(v)e^{\phi/\sqrt{2}}k,\\
k'&=-e^{-\phi/\sqrt{2}}\beta(v)+\sqrt{2}\sigma h(v)e^{\phi/\sqrt{2}}k^2.
\end{aligned}
\right.
\end{equation}
The right-hand sides are smooth for all real $h$, even at its
zeros. By the standard local existence and uniqueness theorem,
there is a smooth solution through any prescribed $\phi(v_0)$ and
$k(v_0)$. Introduce new coordinates by
\begin{equation}\label{eq:of-change}
\quad \tilde t=t+\phi(v),\qquad
\tilde u =e^{-\phi(v)/\sqrt{2}}u+k(v),\qquad
\tilde v =\int_{v_0}^{v}e^{-\phi(s)/\sqrt{2}}\,ds.\quad
\end{equation}
The Jacobian determinant is $e^{-\sqrt{2}\phi(v)}\neq0$, so these are
valid smooth local coordinates. 
The conditions in \eqref{eq:of-change} immediately imply
$$
T=\partial_{\tilde t },\qquad U=e^{-\tilde t /\sqrt{2}}\partial_{\tilde u},
\qquad V(\tilde v )=e^{-\tilde t /\sqrt{2}}.
$$
For the two remaining components, by \eqref{eq:of-raw} and \eqref{eq:of-change}, we have 
\begin{equation}\label{eq:of-transform}
\begin{aligned}
e^{ \tilde t /\sqrt{2}}V(\tilde t )&=e^{\phi/\sqrt{2}}(c+\phi'-4\sigma hu),\\
e^{\tilde  t /\sqrt{2}}V(\tilde u )&=\sqrt{2}\sigma hu^2-\frac{1}{\sqrt{2}}(c+\phi')u
+\beta+e^{\phi/\sqrt{2}}k'-\sqrt{2}\sigma he^{-\sqrt{2}t}.
\end{aligned}
\end{equation}
Define a single function of the new transverse variable by
\begin{equation}\label{eq:of-gamma}
\gamma(\tilde v )=e^{\sqrt{2}\phi(v)}h(v).
\end{equation}
Using \eqref{eq:of-ODE}, the right-hand sides of
\eqref{eq:of-transform} become, respectively,
$$
-4\sigma\gamma(\tilde v )\tilde u ,\qquad
\sqrt{2}\sigma\gamma(\tilde v )\big((\tilde u )^2-e^{-\sqrt{2}\tilde t }\big).
$$
Moreover $b_2=\gamma(\tilde v )e^{-\sqrt{2}\tilde t }$. Thus the $c$ and $\beta$
terms have both disappeared. For notational simplicity, dropping the tildes proves 
\eqref{eq:of-LH}--\eqref{eq:of-frame}. 

For completeness, we verify that the resulting coordinate fields
commute, rather than merely invoking the coordinate theorem.
Set
\begin{equation}\label{eq:of-X}
\begin{aligned}
X_1&=e^{t/\sqrt{2}}(E_1-\sigma E_2)=e^{t/\sqrt{2}}U,\\
X_2&=e^{t/\sqrt{2}}(E_1+\sigma E_2)-L_1X_1-L_2X_3,\\
X_3&=E_3=T.
\end{aligned}
\end{equation}
The normal form gives the explicit identities
\begin{equation}\label{eq:of-deriv}
\begin{aligned}
X_1L_1=-\frac{1}{\sqrt{2}}L_2, \quad  X_1L_2=-4\sigma\gamma(v), \quad 
X_3L_1=2\sigma b_2,  \quad  X_3L_2=0.
\end{aligned}
\end{equation}
Then by \eqref{eq:of-brackets}, we have 
\begin{equation}\label{eq:of-brcalc}
\begin{aligned}
[X_1,X_3]=0, \quad 
[X_1,e^{t/\sqrt{2}}V]=-\frac{1}{\sqrt{2}}L_2X_1-4\sigma\gamma(v)X_3, \quad 
[e^{t/\sqrt{2}}V,X_3]=-2\sigma b_2 X_1.
\end{aligned}
\end{equation}
Consequently, it follows that 
\begin{equation}\label{eq:of-zero-brackets}
\begin{aligned}
[X_1,X_2]&=(-\frac{1}{\sqrt{2}}L_2-X_1L_1)X_1+(-4\sigma\gamma-X_1L_2)X_3=0,\\
[X_2,X_3]&=(-2\sigma b_2+X_3L_1)X_1+(X_3L_2)X_3=0,\\
[X_1,X_3]&=0.
\end{aligned}
\end{equation}
As $X_1,X_2,X_3$ are independent, they are exactly the commuting
coordinate fields $\partial_u,\partial_v,\partial_t$ of
\eqref{eq:of-frame}. This checks all three Lie brackets in detail.
\end{proof}

\begin{theorem}\label{thm:5.10a1}
Let $M$ be a hypersurface of $\mathbb H^2\times\mathbb H^2$ with
constant sectional curvature $-\frac{1}{2}$ and $C=0$. Under the frame
$\{E_1,E_2,E_3\}$ defined by \eqref{eqn:2.5}, suppose that
$b_1,b_2,b_4$ are given by {\bf Case 3} of Lemma
\ref{lemma:4.1aaa1}. Then $M$ admits local coordinates $(u,v,t)$
and a smooth real-valued function $\beta_1(v)$ such that
\begin{equation}\label{eqn:3.1ff1}
\quad b_1=-\tfrac1{\sqrt2}+\beta_1(v)e^{-\sqrt2t},\quad
b_2=\beta_1(v)e^{-\sqrt2t},\quad
b_4=\tfrac1{\sqrt2}+\beta_1(v)e^{-\sqrt2t}.\quad
\end{equation} 
	
Conversely, let $\mathcal D\subset\mathbb R^2$ and
$I\subset\mathbb R$ be open and let $\beta_1:I\to\mathbb R$ be
any smooth function with $\beta_1\not\equiv 0$. On
$\Omega=\{(u,v,t)|\,(u,t)\in\mathcal D,\ v\in I\}$ put
\begin{equation}\label{eq:of-plus-def}
\begin{aligned}
L_1=\sqrt2\,\beta_1(v)(u^2-e^{-\sqrt2t}),\ \ 
L_2=-4\beta_1(v)u,\ \ 
L_3=2b_2+\tfrac1{\sqrt2}L_1,\ \ b_2=\beta_1(v)e^{-\sqrt2t}. 
\end{aligned}
\end{equation}
We define a metric $g$ on $\Omega$ by
\begin{equation}\label{eqn:3.1ff3}
\quad g=2e^{\sqrt2t}\big[(du-L_1dv)^2+dv^2\big]+(dt-L_2dv)^2, 
\end{equation}
and a $(1,1)$-tensor field $A$ on $T\Omega$ by
\begin{equation}\label{eqn:3.1ff4}
\begin{aligned}
A\partial_u&=-\tfrac1{\sqrt2}(L_1\partial_u+\partial_v+L_2\partial_t),\\
A\partial_v&=(-\tfrac1{\sqrt2}+L_3L_1)\partial_u+L_3\partial_v+L_2L_3\partial_t,\\
A\partial_t&=0.
\end{aligned}
\end{equation}
Let $\nu$ be a vector bundle over  $\Omega$ of rank $1$ with metric $\tilde g$, $\nabla ^{\perp} $ a connection on $\nu$ compatible with the metric $\tilde g$, and $N$ a unit vector field in $\nu$. We define a $(1,1)$-tensor field $\tilde{P}$ on $T\Omega\oplus\nu$ by 
\begin{equation}\label{eqn:3.1ff5} 
\begin{aligned}
\tilde P\partial_u&=L_1\partial_u+\partial_v+L_2\partial_t,\\
\tilde P\partial_v&=(1-L_1^2)\partial_u-L_1\partial_v-L_1L_2\partial_t-L_2N,\\
\tilde P\partial_t&=N,\qquad \tilde PN=\partial_t.
\end{aligned}
\end{equation}
Then the integrability conditions are satisfied on $\Omega$ and hence on any simply
connected open subset $\Omega_1$ of $\Omega$, up to an isometry of  $\mathbb{H}^2\times \mathbb{H}^2$, there is a unique isometric immersion $\Psi$ from $(\Omega_1,~g)$ into $\mathbb{H}^2\times \mathbb{H}^2$ such that $A$ is the shape operator of  $\Omega_1$ into $\mathbb{H}^2\times \mathbb{H}^2$, $\nu$ is isometric to the normal bundle of $\Psi(\Omega_1)$ in
$\mathbb{H}^2\times \mathbb{H}^2$ by an isomorphism $\tilde{\Psi}:\nu\to T^{\perp}\Psi(\Omega_1)$,
and for any $Y\in T\Omega_1$, we have
$$
P(\Psi_*Y)=\Psi_*((\tilde PY)^\top)
+\tilde\Psi((\tilde PY)^\perp),\qquad
P(\tilde\Psi N)=\Psi_*\partial_t, 
$$
where $P$ is the product structure of $\mathbb{H}^2\times \mathbb{H}^2$, $(\tilde{P}Y)^\top$ and $(\tilde{P}Y)^{\perp}$ denote the projections of  $\tilde{P}Y$ onto $T\Omega_1$ and $\nu$, respectively.
Moreover, $\Psi(\Omega_1)$ has constant sectional curvature $-\frac{1}{2}$ and product angle function $C=0$.
\end{theorem}

\begin{proof}
Apply Lemma \ref{lemma:5.1aasa1} with $\sigma=1$ and
$\gamma(v)=\beta_1(v)$. Explicitly, $U=E_1-E_2$, $V=E_1+E_2$,
$T=E_3$, and the commuting coordinate fields are
\begin{equation}\label{eqn:3.1ff5a1} 
\partial_u=e^{t/\sqrt2}U, \quad 
\partial_v=e^{t/\sqrt2}V-L_1e^{t/\sqrt2}U-L_2T, \quad 
\partial_t=T. 
\end{equation}
They satisfy all three zero-bracket identities
\eqref{eq:of-zero-brackets}. 
	
We explain how the metric and the two tensors arise in the
coordinate basis. The orthonormal canonical frame satisfies
$g(U,U)=g(V,V)=2$, $g(U,V)=0$, and $T$ is unit
and orthogonal to $U,V$. Since
$$
U=e^{-t/\sqrt2}\partial_u,\quad
V=e^{-t/\sqrt2}(L_1\partial_u+\partial_v+L_2\partial_t),
\quad T=\partial_t,
$$
its dual coframe gives exactly \eqref{eqn:3.1ff3}.
The matrix of $A$ in the $E_1,E_2,E_3$ frame is
$$
\begin{pmatrix}
	-\frac{1}{\sqrt{2}}+b_2&b_2&0\\
	b_2&\frac{1}{\sqrt{2}}+b_2&0\\
	0&0&0
\end{pmatrix},
$$
so $AU=-\frac{1}{\sqrt{2}}V$, $AV=-\frac{1}{\sqrt{2}}U+2b_2V$, and $AT=0$.
Hence applying $A$ to \eqref{eqn:3.1ff5a1}  yields precisely
\eqref{eqn:3.1ff4}. Likewise, from
$$
PU=V,\quad PV=U,\quad
PT=N,\quad PN=T,
$$
one obtains all four formulas \eqref{eqn:3.1ff5}. This proves the asserted formulas directly from
the original canonical frame.

Conversely, take any smooth function $\beta_1:I\to\mathbb{R}$ that is not identically zero, and define the metric $g$ on $\Omega$, the $(1,1)$-tensor $A$ on $T\Omega$, and the $(1,1)$-tensor  $\tilde{P}$ on  $T\Omega\oplus\nu$ by \eqref{eqn:3.1ff3}, \eqref{eqn:3.1ff4}, and \eqref{eqn:3.1ff5}, respectively. 
Then it can be  checked  that all the integrability conditions are satisfied.
Therefore, the conclusions can be obtained by applying  the existence and uniqueness theorems in  \cite{K,L-T-V} directly. 
\end{proof}

\begin{theorem}\label{thm:5.11a1}
Let $M$ be a hypersurface of $\mathbb H^2\times\mathbb H^2$ with
constant sectional curvature $-1/2$ and $C=0$. Under the frame
$\{E_1,E_2,E_3\}$ defined by \eqref{eqn:2.5}, suppose that
$b_1,b_2,b_4$ are given by {\bf Case 4} of Lemma
\ref{lemma:4.1aaa1}. Then $M$ admits local coordinates $(u,v,t)$
and a smooth real-valued function $\beta_2(v)$ such that 
\begin{equation}\label{eqn:3.1ffff1}
\quad b_1=-\tfrac1{\sqrt2}-\beta_2(v)e^{-\sqrt2t},\quad
b_2=\beta_2(v)e^{-\sqrt2t},\quad
b_4=\tfrac1{\sqrt2}-\beta_2(v)e^{-\sqrt2t}.\quad
\end{equation}

Conversely, let $\mathcal D\subset\mathbb R^2$ and
$I\subset\mathbb R$ be open and let $\beta_2:I\to\mathbb R$ be
any smooth function with $\beta_2\not\equiv 0$. On
$\Omega=\{(u,v,t)|\,(u,t)\in\mathcal D,\ v\in I\}$ put 
\begin{equation}\label{eq:of-minus-def}
\begin{aligned}
\tilde L_1=-\sqrt2\,\beta_2(v)(u^2-e^{-\sqrt2t}),\ \ 
\tilde L_2=4\beta_2(v)u,\ \ 
\tilde L_3=-2b_2+\tfrac1{\sqrt2}\tilde L_1,\ \ 
b_2=\beta_2(v)e^{-\sqrt2t}. 
\end{aligned}
\end{equation}
We define a metric $g$ on $\Omega$ by
\begin{equation}\label{eqn:3.1ffff3}
\quad g=2e^{\sqrt2t}\big[(du-\tilde L_1dv)^2+dv^2\big]
+(dt-\tilde L_2dv)^2, 
\end{equation}
and a $(1,1)$-tensor $A$ on $T\Omega$ by
\begin{equation}\label{eqn:3.1ffff4}
\begin{aligned}
A\partial_u&=-\tfrac1{\sqrt2}(\tilde L_1\partial_u+\partial_v+\tilde L_2\partial_t),\\
A\partial_v&=(-\tfrac1{\sqrt2}+\tilde L_1\tilde L_3)\partial_u+\tilde L_3\partial_v
+\tilde L_2\tilde L_3\partial_t,\\
A\partial_t&=0.
\end{aligned}
\end{equation}
Let $\nu$ be a vector bundle over  $\Omega$ of rank $1$ with metric $\tilde g$, $\nabla ^{\perp} $ a connection on $\nu$ compatible with the metric $\tilde g$, and $N$ a unit vector field in $\nu$. We define a $(1,1)$-tensor field $\tilde{P}$ on $T\Omega\oplus\nu$ by 
\begin{equation}\label{eqn:3.1ffff5} 
\begin{aligned}
\tilde P\partial_u&=\tilde L_1\partial_u+\partial_v+\tilde L_2\partial_t,\\
\tilde P\partial_v&=(1-\tilde L_1^2)\partial_u-\tilde L_1\partial_v-\tilde L_1\tilde L_2\partial_t-\tilde L_2N,\\
\tilde P\partial_t&=N,\qquad
\tilde PN=\partial_t.
\end{aligned}
\end{equation}
Then the integrability conditions are satisfied on $\Omega$ and hence on any simply
connected open subset $\Omega_1$ of $\Omega$, up to an isometry of  $\mathbb{H}^2\times \mathbb{H}^2$, there is a unique isometric immersion $\Psi$ from $(\Omega_1,~g)$ into $\mathbb{H}^2\times \mathbb{H}^2$ such that $A$ is the shape operator of  $\Omega_1$ into $\mathbb{H}^2\times \mathbb{H}^2$, $\nu$ is isometric to the normal bundle of $\Psi(\Omega_1)$ in
$\mathbb{H}^2\times \mathbb{H}^2$ by an isomorphism $\tilde{\Psi}:\nu\to T^{\perp}\Psi(\Omega_1)$,
and for any $Y\in T\Omega_1$, we have
$$
P(\Psi_*Y)=\Psi_*((\tilde PY)^\top)
+\tilde\Psi((\tilde PY)^\perp),\qquad
P(\tilde\Psi N)=\Psi_*\partial_t, 
$$
where $P$ is the product structure of $\mathbb{H}^2\times \mathbb{H}^2$, $(\tilde{P}Y)^\top$ and $(\tilde{P}Y)^{\perp}$ denote the projections of  $\tilde{P}Y$ onto $T\Omega_1$ and $\nu$, respectively.
Moreover, $\Psi(\Omega_1)$ has constant sectional curvature $-\frac{1}{2}$ and product angle function $C=0$. 
\end{theorem}

\begin{proof}
Apply Lemma \ref{lemma:5.1aasa1} with $\sigma=-1$ and
$\gamma(v)=\beta_2(v)$. Explicitly, $U=E_1+E_2$, $V=E_1-E_2$,
$T=E_3$, and the commuting coordinate fields are 
\begin{equation}\label{eqn:3.1ffff5aa1} 
\partial_u=e^{t/\sqrt2}U, \quad 
\partial_v=e^{t/\sqrt2}V-\tilde L_1e^{t/\sqrt2}U-\tilde L_2E_3, \quad 
\partial_t=E_3. 
\end{equation}
They satisfy all three zero-bracket identities
\eqref{eq:of-zero-brackets}. 
	
For the explicit tensor formulas, note that
$AU=-\frac{1}{\sqrt{2}}V$, $AV=-\frac{1}{\sqrt{2}}U-2b_2V$, $AT=0$ and
$PU=V$, $PV=U$,
$PT=N$, $PN=T$. Substituting the
coordinate expressions of $U,V,T$ into these identities
gives exactly \eqref{eqn:3.1ffff4} and
\eqref{eqn:3.1ffff5}, while orthonormality gives the metric
\eqref{eqn:3.1ffff3}.

Conversely, take any smooth function $\beta_2:I\to\mathbb{R}$ that is not identically zero, and define the metric $g$ on $\Omega$, the $(1,1)$-tensor $A$ on $T\Omega$, and the $(1,1)$-tensor  $\tilde{P}$ on  $T\Omega\oplus\nu$ by \eqref{eqn:3.1ffff3}, \eqref{eqn:3.1ffff4}, and \eqref{eqn:3.1ffff5}, respectively. 
Then it can be  checked  that all the integrability conditions are satisfied.
Therefore, the conclusions can be obtained by applying  the existence and uniqueness theorems in  \cite{K,L-T-V} directly. 
\end{proof}

In the following, we deal with {\bf Case 5} and {\bf Case 6} of Lemma \ref{lemma:4.1aaa1}. By $b_1-b_4\neq\sqrt{2}$, we define
\begin{equation}\label{eqn:bibar}
b_1^*=\frac{b_1+b_4}{2-\sqrt{2}(b_1-b_4)},\ \
b_2^*=\frac{2b_2}{2-\sqrt{2}(b_1-b_4)}.
\end{equation}
Then, combining with $b_1b_4-b_2^2=-\tfrac{1}{2}$, we have
\begin{equation}\label{eqn:bi}
b_1=\frac{1}{\sqrt{2}}-\frac{\sqrt{2}-2b_1^*}{1-2b_1^{*2}+2b_2^{*2}},\
b_2=\frac{2b_2^*}{1-2b_1^{*2}+2b_2^{*2}},\
b_4=-\frac{1}{\sqrt{2}}+\frac{\sqrt{2}+2b_1^*}{1-2b_1^{*2}+2b_2^{*2}}.
\end{equation}

Now, by \eqref{eqn:4.26}--\eqref{eqn:4.30}, we have
\begin{equation}\label{eqn:Eibibar}
\begin{aligned}
&E_3b_1^*=-\sqrt{2}b_1^*,\ \
E_1b_1^*=-E_2b_2^*+\frac{2(\sqrt{2}b_2^*E_1b_2^*+(1-\sqrt{2}b_1^*)E_2b_2^*)}{1-2b_1^{*2}+2b_2^{*2}},\\
&E_3b_2^*=-\sqrt{2}b_2^*,\ \ E_2b_1^*=-E_1b_2^*+\frac{2((1+\sqrt{2}b_1^*)E_1b_2^*-\sqrt{2}b_2^*E_2b_2^*)}{1-2b_1^{*2}+2b_2^{*2}}.
\end{aligned}
\end{equation}
Next, we define
\begin{equation}\label{eqn:Eibar}
\begin{aligned}
&E_1^*=(1+\sqrt{2}b_1^{*}-\sqrt{2}b_2^{*})E_1+(-1+\sqrt{2}b_1^{*}-\sqrt{2}b_2^{*})E_2,\\
&E_2^*=(1+\sqrt{2}b_1^{*}+\sqrt{2}b_2^{*})E_1+(1-\sqrt{2}b_1^{*}-\sqrt{2}b_2^{*})E_2,\ \
E_3^*=E_3.
\end{aligned}
\end{equation}
Then, by \eqref{eqn:2.25a11}, \eqref{eqn:Eibibar} and \eqref{eqn:Eibar}, one can check that
\begin{equation}\label{eqn:LieEibar}
[E_1^*,E_3^*]=\frac{1}{\sqrt{2}}E_1^*,\ \ [E_2^*,E_3^*]=\frac{1}{\sqrt{2}}E_2^*,\ \
[E_1^*,E_2^*]=-8b_2^*E_3^*.
\end{equation}

\begin{theorem}\label{thm:5.12a1}
Let $M$ be a hypersurface of $\mathbb{H}^2\times \mathbb{H}^2$ with constant sectional curvature $-\frac{1}{2}$ and $C=0$.
Under the frame $\{E_1,E_2,E_3\}$ defined by \eqref{eqn:2.5}, if functions $b_1,b_2$ and $b_4$ defined in \eqref{eqn:2.7} are given by {\bf Case 5} of Lemma \ref{lemma:4.1aaa1},
then $M$ admits local coordinates $(u,v,t)$
and there is a function $h_3:(u,v)\mapsto \mathbb{R}$
such that the functions $b_1,b_2,b_4$ are given by
\eqref{eqn:bijth5},
where $h_3$ satisfies the following ``cosh-Gordon equation'':
\begin{equation}\label{eqn:3.1fffff2}
\frac{\partial^2 h_3}{\partial u \partial v}
=4\sqrt{2}\cosh(h_3).
\end{equation}

Conversely, let $\mathcal{D}$ be an open subset in $\mathbb{R}^2$ and  $h_3:\mathcal{D}\to\mathbb{R}, (u,v)\mapsto h_3(u,v)\in\mathbb{R}$
be a function which satisfies the differential equation \eqref{eqn:3.1fffff2}.
We define a metric $g$ on $\Omega=\{(u,v,t)\in \mathcal{D}\times \mathbb{R}\}$ by
\eqref{eqn:3.1fffff3},
and a $(1,1)$-tensor field $A$ on $T\Omega$ by
\eqref{eqn:3.1fffff4}.
Let $\nu$ be a vector bundle over  $\Omega$ of rank $1$ with metric $\tilde{g}$, $\nabla ^{\perp} $ a connection on $\nu$ compatible with the metric $\tilde{g}$, and $N$ a unit vector field in $\nu$.   We define a $(1,1)$-tensor field $\tilde{P}$ on $T\Omega\oplus \nu$ by
\eqref{eqn:3.1fffff5}.
Then the integrability conditions are satisfied on $\Omega$ and hence on any simply
connected open subset $\Omega_1$ of $\Omega$, up to an isometry of  $\mathbb{H}^2\times \mathbb{H}^2$, there is a unique isometric immersion $\Psi$ from $(\Omega_1,~g)$ into $\mathbb{H}^2\times \mathbb{H}^2$ such that $A$ is the shape operator of  $\Omega_1$ into $\mathbb{H}^2\times \mathbb{H}^2$, $\nu$ is isometric to the normal bundle of $\Psi(\Omega_1)$ in
$\mathbb{H}^2\times \mathbb{H}^2$ by an isomorphism $\tilde{\Psi}:\nu\to T^{\perp}\Psi(\Omega_1)$,
and for any $Y\in T\Omega_1$, we have
\begin{equation}
P(\Psi_{*}Y)=\Psi_{*}((\tilde{P}Y)^{\top})+\tilde{\Psi}((\tilde{P}Y)^{\perp}),~
P(\tilde{\Psi}N)=\Psi_*(\partial_t),
\end{equation}
where $P$ is the product structure of $\mathbb{H}^2\times \mathbb{H}^2$, $(\tilde{P}Y)^{\top}$ and $(\tilde{P}Y)^{\perp}$ denote the projections of  $\tilde{P}Y$ onto $T\Omega_1$ and $\nu$, respectively.
Moreover, $\Psi(\Omega_1)$ has constant sectional curvature $-\frac{1}{2}$ and product angle function $C=0$.
\end{theorem}
\begin{proof}
For {\bf Case 5} of Lemma \ref{lemma:4.1aaa1}, we have
$$
\begin{aligned}
&b_1^*=\frac{1}{\sqrt{2}}e^{-\sqrt{2}(t+l_1)}\sinh(h_3),\ \
b_2^*=\frac{1}{\sqrt{2}}e^{-\sqrt{2}(t+l_1)}\cosh(h_3),\\
&b_2^*-b_1^*>0,\ \ b_2^*+b_1^*>0,\ \ 1-2b_1^{*2}+2b_2^{*2}\neq0,
\end{aligned}
$$
where the function $t$ satisfies $E_3t=1$ and the functions $l_1, h_3:M\rightarrow \mathbb{R}$
satisfy $E_3l_1=E_3h_3=0$.

Now, we define $\rho_1=\frac{1}{\sqrt{b_2^*-b_1^*}}$ and $\rho_2=\frac{1}{\sqrt{b_2^*+b_1^*}}$.
Let $\tilde{E}_1^*=\rho_1E_1^*$, $\tilde{E}_2^*=\rho_2E_2^*$ and $\tilde{E}_3^*=E_3^*$.
Then, it can be checked that
\begin{equation}\label{eqn:Lietilde5}
[\tilde{E}_1^*,\tilde{E}_3^*]=0,\ \ [\tilde{E}_2^*,\tilde{E}_3^*]=0,\ \
[\tilde{E}_1^*,\tilde{E}_2^*]=-8\frac{b_2^*}{\sqrt{b_2^{*2}-b_1^{*2}}}\tilde{E}_3^*.
\end{equation}

Next, we define a differential system:
\begin{equation}\label{eqn:diff}
\left\{
\begin{aligned}
&\tilde{E}_1^*(u)=1,\ \ \tilde{E}_2^*(u)=0,\ \ \tilde{E}_3^*(u)=0,\\
&\tilde{E}_1^*(v)=0,\ \ \tilde{E}_2^*(v)=1,\ \ \tilde{E}_3^*(v)=0.
\end{aligned}\right.
\end{equation}
Then by \eqref{eqn:Lietilde5}, we know that there are functions $u$ and $v$ satisfying \eqref{eqn:diff} locally on $M$. Now, we choose 
$f=-\frac{1}{2\sqrt{2}}\ln\Big(2(b_2^{*2}-b_1^{*2})\Big)=t+l_1$.
In the following, we use $t$ to denote $f$.
We define the map
$$
\Phi: M \longrightarrow \mathbb{R}^{3},
\quad p \mapsto(u(p),v(p),t(p)).
$$
Then $\Phi_*$ is given by
\begin{equation}\label{eqn:dPhi5}
\Phi_*\left(
        \begin{array}{c}
          \tilde{E}_1^* \\
          \tilde{E}_2^* \\
          \tilde{E}_3^* \\
        \end{array}
      \right)
=\left(\begin{array}{ccc}
1 & 0 & \tilde{d}_1:=\tilde{E}_1^*(t) \\
0 & 1 & \tilde{d}_2:=\tilde{E}_2^*(t) \\
0 & 0 & 1
\end{array}\right)
\left(
\begin{array}{c}
	\partial_u \\
	\partial_v \\
	\partial_t \\
\end{array}
\right),
\end{equation}
which implies that $\{u,v,t\}$ are
local coordinates on $M$. It follows from $E_3h_3=0$ that the function $h_3$ depends only on variables $\{u,v\}$. Then for {\bf Case 5} of Lemma \ref{lemma:4.1aaa1},
we have
\begin{equation}\label{eqn:bijth5}
\begin{aligned}
b_1&=\tfrac{1}{\sqrt{2}}{\rm sech}(\sqrt{2}(t))\Big(\sinh (h_3)-\sinh(\sqrt{2}(t))\Big),\\
b_2&=\tfrac{1}{\sqrt{2}}{\rm sech}(\sqrt{2}(t))\cosh (h_3) ,\\
b_4&=\tfrac{1}{\sqrt{2}}{\rm sech}(\sqrt{2}(t))\Big(\sinh (h_3)+\sinh(\sqrt{2}(t))\Big),\\
b_1^*&=\frac{1}{\sqrt{2}}e^{-\sqrt{2}t}\sinh(h_3),\ \
b_2^*=\frac{1}{\sqrt{2}}e^{-\sqrt{2}t}\cosh(h_3).
\end{aligned}
\end{equation}

By using \eqref{eqn:Eibar}, \eqref{eqn:dPhi5}, $\tilde{E}_1^*=\rho_1E_1^*$
and $\tilde{E}_2^*=\rho_2E_2^*$, we have
\begin{equation}\label{eqn:transframe5}
	\begin{pmatrix}E_1\\E_2\end{pmatrix}
	=B^{-1}\begin{pmatrix}\rho_1^{-1}&0\\0&\rho_2^{-1}\end{pmatrix}
	\begin{pmatrix}\partial_u+\tilde d_1\partial_t\\
		\partial_v+\tilde d_2\partial_t\end{pmatrix},\qquad E_3=\partial_t,
\end{equation}
where
$$
B=\begin{pmatrix}
	1+\sqrt2b_1^*-\sqrt2b_2^*&-1+\sqrt2b_1^*-\sqrt2b_2^*\\
	1+\sqrt2b_1^*+\sqrt2b_2^*&1-\sqrt2b_1^*-\sqrt2b_2^*
\end{pmatrix},\quad
\det B=2(1-2(b_1^*)^2+2(b_2^*)^2)\ne0.
$$

Now, by using \eqref{eqn:4.29}, \eqref{eqn:4.30} and \eqref{eqn:transframe5}, we can get
$$
\tilde{d}_1=\frac{(h_3)_u}{\sqrt{2}},\ \ \tilde{d}_2=\frac{-(h_3)_v}{\sqrt{2}}.
$$
Then we can check the compatibility conditions of $b_1,b_2,b_4$ with respect to the Lie brackets $[E_i,E_j]$, $1\leq i<j\leq3$,
and obtain that
$$
\frac{\partial^2 h_3}{\partial u \partial v}
=4\sqrt{2}\cosh(h_3).
$$
In fact, the  compatibility conditions of  $b_1,b_2,b_4$  with respect to the Lie brackets $[E_1,E_3]$ and $[E_2,E_3]$ are automatically satisfied and the  compatibility conditions of  $b_1,b_2,b_4$  with respect to the Lie bracket $[E_1,E_2]$  imply the above differential equation.

Conversely, we assume that the function $h_3$ satisfies \eqref{eqn:3.1fffff2}, the metric $g$ on $\Omega$ is defined by
\begin{equation}\label{eqn:3.1fffff3}
\begin{aligned}
g=&\frac{1}{2}(8\sqrt{2}\cosh(\sqrt{2}t+h_3)+(h_3)_u^2)du^2
		+\frac{1}{2}(8\sqrt{2}\cosh(\sqrt{2}t-h_3)+(h_3)_v^2)dv^2+dt^2\\
		&+(8\sqrt{2}\sinh(h_3)-(h_3)_u(h_3)_v)dudv-(\sqrt{2}(h_3)_u)dudt+(\sqrt{2}(h_3)_v)dvdt,
\end{aligned}
\end{equation}
the $(1,1)$-tensor $A$ on $T\Omega$ is defined by
\begin{equation}\label{eqn:3.1fffff4}
\begin{aligned}
		A\partial_u=&\frac{{\rm sech}(\sqrt{2}t)\sinh(h_3)}{\sqrt{2}}\partial_u
		-\frac{\cosh(h_3)+\sinh(h_3)\tanh(\sqrt{2}t)}{\sqrt{2}}\partial_v\\
		&+\frac{1}{2}{\rm sech}(\sqrt{2}t)(\cosh(\sqrt{2}t+h_3)(h_3)_v+\sinh(h_3)(h_3)_u)\partial_t,\\
		A\partial_v=&\frac{-\cosh(h_3)+\sinh(h_3)\tanh(\sqrt{2}t)}{\sqrt{2}}\partial_u+
		\frac{{\rm sech}(\sqrt{2}t)\sinh(h_3)}{\sqrt{2}}\partial_v\\
		&-\frac{1}{2}{\rm sech}(\sqrt{2}t)(\cosh(\sqrt{2}t-h_3)(h_3)_u+\sinh(h_3)(h_3)_v)\partial_t,\\
		A\partial_t=&0,
\end{aligned}
\end{equation}
and the $(1,1)$-tensor $\tilde{P}$ on $T\Omega\oplus \nu$ is defined by
\begin{equation}\label{eqn:3.1fffff5}
	\begin{aligned}
		\tilde{P}\partial_u=&-\cosh(h_3){\rm sech}(\sqrt{2}t)\partial_u+(\sinh(h_3)+\cosh(h_3)\tanh(\sqrt{2}t))\partial_v\\
		&-\frac{{\rm sech}(\sqrt{2}t)(\sinh(\sqrt{2}t+h_3)(h_3)_v+\cosh(h_3)(h_3)_u)}{\sqrt{2}}\partial_t-\frac{(h_3)_u}{\sqrt{2}}N,\\
		\tilde{P}\partial_v=&(-\sinh(h_3)+\cosh(h_3)\tanh(\sqrt{2}t))\partial_u+\cosh(h_3){\rm sech}(\sqrt{2}t)\partial_v\\
		&+\frac{{\rm sech}(\sqrt{2}t)(\sinh(\sqrt{2}t-h_3)(h_3)_u-\cosh(h_3)(h_3)_v)}{\sqrt{2}}\partial_t+\frac{(h_3)_v}{\sqrt{2}}N,\\
		\tilde{P}\partial_t=&N,\\
		\tilde{P}N=&\partial_t.\\	
	\end{aligned}
\end{equation}
Then it can be  checked  that all the integrability conditions are satisfied.
Therefore, the conclusions can be obtained by applying  the existence and uniqueness theorems in  \cite{K,L-T-V} directly.
\end{proof}

\begin{theorem}\label{thm:5.13a1}
Let $M$ be a hypersurface of $\mathbb{H}^2\times \mathbb{H}^2$ with constant sectional curvature $-\frac{1}{2}$ and $C=0$.
Under the frame $\{E_1,E_2,E_3\}$ defined by \eqref{eqn:2.5}, if $b_1,b_2$ and $b_4$ are
given by {\bf Case 6} of Lemma \ref{lemma:4.1aaa1}, then $M$ admits local coordinates $(u,v,t)$
and there is a smooth function $h_4:(u,v)\mapsto \mathbb{R}$ 
such that the functions $b_1,b_2,b_4$ are given by \eqref{eqn:bijth6},
where $h_4$ satisfies the following ``sinh-Gordon equation'':
\begin{equation}\label{eqn:3.1ffffff2}
\frac{\partial^2 h_4}{\partial u \partial v}
=4\sqrt{2}\sinh(h_4).
\end{equation}

Conversely, let $\mathcal{D}\subset\mathbb{R}^2$ be an open subset, and let
$h_4:\mathcal{D}\to\mathbb{R},\ (u,v)\mapsto h_4(u,v)\in\mathbb{R}$, be any smooth function that is not identically zero and satisfies the differential equation \eqref{eqn:3.1ffffff2}. 
Let $\Omega_0=\{(u,v,t)\in \mathcal{D}\times \mathbb{R}|\ t=0\}$, $\Omega=\{(u,v,t)\in \mathcal{D}\times \mathbb{R}\}-\Omega_0\subseteq \mathbb{R}^3$, we define a metric $g$ on $\Omega$ by
\eqref{eqn:3.1ffffff3},
and a $(1,1)$-tensor field $A$ on $T\Omega$ by
\eqref{eqn:3.1ffffff4}.
Let $\nu$ be a vector bundle over  $\Omega$ of rank $1$ with metric $\tilde{g}$, $\nabla ^{\perp} $ a connection on $\nu$ compatible with the metric $\tilde{g}$, and $N$ a unit vector field in $\nu$.   We define a $(1,1)$-tensor field $\tilde{P}$ on $T\Omega\oplus \nu$ by
\eqref{eqn:3.1ffffff5}.
Then the integrability conditions are satisfied on $\Omega$ and hence on any simply
connected open subset $\Omega_1$ of $\Omega$, up to an isometry of  $\mathbb{H}^2\times \mathbb{H}^2$, there is a unique isometric immersion $\Psi$ from $(\Omega_1,~g)$ into $\mathbb{H}^2\times \mathbb{H}^2$ such that $A$ is the shape operator of  $\Omega_1$ into $\mathbb{H}^2\times \mathbb{H}^2$, $\nu$ is isometric to the normal bundle of $\Psi(\Omega_1)$ in
$\mathbb{H}^2\times \mathbb{H}^2$ by an isomorphism $\tilde{\Psi}:\nu\to T^{\perp}\Psi(\Omega_1)$,
and for any $Y\in T\Omega_1$, we have
\begin{equation}
P(\Psi_{*}Y)=\Psi_{*}((\tilde{P}Y)^{\top})+\tilde{\Psi}((\tilde{P}Y)^{\perp}),~
P(\tilde{\Psi}N)=\Psi_*(\partial_t),
\end{equation}
where $P$ is the product structure of $\mathbb{H}^2\times \mathbb{H}^2$, $(\tilde{P}Y)^{\top}$ and $(\tilde{P}Y)^{\perp}$ denote the projections of  $\tilde{P}Y$ onto $T\Omega_1$ and $\nu$, respectively.
Moreover, $\Psi(\Omega_1)$ has constant sectional curvature $-\frac{1}{2}$ and product angle function $C=0$.
\end{theorem}
\begin{proof}
For {\bf Case 6} of Lemma \ref{lemma:4.1aaa1}, we have
$$
\begin{aligned}
&b_1^*=\frac{1}{\sqrt{2}}e^{-\sqrt{2}(t+l_2)}\cosh(h_4),\ \
b_2^*=\frac{1}{\sqrt{2}}e^{-\sqrt{2}(t+l_2)}\sinh(h_4),\\
&b_1^*-b_2^*>0,\ \ b_2^*+b_1^*>0,\ \ 1-2b_1^{*2}+2b_2^{*2}\neq0,
\end{aligned}
$$
where the function $t$ satisfies $E_3t=1$ and the functions $l_2, h_4:M\rightarrow \mathbb{R}$
satisfy $E_3l_2=E_3h_4=0$.

Now, we define $\rho_1=\frac{1}{\sqrt{b_1^*-b_2^*}}$ and $\rho_2=\frac{1}{\sqrt{b_2^*+b_1^*}}$.
Let $\tilde{E}_1^*=\rho_1E_1^*$, $\tilde{E}_2^*=\rho_2E_2^*$ and $\tilde{E}_3^*=E_3^*$.
Then, it can be checked that
\begin{equation}\label{eqn:Lietilde6}
[\tilde{E}_1^*,\tilde{E}_3^*]=0,\ \ [\tilde{E}_2^*,\tilde{E}_3^*]=0,\ \
[\tilde{E}_1^*,\tilde{E}_2^*]=-8\frac{b_2^*}{\sqrt{b_1^{*2}-b_2^{*2}}}\tilde{E}_3^*.
\end{equation}

We still define the differential system \eqref{eqn:diff}. By  \eqref{eqn:Lietilde6},
we know that there are functions $\{u,v\}$ satisfying \eqref{eqn:diff} on $M$.
Now, we choose 
$f=-\frac{1}{2\sqrt{2}}\ln\Big(2(b_1^{*2}-b_2^{*2})\Big)=t+l_2$. In the following,
we still use $t$ to denote $f$.

We define the map
$$
\Phi: M \longrightarrow \mathbb{R}^{3},
\quad p \mapsto(u(p),v(p),t(p)).
$$
Then $\Phi_*$ is given by
\begin{equation}\label{eqn:dPhi6}
\Phi_*\left(
        \begin{array}{c}
          \tilde{E}_1^* \\
          \tilde{E}_2^* \\
          \tilde{E}_3^* \\
        \end{array}
      \right)
=\left(\begin{array}{ccc}
1 & 0 & \tilde{d}_1:=\tilde{E}_1^*(t) \\
0 & 1 & \tilde{d}_2:=\tilde{E}_2^*(t) \\
0 & 0 & 1
\end{array}\right)
\left(
\begin{array}{c}
	\partial_u \\
	\partial_v \\
	\partial_t \\
\end{array}
\right),
\end{equation}
which implies that $\{u,v,t\}$ are
local coordinates on $M$. It follows from $E_3h_4=0$ that the function $h_4$ depends only on variables $\{u,v\}$. Then for {\bf Case 6} of Lemma \ref{lemma:4.1aaa1},
we have
\begin{equation}\label{eqn:bijth6}
\begin{aligned}
b_1&=\tfrac{1}{\sqrt{2}}{\rm csch}(\sqrt{2}t)\Big(\cosh(h_4)-\cosh(\sqrt{2}t)\Big),\\
b_2&=\tfrac{1}{\sqrt{2}}{\rm csch}(\sqrt{2}t)\sinh(h_4),\\
b_4&=\tfrac{1}{\sqrt{2}}{\rm csch}(\sqrt{2}t)\Big(\cosh(h_4)+\cosh(\sqrt{2}t)\Big),\\
b_1^*&=\frac{1}{\sqrt{2}}e^{-\sqrt{2}t}\cosh(h_4),\ \
b_2^*=\frac{1}{\sqrt{2}}e^{-\sqrt{2}t}\sinh(h_4).
\end{aligned}
\end{equation}

By using \eqref{eqn:Eibar}, \eqref{eqn:dPhi6}, $\tilde{E}_1^*=\rho_1E_1^*$
and $\tilde{E}_2^*=\rho_2E_2^*$, we have 
\begin{equation}\label{eqn:transframe6}
	\begin{pmatrix}E_1\\E_2\end{pmatrix}
	=B^{-1}\begin{pmatrix}\rho_1^{-1}&0\\0&\rho_2^{-1}\end{pmatrix}
	\begin{pmatrix}\partial_u+\tilde d_1\partial_t\\
		\partial_v+\tilde d_2\partial_t\end{pmatrix},\qquad E_3=\partial_t,
\end{equation}
where
$$
B=\begin{pmatrix}
	1+\sqrt2b_1^*-\sqrt2b_2^*&-1+\sqrt2b_1^*-\sqrt2b_2^*\\
	1+\sqrt2b_1^*+\sqrt2b_2^*&1-\sqrt2b_1^*-\sqrt2b_2^*
\end{pmatrix},\quad
\det B=2(1-2(b_1^*)^2+2(b_2^*)^2)\ne0.
$$

Now, by using \eqref{eqn:4.29}, \eqref{eqn:4.30} and \eqref{eqn:transframe6}, we can have
$$
\tilde{d}_1=\frac{(h_4)_u}{\sqrt{2}},\ \ \tilde{d}_2=\frac{-(h_4)_v}{\sqrt{2}}.
$$
Then we can check the compatibility conditions of $b_1,b_2,b_4$ with respect to the Lie brackets $[E_i,E_j]$, $1\leq i<j\leq3$,
and obtain that
$$
\frac{\partial^2 h_4}{\partial u \partial v}
=4\sqrt{2}\sinh(h_4).
$$
In fact, the  compatibility conditions of  $b_1,b_2,b_4$  with respect to the Lie brackets $[E_1,E_3]$ and $[E_2,E_3]$ are automatically satisfied and the  compatibility conditions of  $b_1,b_2,b_4$  with respect to the Lie bracket $[E_1,E_2]$  imply the above differential equation.

Conversely, we assume that $h_4$ is not identically zero and satisfies \eqref{eqn:3.1ffffff2}, the metric $g$ on $\Omega$ is defined by 
\begin{equation}\label{eqn:3.1ffffff3}
	\begin{aligned}
		g=&\frac{1}{2}(8\sqrt{2}\cosh(\sqrt{2}t+h_4)+(h_4)_u^2)du^2
		+\frac{1}{2}(8\sqrt{2}\cosh(\sqrt{2}t-h_4)+(h_4)_v^2)dv^2+dt^2\\
		&+(8\sqrt{2}\cosh(h_4)-(h_4)_u(h_4)_v)dudv-(\sqrt{2}(h_4)_u)dudt+(\sqrt{2}(h_4)_v)dvdt,
	\end{aligned}
\end{equation}
the $(1,1)$-tensor $A$ on $T\Omega$ is defined by
\begin{equation}\label{eqn:3.1ffffff4}
	\begin{aligned}
		A\partial_u=&\frac{{\rm csch}(\sqrt{2}t)\cosh(h_4)}{\sqrt{2}}\partial_u
		-\frac{\sinh(h_4)+\cosh(h_4)\coth(\sqrt{2}t)}{\sqrt{2}}\partial_v\\
		&+\frac{1}{2}{\rm csch}(\sqrt{2}t)(\cosh(\sqrt{2}t+h_4)(h_4)_v+\cosh(h_4)(h_4)_u)\partial_t,\\
		A\partial_v=&\frac{\sinh(h_4)-\cosh(h_4)\coth(\sqrt{2}t)}{\sqrt{2}}\partial_u+
		\frac{{\rm csch}(\sqrt{2}t)\cosh(h_4)}{\sqrt{2}}\partial_v\\
		&-\frac{1}{2}{\rm csch}(\sqrt{2}t)(\cosh(\sqrt{2}t-h_4)(h_4)_u+\cosh(h_4)(h_4)_v)\partial_t,\\
		A\partial_t=&0,
	\end{aligned}
\end{equation}
and the $(1,1)$-tensor $\tilde{P}$ on $T\Omega\oplus \nu$ is defined by
\begin{equation}\label{eqn:3.1ffffff5}
	\begin{aligned}
		\tilde{P}\partial_u=&-\sinh(h_4){\rm csch}(\sqrt{2}t)\partial_u+(\cosh(h_4)+\sinh(h_4)\coth(\sqrt{2}t))\partial_v\\
		&-\frac{{\rm csch}(\sqrt{2}t)(\sinh(\sqrt{2}t+h_4)(h_4)_v+\sinh(h_4)(h_4)_u)}{\sqrt{2}}\partial_t-\frac{(h_4)_u}{\sqrt{2}}N,\\
		\tilde{P}\partial_v=&(\cosh(h_4)-\sinh(h_4)\coth(\sqrt{2}t))\partial_u+\sinh(h_4){\rm csch}(\sqrt{2}t)\partial_v\\
		&+\frac{{\rm csch}(\sqrt{2}t)(\sinh(\sqrt{2}t-h_4)(h_4)_u-\sinh(h_4)(h_4)_v)}{\sqrt{2}}\partial_t+\frac{(h_4)_v}{\sqrt{2}}N,\\
		\tilde{P}\partial_t=&N,\\
		\tilde{P}N=&\partial_t.\\	
	\end{aligned}
\end{equation}
Then it can be  checked  that all the integrability conditions are satisfied.
Therefore, the conclusions can be obtained by applying  the existence and uniqueness theorems in  \cite{K,L-T-V} directly.
\end{proof}

\noindent{\bf Completion of the proof of Theorem \ref{thm:1.1}}.

Since $M$  has  constant sectional curvature $\kappa$, from Proposition \ref{prop:4.5}, we know that $M$ has constant sectional curvature $\kappa=-\tfrac{1}{2}$. According to the proof of Proposition \ref{prop:4.5}, we consider two situations depending on whether $C$ is $0$ or not on $M$.
When $C>0$ on $M$, then by Theorem \ref{thm:4.1abc}, 
$M$ is an open part of Example \ref{exam:4.1ss1ss11ab1}.
When $C\equiv0$ on $M$, then by Lemma \ref{lemma:4.1aaa1}, there are six cases.
For {\bf Case 1} and {\bf Case 2} of Lemma \ref{lemma:4.1aaa1}, by Proposition \ref{prop:4.1a1},
$M$ is an open part of $M_{k_1,\frac{1}{k_1}}^{1/2}$ given in Example \ref{exam:3.4fg1} for some constant $0<k_1\leq1$. For {\bf Case 3}, {\bf Case 4}, {\bf Case 5} and {\bf Case 6} of Lemma \ref{lemma:4.1aaa1}, we know that $M$ is an open part of one of the hypersurfaces described by Theorems \ref{thm:5.10a1}--\ref{thm:5.13a1}, respectively. We complete the proof
of Theorem \ref{thm:1.1}.

\section{Proofs of Theorem \ref{thm:1.2} and Theorem \ref{thm:1.3}}\label{sect:6}
\subsection{Proof of Theorem \ref{thm:1.2}}\label{sect:6.1}~

In this subsection, we assume that $M$ has constant product angle function and
constant mean curvature.
When $C=\pm1$ on $M$, by Lemma \ref{lem:2.2}, we know that $M$ is locally the
product of a curve $\Gamma$ in $\mathbb{H}^2$ and an open
subset of $\mathbb{H}^2$. Furthermore, if $M$ has constant mean curvature,
then $M$ is an open part of $M_\Gamma$, where $\Gamma$ is a curve of $\mathbb{H}^2$ with constant curvature.

When $C\neq\pm1$, we choose the frame $\{E_1,E_2,E_3\}$ defined by \eqref{eqn:2.5}.
Then, by $\nabla C=-2 AX$, we have $b_3=b_5=b_6=0$. It follows from \eqref{eqn:2.8a1212} that
$\rho=-2+2(b_1b_4-b_2^2)$. By \eqref{eqn:2.6a11} and \eqref{eqn:2.5aaa}, we have
\begin{equation}\label{eqn:3.2a1}
C\|A\|^2={\rm tr}(TA^2)=3H(b_1-b_4), 
\end{equation}
\begin{equation}\label{eqn:3.2a2}
9H^2-\|A\|^2=2(b_1b_4-b_2^2).
\end{equation}

By using \eqref{eqn:2.26}--\eqref{eqn:2.28}, \eqref{eqn:3.2a1}, \eqref{eqn:3.2a2},
$\|A\|^2=b_1^2+2b_2^2+b_4^2$ and $X=\sqrt{1-C^2}E_3$, we can check that
\begin{equation}\label{eqn:3.2aas1}
3H(X\|A\|^2)=C[-(1-C^2)\|A\|^2+(9H^2-\|A\|^2)^2].
\end{equation}

If $H\neq0$, by \eqref{eqn:3.2aas1}, we have
$$
X(\|A\|^2)=\frac{C[-(1-C^2)\|A\|^2+(9H^2-\|A\|^2)^2]}{3H}.
$$

Taking the derivative of $(b_1-b_4)$ with respect to $X$, with the use of \eqref{eqn:2.26}--\eqref{eqn:2.28}, \eqref{eqn:3.2a1} and \eqref{eqn:3.2a2}, we can check that
$$
X(b_1-b_4)=-1+C^2-C^2\|A\|^2+(9H^2-\|A\|^2)+\frac{C^2\|A\|^4}{9H^2}.
$$

Now, taking the derivative of \eqref{eqn:3.2a1} with respect to $X$, with the use of the expressions of $X(\|A\|^2)$ and $X(b_1-b_4)$, it follows that
$$
(C^2-9H^2)\|A\|^2=9H^2(1-9H^2).
$$
If $9H^2=C^2$, then $0=9H^2(1-9H^2)=C^2(1-C^2)$, so $C=0$ and $H=0$, which is a contradiction.
Then we have $\|A\|^2=\tfrac{9H^2(1-9H^2)}{C^2-9H^2}$,
which means that $\|A\|$ is constant. It implies that $M$ has constant principal curvatures.
Then, according to Theorem 1.1 of \cite{GMY},
$M$ is either an open part of $M_\tau$ for some $\tau<-1$, or
$M$ is an open part of $M_{1,1}^c$ for some $c\in(0,\frac{1}{2})\cup(\frac{1}{2},1)$, or
$M$ is an open part of $M_{1,-1}^c$ for some $c\in(0,1)$.

If $H=0$, then by \eqref{eqn:3.2a1}, we have $C=0$ or $\|A\|=0$. Note that,
according to Corollary 1.3 of \cite{Me-T}, we know that the hypersurfaces with $\|A\|=0$ correspond to $C=\pm1$.
Thus, $M$ is a minimal hypersurface with $C=0$, which has been classified by Theorem \ref{thm:3.2aa}.
We complete the proof of Theorem \ref{thm:1.2}.

\subsection{Proof of Theorem \ref{thm:1.3}}\label{sect:6.2}~

In this subsection, we assume that $M$ has constant product angle function and
constant scalar curvature.
When $C=\pm1$ on $M$, by Lemma \ref{lem:2.2}, we know that $M$ is locally the product of a curve $\Gamma$ in $\mathbb{H}^2$ and an open
subset of $\mathbb{H}^2$, which corresponds to the example $M_\Gamma$ given in
Example \ref{exam:3.1f1f2}. Conversely, for every regular curve $\Gamma$, 
the scalar curvature of $M_\Gamma$ equals $-2$. Hence, in this case,  $M$ is an open part of $M_\Gamma$ given in Example \ref{exam:3.1f1f2}.
(Remark: due to the same reason here, for any hypersurface $M$  with constant product angle function and
constant scalar curvature in $\mathbb{S}^2\times\mathbb{S}^2$, if $C=\pm1$ on $M$, then $M$ is locally the product of a curve $\Gamma$ in $\mathbb{S}^2$ and an open
subset of $\mathbb{S}^2$. Hence, the correct statement for Case (2)(a) of Theorem 1.2  in \cite{LVWY} should be: $C^2 = 1$ and $M$ is an open part of $\Gamma\times \mathbb{S}^2$, where
$\Gamma$ is a curve in $\mathbb{S}^2$.)

When $C\neq\pm1$, we choose the frame $\{E_1,E_2,E_3\}$ defined by \eqref{eqn:2.5}. Since the scalar curvature
$\rho=-2+2(b_1b_4-b_2^2)$ is constant, by using \eqref{eqn:2.26}--\eqref{eqn:2.28} and $X=\sqrt{1-C^2}E_3$, it follows from $X\rho=0$ that
\begin{equation}\label{eqn:4.1a1}
(1-C)b_1(\rho+C+3)=(1+C)b_4(\rho-C+3).
\end{equation}
Taking the derivative of equation \eqref{eqn:4.1a1} with respect to $X$, using \eqref{eqn:2.26}--\eqref{eqn:2.28}, \eqref{eqn:4.1a1} and $\rho=-2+2(b_1b_4-b_2^2)$, we can obtain
$$
(\rho+1)(\rho+3)+C^2=0.
$$
Hence there are two possible values of the scalar curvature: $\rho=-2\pm\sqrt{1-C^2}$.

{\bf Case i}: $\rho=-2+\sqrt{1-C^2}$.

Putting this value in \eqref{eqn:4.1a1}, it follows that
\begin{equation}\label{eqn:4.1a121212}
b_1\sqrt{1-C}=b_4\sqrt{1+C}.
\end{equation}
Then, by $b_1b_4-b_2^2=\frac{1}{2}(\rho+2)=\frac{\sqrt{1-C^2}}{2}$, we have $b_1\neq0$ and
$\sqrt{\frac{1-C}{1+C}}b_1^2-b_2^2=\frac{\sqrt{1-C^2}}{2}$.

Taking the derivative of $\sqrt{\frac{1-C}{1+C}}b_1^2-b_2^2=\frac{\sqrt{1-C^2}}{2}$ with respect to $E_1$ and $E_2$, with the use of \eqref{eqn:2.29}, \eqref{eqn:2.30} and \eqref{eqn:4.1a121212}, we get
$$
b_2E_1b_2-b_1E_2b_2=0,\ \ \sqrt{\frac{1-C}{1+C}}b_1E_1b_2-b_2E_2b_2=0.
$$
Now, by $\sqrt{\frac{1-C}{1+C}}b_1^2-b_2^2=\frac{\sqrt{1-C^2}}{2}\neq0$, the solution to the above two equations is $E_1b_2=E_2b_2=0$.

From the fact that $[E_1,E_2]=\frac{-2b_2}{\sqrt{1-C^2}}E_3$, with the use of \eqref{eqn:2.27} and
\eqref{eqn:4.1a121212}, we have
$$
0=[E_1,E_2]b_2=\frac{-2b_2}{\sqrt{1-C^2}}E_3b_2
=\frac{-2b_2^2b_1}{\sqrt{1-C^2}}\frac{\sqrt{1-C}-\sqrt{1+C}}{\sqrt{1+C}}.
$$
Then we have $b_2=0$ or $C=0$.

{\bf Case i-1}: If $b_2=0$, then $b_1^2=\frac{1+C}{2}$ and $b_4^2=\frac{1-C}{2}$ and
$M$ has constant principal curvatures. Thus, from Theorem 1.1 of \cite{GMY}, we know that $M$ is an open part of $M_{1,-1}^c$ for $c\in(0,1)$.

{\bf Case i-2}: If $C=0$, then $b_1=b_4$. By \eqref{eqn:2.27} and $E_1b_2=E_2b_2=0$, we have that
$b_2$ is constant. It follows that $b_1$ and $b_4$ are also constant.
So, it implies that $M$ has constant principal curvatures.
Thus, from Theorem 1.1 of \cite{GMY}, we know that $M$ is either an open part of $M_{1,-1}^{1/2}$,
or an open part of $M_\tau$ for some $\tau<-1$.

{\bf Case ii}: $\rho=-2-\sqrt{1-C^2}$.

Putting this value in \eqref{eqn:4.1a1}, it follows that
$$
b_1\sqrt{1-C}(\sqrt{1+C}-\sqrt{1-C})=b_4\sqrt{1+C}(\sqrt{1-C}-\sqrt{1+C}).
$$
We divide the discussion into two subcases depending on the value of $C$.

{\bf Case ii-1}: If $C\neq0$, then we have
\begin{equation}\label{eqn:4.2a121212}
b_1\sqrt{1-C}=-b_4\sqrt{1+C}.
\end{equation}
Then by $b_1b_4-b_2^2=\frac{1}{2}(\rho+2)=-\frac{\sqrt{1-C^2}}{2}$, it holds  $\sqrt{\frac{1-C}{1+C}}b_1^2+b_2^2=\frac{\sqrt{1-C^2}}{2}$.

Taking the derivative of $\sqrt{\frac{1-C}{1+C}}b_1^2+b_2^2=\frac{\sqrt{1-C^2}}{2}$ with respect to $E_1$ and $E_2$, with the use of \eqref{eqn:2.29}, \eqref{eqn:2.30} and \eqref{eqn:4.2a121212}, we get
$$
-b_2E_1b_2+b_1E_2b_2=0,\ \ \sqrt{\frac{1-C}{1+C}}b_1E_1b_2+b_2E_2b_2=0.
$$
Now, by $\sqrt{\frac{1-C}{1+C}}b_1^2+b_2^2=\frac{\sqrt{1-C^2}}{2}\neq0$,
the solution to the above two equations is $E_1b_2=E_2b_2=0$.

From the fact that $[E_1,E_2]=\frac{-2b_2}{\sqrt{1-C^2}}E_3$, we have
$0=[E_1,E_2]b_2=\frac{-2b_2}{\sqrt{1-C^2}}E_3b_2$. Then by \eqref{eqn:2.27} and
\eqref{eqn:4.2a121212}, we get
$$0=b_2^2b_1(\sqrt{1-C}+\sqrt{1+C}).$$

If $b_2$ were nonzero at a point, it would be nonzero on a neighborhood.
There $b_1=0$, hence $b_4=0$ by \eqref{eqn:4.2a121212}.
But \eqref{eqn:2.26} would give
$0=-\sqrt{1-C^2}/2-b_2^2\sqrt{(1+C)/(1-C)}<0$.
Therefore $b_2$ vanishes everywhere, which also means that $M$ has constant principal curvatures. 
By Theorem 1.1 of \cite{GMY}, we know that $M$ is an open part of $M_{1,1}^c$ for $c\in(0,\frac{1}{2})\cup(\frac{1}{2},1)$. 

{\bf Case ii-2}: If $C=0$, then $\rho=-3$ and $b_1b_4-b_2^2=-\frac{1}{2}$.

Now, by \eqref{eqn:2.8a1212}, the Ricci tensor is expressed by
$$
{\rm Ric}(E_1)=-E_1,\ \ {\rm Ric}(E_2)=-E_2,\ \ {\rm Ric}(E_3)=-E_3.
$$
Since ${\rm dim}~M=3$, $M$ has constant sectional curvature $-\tfrac{1}{2}$,
which has been classified by Theorem \ref{thm:1.1}. 
We complete the proof of Theorem \ref{thm:1.3}.

\section{Appendix}\label{sect:7}

In this section, we give a complete classification of the hypersurfaces in $\mathbb{S}^m\times \mathbb{S}^n$ ($m\geq3$, $n\geq2$), 
$\mathbb{H}^m\times \mathbb{H}^n$ ($m\geq3$, $n\geq2$), and $\mathbb{S}^m\times \mathbb{H}^n$ ($m,n\geq2$) with constant sectional curvature. 
For $\epsilon\in \{-1,1\}$, 
we denote by $\bar{M}^m(\epsilon)$ the unit $m$-sphere
$\mathbb S^{m}$ 
endowed with the canonical metric of constant sectional curvature $1$ when $\epsilon= 1$,
and the hyperbolic space 
$\mathbb{H}^{m}$ endowed with the canonical metric of constant sectional curvature $-1$ when $\epsilon= -1$. 
We endow $\bar{M}^m(\epsilon_1)\times \bar{M}^n(\epsilon_2)$ 
with the product metric $g$, where $(\epsilon_1,\epsilon_2)\in\{(1,1), (1,-1), (-1,-1)\}$.  
The product structure $P$ on $\bar{M}^m(\epsilon_1)\times \bar{M}^n(\epsilon_2)$
is defined by
$$
P(v_1,v_2)=(v_1,-v_2),~\hbox{for any}~(v_1,v_2)\in T(\bar{M}^m(\epsilon_1)\times \bar{M}^n(\epsilon_2)).
$$
Let $\bar{\nabla}$ be the Levi-Civita connection of the metric 
$g$ on $\bar{M}^m(\epsilon_1)\times \bar{M}^n(\epsilon_2)$. 
The curvature tensor $\bar{R}$ of $\bar{M}^m(\epsilon_1)\times \bar{M}^n(\epsilon_2)$ 
with the Riemannian product metric $g$ is given by
\begin{equation}\label{eqn:apxql}
\begin{aligned}
g(\bar{R}(U,Y)Z,W)=&\epsilon_1\Big(g(U_1,W_1)g(Y_1,Z_1)-g(U_1,Z_1)g(Y_1,W_1)\Big)\\
&+\epsilon_2\Big(g(U_2,W_2)g(Y_2,Z_2)-g(U_2,Z_2)g(Y_2,W_2)\Big), 
\end{aligned}
\end{equation}
where $U_1=\frac{PU+U}{2}$ and $U_2=\frac{U-PU}{2}$ for any $U\in T(\bar{M}^m(\epsilon_1)\times \bar{M}^n(\epsilon_2))$. 

Let $M$ be an orientable hypersurface of $\bar{M}^m(\epsilon_1)\times \bar{M}^n(\epsilon_2)$ with a unit normal vector field $N$ and the induced 
metric, still denoted by $g$. Then, with respect
to the product structure $P$, the smooth product angle function $C:M\rightarrow\mathbb R$
and the vector field $X$ tangent to $M$ are defined by
\begin{align*}
	C:=g(PN,N ),\ \ 	X:=PN-CN.
\end{align*}
It is obvious that $-1\le C\le1$ and $\|X\|^2:=g(X,X)=1-C^2$.
For any tangential vector field $Y$ of $M$, acting by the product structure $P$, we have the following decomposition
$$
PY=TY+\mu(Y)N, 
$$
where $TY$ and $\mu(Y)N$ are the tangential and normal parts of $PY$. Thus $T$ is a tensorial field of type $(1,1)$, and $\mu$ is a $1$-form over $M$. Moreover, $\mu(Y)=g(PY,N)$.

Let $\nabla$ be the Levi-Civita connection of the induced metric
$g$ on $M$. Then the Gauss and Weingarten
formulae are given, respectively, by:
\begin{equation}\label{eqn:APX2.5}
	\bar{\nabla}_YZ=\nabla_YZ+g(AY,Z)N, ~~
	\bar{\nabla}_YN=-AY,~~\hbox{for any}~ Y,Z\in TM,
\end{equation}
where $A$ is the shape operator of $M$.

Let $R$ be the Riemannian curvature tensor of $M$. Then, the Gauss
and Codazzi equations of $M$ are given respectively as follows:
\begin{equation}\label{eqn:APX2.6}
	\begin{aligned}
		g(R(U,Y)Z,W)=g(\bar{R}(U,Y)Z,W)+g(AY,Z)g(AU,W)-g(AU,Z)g(AY,W),
	\end{aligned}
\end{equation}
\begin{equation}\label{eqn:APX2.7}
	g((\nabla_YA)Z-(\nabla_ZA)Y,W)=g(\bar{R}(Y,Z)W,N),
\end{equation}
where $U, Y, Z, W\in TM$.

As in the proof of Lemma 2.1 in \cite{GMY}, we still have that 
the gradient of $C$ and the covariant derivative of $X$ are given by 
\begin{equation}\label{eqn:APX2.9}
	\nabla C=-2AX,~~ \nabla_YX=CAY-TAY, ~~\hbox{for any}~ Y\in TM.
\end{equation}

In the following, by using the so-called Tsinghua principle (see the description in Sect. \ref{sect:1}), one can establish a very useful identity on a hypersurface in any Riemannian manifold with constant sectional curvature.

\begin{lemma}[\cite{AD}]\label{thm:7.1}
Let $M$ be a hypersurface with constant sectional curvature  in a Riemannian manifold $(\bar{M},g)$. Then, it holds
\begin{equation}\label{eqn:4.1q1}
\begin{aligned}
\mathop\mathfrak{S}\limits_{W,U,Y}g(\bar{R}(U,Y)Z,AW)=0, \ \ \forall \ U,Y,Z,W\in TM, 
\end{aligned}
\end{equation}
where the symbol $\mathfrak{S}$ stands for the cyclic summation, $\bar{R}$ denotes the curvature tensor of $\bar{M}$, and $A$ denotes the shape operator of $M$.
\end{lemma}

Now, we can state our result as follows:

\begin{theorem}\label{thm:1.4}
Let $M$ be a hypersurface of $M^m(\epsilon_1)\times M^n(\epsilon_2)$ with constant sectional curvature $\kappa$,  where $m,n\geq 2$ and $(\epsilon_1,\epsilon_2)\in\{(1,1),(1,-1),(-1,-1)\}$. For $\mathbb{S}^m\times \mathbb{S}^n$ and $\mathbb{H}^m\times \mathbb{H}^n$, we further assume $m\geq 3$ and $n\geq 2$. Then $\kappa=-\frac{1}{2}$, $\epsilon_1=\epsilon_2=-1$ and $M$ is either
\begin{enumerate}
\item[(1)]
locally given by the immersion $\Phi_1$, see \eqref{eqn:ap3.1ss};  or 
\item[(2)]
locally given by the immersion $\Phi_2$, see \eqref{eqn:ap3.1sss}; or
\item[(3)]
locally given by the immersion $\Phi_3$, see \eqref{eqn:ap3.1ssss}. 
\end{enumerate}
\end{theorem}
\begin{proof}

Assume that the hypersurface $M$ has constant sectional curvature $\kappa$. We will discuss two cases depending on the value of the product angle function $C$ on $M$.

{\bf Case I}. $C\neq\pm1$.

In this case,  it holds that $PN=CN+X$. 
We can take a local orthonormal frame field $\{E_\alpha\}_{\alpha=1}^{m+n-1}$ 
such that $E_i\in P(1)$ for $1\leq i\leq m-1$, $E_p\in P(-1)$ for $m\leq p\leq m+n-2$ and  $E_{m+n-1}=\tfrac{X}{\sqrt{1-C^2}}$, where $P(1)$ (resp. $P(-1)$) is the eigenspace of $P$ corresponding to $1$ (resp. $-1$). Then it holds that 
\begin{equation}\label{eqn:4.1}
	PN=CN+\sqrt{1-C^2}E_{m+n-1},\ \ PE_{m+n-1}=\sqrt{1-C^2}N-CE_{m+n-1}.
\end{equation}
Let us write $AE_{\alpha}=\sum_{\beta=1}^{m+n-1}a_{\alpha\beta}E_\beta$, $1\leq \alpha\leq m+n-1$, where $a_{\alpha\beta}=a_{\beta\alpha}$ for
$1\leq \alpha,\beta\leq m+n-1$.

Now, taking in \eqref{eqn:4.1q1}, respectively, $1\leq i, j\leq m-1$ and $m\leq p,q\leq m+n-2$,
$$
(U,Y,Z,W)=(E_p,E_i,E_j,E_j), (E_p,E_i,E_q,E_q), (E_p,E_i,E_{m+n-1},E_{m+n-1}),
$$
with the use of \eqref{eqn:apxql}, we obtain 
\begin{equation}\label{eqn:ap1}
\epsilon_1(1-\delta_{ij})a_{ip}=0,
\end{equation}
\begin{equation}\label{eqn:ap2}
	\epsilon_2(1-\delta_{pq})a_{ip}=0,
\end{equation}
\begin{equation}\label{eqn:ap3}
a_{ip}\big(( C-1)\epsilon_1 +(1+C)\epsilon_2\big)=0.
\end{equation}

Taking in \eqref{eqn:4.1q1}, respectively, $1\leq i, j\leq m-1$ and $m\leq p,q\leq m+n-2$,
$$
(U,Y,Z,W)=(E_{m+n-1},E_i,E_p,E_p),(E_{m+n-1},E_p,E_i,E_i),
$$
with the use of \eqref{eqn:apxql}, we obtain  
\begin{equation}\label{eqn:ap4}
	a_{i,m+n-1}(1+C)\epsilon_2=0,
\end{equation}
\begin{equation}\label{eqn:ap5}
	a_{p,m+n-1}(C-1)\epsilon_1=0.
\end{equation}

For any $(\epsilon_1,\epsilon_2)\in \{(1,1),(1,-1),(-1,-1)\}$, $m,n\geq2$ and $m+n\geq5$, by 
\eqref{eqn:ap1} and \eqref{eqn:ap2}, it follows that $a_{ip}=0$ for $1\leq i\leq m-1$ and $m\leq p\leq m+n-2$. 
For $(\epsilon_1,\epsilon_2)=(1,-1)$ and $m=n=2$, by \eqref{eqn:ap3}, we still have $a_{12}=0$. 
For any $(\epsilon_1,\epsilon_2)\in \{(1,1),(1,-1),(-1,-1)\}$ and $m,n\geq2$, by using $C\neq\pm1$,  \eqref{eqn:ap4} and \eqref{eqn:ap5}, we have $a_{i,m+n-1}=a_{p,m+n-1}=0$ for $1\leq i\leq m-1$ and $m\leq p\leq m+n-2$. Thus, on hypersurface $M$, it always holds that the subdistributions 
${\rm Span}\,\{E_i\}_{i=1}^{m-1}$, ${\rm Span}\,\{E_p\}_{p=m}^{m+n-2}$ and ${\rm Span}\,\{E_{m+n-1}\}$ 
are $A$-invariant, respectively. This allows us to further choose the frame $\{E_i\}_{i=1}^{m+n-1}$ so that 
$$
AE_i=\lambda_iE_i,  AE_p=\lambda_pE_p, AE_{m+n-1}=\lambda_{m+n-1}E_{m+n-1}, 
\ \ \forall \ 1\leq i\leq m-1,\ m\leq p\leq m+n-2.
$$

Now, we calculate the sectional curvatures $K(E_i,E_j)$, $K(E_i,E_p)$, $K(E_p,E_q)$, $K(E_i,E_{m+n-1})$ and 
$K(E_p,E_{m+n-1})$. By using the Gauss equation \eqref{eqn:APX2.6} and the fact that $M$ has constant sectional curvature $\kappa$, 
we obtain 
\begin{gather}
	K(E_i,E_j)=\epsilon_1+\lambda_i\lambda_j=\kappa,\ \ i\neq j,\label{eqn:4.7s}\\
	K(E_i,E_p)=\lambda_i\lambda_{p}=\kappa,\label{eqn:4.8s}\\
	K(E_p,E_q)=\epsilon_2+\lambda_p\lambda_q=\kappa,\ \ p\neq q,\label{eqn:4.7adss}\\
	K(E_i,E_{m+n-1})=\tfrac{1}{2}(1-C)\epsilon_1+\lambda_i\lambda_{m+n-1}=\kappa,\label{eqn:4.9s}\\
	K(E_p,E_{m+n-1})=\tfrac{1}{2}(1+C)\epsilon_2+\lambda_p\lambda_{m+n-1}=\kappa,\label{eqn:4.9as}
\end{gather}
where \eqref{eqn:4.7s} only appears for $m\geq3$, and 
\eqref{eqn:4.7adss} only appears for $n\geq3$.

In the following, according to the ambient space, we discuss three subcases. 

{\bf Case I-i}. $\mathbb{S}^m\times\mathbb{S}^n$ for $m\geq3$ and $n\geq2$. 

In this case, $\epsilon_1=\epsilon_2=1$. If $\kappa=0$, then by \eqref{eqn:4.8s}, it holds 
$\lambda_i=0$ or $\lambda_p=0$ which contradicts  \eqref{eqn:4.7s}, \eqref{eqn:4.9as} 
and $C\neq\pm1$. 
So $\kappa\neq0$ and $\lambda_i\lambda_p\neq0$. By \eqref{eqn:4.8s}, we have 
$\lambda_i=\lambda_j$ for $1\leq i\neq j\leq m-1$.
Now, by \eqref{eqn:4.7s} and
\eqref{eqn:4.8s}, we have $\lambda_p=\tfrac{1}{\lambda_i}+\lambda_i$. By \eqref{eqn:4.7s} 
and \eqref{eqn:4.9s},
we get $\lambda_{m+n-1}=\tfrac{1+2\lambda_i^2+C}{2\lambda_i}$. By substituting $\lambda_p=\tfrac{1}{\lambda_i}+\lambda_i$ and  $\lambda_{m+n-1}=\tfrac{1+2\lambda_i^2+C}{2\lambda_i}$ into \eqref{eqn:4.9as}, with the use of  $\kappa=1+\lambda_i^2$, 
it deduces that $(1+C)(1+2\lambda_i^2)=0$, which contradicts $C\neq\pm1$. 
Thus, {\bf Case I-i} does not occur. 

\vskip 2mm

{\bf Case I-ii}. $\mathbb{H}^m\times\mathbb{H}^n$ for $m\geq3$ and $n\geq2$. 

In this case, $\epsilon_1=\epsilon_2=-1$. 
If $\kappa=0$, then by \eqref{eqn:4.8s}, it holds 
$\lambda_i=0$ or $\lambda_p=0$ which contradicts  \eqref{eqn:4.7s}, \eqref{eqn:4.9as} 
and $C\neq\pm1$. 
So $\kappa\neq0$ and $\lambda_i\lambda_p\neq0$. 
Then by  \eqref{eqn:4.8s}, we have $\lambda_i=\lambda_j$ and $\lambda_p=\lambda_q$ for  $1\leq i,j\leq m-1$ and $m\leq p,q\leq m+n-2$. By \eqref{eqn:4.7s} and \eqref{eqn:4.8s}, we have $\lambda_{p}=\lambda_{i}-\frac{1}{\lambda_{i}}$. 
By \eqref{eqn:4.7s} and \eqref{eqn:4.9s}, we get  $\lambda_{m+n-1}=-\frac{1+C-2\lambda_i^2}{2\lambda_i}$. By substituting 
$\lambda_{p}=\lambda_{i}-\frac{1}{\lambda_{i}}$ and $\lambda_{m+n-1}=-\frac{1+C-2\lambda_i^2}{2\lambda_i}$ into 
\eqref{eqn:4.9as}, we deduce that  
$(1+C)(1-2\lambda_i^2)=0$. Since $C\neq\pm1$, we get $\lambda_i=\pm\frac{1}{\sqrt{2}}$ 
and $\kappa=-\frac{1}{2}$. 
Up to a sign of the normal vector field $N$, we can assume that 
$\lambda_i=\frac{1}{\sqrt{2}}$, $\lambda_p=\frac{-1}{\sqrt{2}}$ and 
$\lambda_{m+n-1}=\frac{-C}{\sqrt{2}}$.  By \eqref{eqn:APX2.9}, 
it holds $E_{m+n-1}C=C\sqrt{2(1-C^2)}$. 
Since the distributions ${\rm Span}\,\{E_i\}_{i=1}^{m-1}$, ${\rm Span}\,\{E_p\}_{p=m}^{m+n-2}$ and ${\rm Span}\,\{E_{m+n-1}\}$ are $A$-invariant, one can check that $M$ satisfies the condition of $AT=TA$. Note that, 
hypersurfaces satisfying $AT=TA$ are called class $\mathcal{A}$ hypersurfaces 
in \cite{CT}. 

In the following, we begin to find the special examples of class $\mathcal{A}$ in $\mathbb{H}^m\times\mathbb{H}^n$ with 
constant sectional curvature. 
Let $f_1:M_1\rightarrow \mathbb{H}^m$ and $f_2:M_2\rightarrow \mathbb{H}^n$ 
be two hypersurfaces with unit normal vector fields $N^{f_1}$ and $N^{f_2}$, respectively. 
By Theorem 3.1 of \cite{CT}, we know that $M$ is locally given by the immersion: 
\begin{equation}\label{eqn:APXap3.1}
\begin{aligned}
\Phi:I\times M_1\times M_2&\longrightarrow \mathbb{H}^m\times \mathbb{H}^n,\\
(\tilde{t},x_1,x_2)&\longmapsto (p(\tilde{t},x_1),q(\tilde{t},x_2)),
\end{aligned}
\end{equation}
where
$$
\left\{\begin{aligned}
	p(\tilde{t},x_1)=\cosh(u_1(\tilde{t}))f_1(x_1)+\sinh(u_1(\tilde{t}))N^{f_1}(x_1),\\
	q(\tilde{t},x_2)=\cosh(u_2(\tilde{t}))f_2(x_2)+\sinh(u_2(\tilde{t}))N^{f_2}(x_2),  
\end{aligned}\right.
$$
where $u_1(\tilde{t})$, $u_2(\tilde{t})$ are two smooth functions satisfying
$(u'_{1}(\tilde{t}))^2+(u'_{2}(\tilde{t}))^2=1$ and $u'_{1}(\tilde{t})<0$, $u'_{2}(\tilde{t})<0$.
Moreover, it holds that $(\frac{\partial p}{\partial \tilde{t}},\frac{\partial q}{\partial \tilde{t}})=E_{m+n-1}$. By 
$PN=CN+\sqrt{1-C^2}E_{m+n-1}$, we have $N=(\sqrt{\frac{1+C}{1-C}}\frac{\partial p}{\partial \tilde{t}},-\sqrt{\frac{1-C}{1+C}}\frac{\partial q}{\partial \tilde{t}})$.

We integrate $E_{m+n-1}C=C\sqrt{2(1-C^2)}$ along the integral curves of $E_{m+n-1}$, and obtain
$$
C=\frac{1}{\cosh(\sqrt{2}\tilde{t}+d_1)},\ {\rm or}\ \  C=0,
\ {\rm or}\ \  C=\frac{-1}{\cosh(\sqrt{2}\tilde{t}+d_1)}, 
$$
where $d_1$ is constant. By taking the reparametrization $t=\tilde{t}+\frac{d_1}{\sqrt{2}}$, it follows that 
$$
C=\frac{1}{\cosh(\sqrt{2}t)},\ {\rm or}\ \  C=0,
\ {\rm or}\ \  C=\frac{-1}{\cosh(\sqrt{2}t)}.  
$$

Let $\{X_i\}_{i=1}^{m-1}$ and $\{X_p\}_{p=m}^{m+n-2}$ be bases of $T_{x_1}M_1$ and $T_{x_2}M_2$ such that $\{df_1(X_i)\}_{i=1}^{m-1}$ and $\{df_2(X_p)\}_{p=m}^{m+n-2}$ are 
principal curvature vectors of $f_1(M_1)$ and $f_2(M_2)$, respectively. Assume that 
$$
A^{f_1}df_1(X_i)=\lambda^{f_1}_i(x_1)df_1(X_i),\ A^{f_2}df_2(X_p)=\mu^{f_2}_p(x_2)df_2(X_p),\ \forall \ 
1\leq i \leq m-1,\ m\leq p \leq m+n-2, 
$$
where $A^{f_1}$ and $A^{f_2}$ are the shape operators of $f_1(M_1)\hookrightarrow\mathbb{H}^m$ and $f_2(M_2)\hookrightarrow\mathbb{H}^n$,  respectively.
Now, at the point $\Phi(t,x_1,x_2)$, the basis $\{E_\alpha\}_{\alpha=1}^{m+n-1}$ of $\Phi$ can be given as follows:
\begin{equation*}
	\begin{aligned}
		&E_i=\left(df_1(X_i),0\right),\ 
		E_p=\left(0,df_2(X_p)\right), \ 
		1\leq i \leq m-1,\ m\leq p \leq m+n-2, \\ 
		&E_{m+n-1}=\tfrac{\partial}{\partial t}(p(t,x_1),q(t,x_2))
		=\big(u'_{1}(t)\sinh(u_1(t))f_1(x_1)
		+u'_{1}(t)\cosh(u_1(t))N^{f_1}(x_1),\\
		&~~~~~~~~~~~~~~~~~~~~~~~~   u'_{2}(t)\sinh(u_2(t))f_2(x_2)
		+u'_{2}(t)\cosh(u_2(t))N^{f_2}(x_2)\big).
	\end{aligned}
\end{equation*}
The unit normal vector field $N$ of $M$ in $\mathbb{H}^m\times \mathbb{H}^n$ is given by
\begin{equation}\label{eqn:apx3.3}
	\begin{aligned}
		N=&\big(u'_{2}(t)\sinh(u_1(t))f_1(x_1)
		+u'_{2}(t)\cosh(u_1(t))N^{f_1}(x_1),\\
		&\ \ \ -u'_{1}(t)\sinh(u_2(t))f_2(x_2)
		-u'_{1}(t)\cosh(u_2(t))N^{f_2}(x_2)\big).
	\end{aligned}
\end{equation}
It follows that
\begin{equation}\label{eqn:apxCC}
	C=(u'_{2}(t))^2-(u'_{1}(t))^2=1-2(u'_{1}(t))^2.
\end{equation}

Now, by direct calculation, at the point $\Phi(t,x_1,x_2)$, we can get 
\begin{equation}\label{eqn:apx3.5}
	\left\{\begin{aligned}
		&AE_i=-\tfrac{\sinh(u_1(t))
			-\lambda^{f_1}_i(x_1)\cosh(u_1(t))}
		{\cosh(u_1(t))-\lambda^{f_1}_i(x_1)\sinh(u_1(t))}u'_{2}(t)E_i,\ \ 1\leq i\leq m-1,\\
		&AE_p=\tfrac{\sinh(u_2(t))
			-\mu^{f_2}_p(x_2)\cosh(u_2(t))}
		{\cosh(u_2(t))-\mu^{f_2}_p(x_2)\sinh(u_2(t))}u'_{1}(t)E_p,\ \ m\leq p\leq m+n-2,\\
		&AE_{m+n-1}=\left(u'_{2}(t)u''_{1}(t)-u'_{1}(t)u''_{2}(t)\right)E_{m+n-1}.
	\end{aligned}\right.
\end{equation}
Thus, it holds that 
\begin{equation}\label{eqn:apx3.5ss}
\begin{aligned}
&\lambda_i=-\tfrac{\sinh(u_1(t))
	-\lambda^{f_1}_i(x_1)\cosh(u_1(t))}
{\cosh(u_1(t))-\lambda^{f_1}_i(x_1)\sinh(u_1(t))}u'_{2}(t), \ \ \lambda_p=\tfrac{\sinh(u_2(t))
	-\mu^{f_2}_p(x_2)\cosh(u_2(t))}
{\cosh(u_2(t))-\mu^{f_2}_p(x_2)\sinh(u_2(t))}u'_{1}(t),\\
&\lambda_{m+n-1}=u'_{2}(t)u''_{1}(t)-u'_{1}(t)u''_{2}(t).
	\end{aligned}
\end{equation}

By \eqref{eqn:APX2.9} and $E_{m+n-1}=\tfrac{X}{\sqrt{1-C^2}}$, we can have 
 $\nabla_{E_{m+n-1}}{E_{m+n-1}}=0$. 
Then by using \eqref{eqn:4.1}, $AE_{m+n-1}=-\frac{C}{\sqrt{2}}E_{m+n-1}$ and $\nabla_{E_{m+n-1}}{E_{m+n-1}}=0$, 
we can get 
\begin{equation}\label{eqn:APX4.mm1}
	\begin{aligned}
		(\tfrac{\partial^2 p}{\partial t^2},\tfrac{\partial^2 q}{\partial t^2})&
		=\bar{\nabla}_{E_{m+n-1}}{E_{m+n-1}}-g( D_{E_{m+n-1}}{E_{m+n-1}},\tfrac{(p,q)}{\sqrt{2}})\tfrac{(p,q)}{\sqrt{2}}\\
		&~~~~~~~~~~~~~~~~    -g( D_{E_{m+n-1}}{E_{m+n-1}},\tfrac{(p,-q)}{\sqrt{2}})\tfrac{(p,-q)}{\sqrt{2}}\\
		&=-\tfrac{C}{\sqrt{2}}N+\tfrac{1}{2}(p,q)-\tfrac{C}{2}(p,-q)\\
		&=(-C\sqrt{\tfrac{1+C}{2(1-C)}}\tfrac{\partial p}{\partial t},C\sqrt{\tfrac{1-C}{2(1+C)}}\tfrac{\partial q}{\partial t})+(\tfrac{1-C}{2}p,\tfrac{1+C}{2}q),
	\end{aligned}
\end{equation}
where $g$ also denotes the pseudo-Euclidean inner product 
on $\mathbb{R}_2^{m+n+2}$, and $D$ is the canonical pseudo-Euclidean connection on $\mathbb{R}_2^{m+n+2}$. 
It can be checked that equation \eqref{eqn:APX4.mm1} is equivalent to the following 
system of ordinary differential equations:
\begin{equation}\label{eqn:APX4.mm1sx}
	\left\{\begin{aligned}
	&u_1''(t)=-\frac{1}{\sqrt{2}}\Big((u_2'(t))^2-(u_1'(t))^2\Big)u_2'(t),\\
	&u_2''(t)=\frac{1}{\sqrt{2}}\Big((u_2'(t))^2-(u_1'(t))^2\Big)u_1'(t). 
\end{aligned}\right.
\end{equation}

In the following, we divide the discussion into three subcases, according to the expression of the function $C$ on $M$. 

{\bf Case I-ii-1}. $C=\frac{1}{\cosh(\sqrt{2}t)}$. 

In this subcase, from $(u'_{1}(t))^2+(u'_{2}(t))^2=1$, $(u'_{2}(t))^2-(u'_{1}(t))^2=C$ and $C=\frac{1}{\cosh(\sqrt{2}t)}$, we solve the equation \eqref{eqn:APX4.mm1sx} and obtain 
\begin{equation}\label{eqn:APXpq}
	\begin{aligned}
		u_1(t)=\ln\Big(\sqrt{2}\cosh(\tfrac{t}{\sqrt{2}})+\sqrt{\cosh(\sqrt{2}t)}\Big)-a,\ \
		u_2(t)={\rm arcsinh}\Big(-\sqrt{2}\sinh(\tfrac{t}{\sqrt{2}})\Big)-b, 
	\end{aligned}
\end{equation}
where $a,b$ are constants, and the increasing canonical-flow branch has $t<0$. 
From $\lambda_i=\frac{1}{\sqrt{2}}$ and $\lambda_p=\frac{-1}{\sqrt{2}}$ for $1\leq i\leq m-1$ and $m\leq p\leq m+n-2$, by using \eqref{eqn:apx3.5ss} and \eqref{eqn:APXpq}, we have 
$\lambda^{f_1}_i(x_1)=-\tanh(a)$ and $\mu^{f_2}_p(x_2)=-\coth(b)$. It means that 
$f_1(M_1)$ is a totally umbilical hypersurface of 
$\mathbb{H}^m$ with constant principal curvatures $-\tanh(a)$, and $f_2(M_2)$ is a totally umbilical hypersurface of 
$\mathbb{H}^n$ with constant principal curvatures $-\coth(b)$. 
Now, we have 
$$
\left\{
\begin{aligned}
	&p(t,x_1)=\sqrt{2}\cosh(\tfrac{t}{\sqrt{2}})\big(\cosh(a)f_1(x_1)-\sinh(a)N^{f_1}(x_1)\big)\\
&\ \ \ \ \ \ \ \ \ \ \ \ \ \ \ \ \ \ \ \ \ \ \ +\sqrt{\cosh(\sqrt{2}t)}\big(-\sinh(a)f_1(x_1)+\cosh(a)N^{f_1}(x_1)\big),\\
	&q(t,x_2)=\sqrt{\cosh(\sqrt{2}t)}\big(\cosh(b)f_2(x_2)-\sinh(b)N^{f_2}(x_2)\big)\\
	&\ \ \ \ \ \ \ \ \ \ \ \ \ \ \ \ \ \ \ \ \ \ \
	-\sqrt{2}\sinh(\tfrac{t}{\sqrt{2}})\big(-\sinh(b)f_2(x_2)+\cosh(b)N^{f_2}(x_2)\big).
\end{aligned}\right.
$$
One can check that $\cosh(a)f_1(x_1)-\sinh(a)N^{f_1}(x_1)$ is a totally geodesic hypersurface in $\mathbb{H}^{m}$, and $\cosh(b)f_2(x_2)-\sinh(b)N^{f_2}(x_2)$ is a point in $\mathbb{H}^{n}$. 

Since the immersion $\Phi$ is constructed by using the parallel hypersurfaces of $f_1(M_1)$ and $f_2(M_2)$,  
up to an isometry of $\mathbb{H}^m\times \mathbb{H}^n$, we can rewrite $\Phi$ as follows:  
\begin{equation}\label{eqn:ap3.1ss}
	\begin{aligned}
		\Phi_1:I\times \mathbb{H}^{m-1}\times \mathbb{S}^{n-1}&\longrightarrow \mathbb{H}^m\times \mathbb{H}^n,\\
		(t,x,y)&\longmapsto (p(t,x),q(t,y)),
	\end{aligned}
\end{equation}
where
$$
\left\{\begin{aligned}
	&p(t,x)=\Big(\sqrt{2}\cosh(\tfrac{t}{\sqrt{2}})\ x,\sqrt{\cosh(\sqrt{2}t)}\Big),\\
	&q(t,y)=\Big(\sqrt{\cosh(\sqrt{2}t)},-\sqrt{2}\sinh(\tfrac{t}{\sqrt{2}})\ y\Big).
\end{aligned}\right.
$$

{\bf Case I-ii-2}. $C=0$. 

In this subcase, we have $u'_{2}(t)^2-u'_{1}(t)^2=0$ and $u'_{1}(t)^2+u'_{2}(t)^2=1$. Thus, 
without loss of generality, we can assume that $u_{1}(t)=-\frac{t}{\sqrt{2}}-a$ and 
$u_{2}(t)=-\frac{t}{\sqrt{2}}-b$, where $a$ and $b$ are constants.
From $\lambda_i=\frac{1}{\sqrt{2}}$, $\lambda_p=\frac{-1}{\sqrt{2}}$ for $1\leq i\leq m-1$ and $m\leq p\leq m+n-2$, and \eqref{eqn:apx3.5ss}, it follows that  $\lambda^{f_1}_i(x_1)=\mu^{f_2}_p(x_2)=-1$, which means that $f_1(M_1)$ (resp. $f_2(M_2)$) is a horosphere of 
$\mathbb{H}^m$ (resp. $\mathbb{H}^n$) with constant principal curvatures $-1$. 

Since the immersion $\Phi$ is constructed by using the parallel hypersurfaces of $f_1(M_1)$ and $f_2(M_2)$, 
up to an isometry of $\mathbb{H}^m\times \mathbb{H}^n$, we can rewrite $\Phi$ as follows:  
\begin{equation}\label{eqn:ap3.1sss}
	\begin{aligned}
		\Phi_2:I\times \mathbb{R}^{m-1}\times \mathbb{R}^{n-1}&\longrightarrow \mathbb{H}^m\times \mathbb{H}^n,\\
		\big(t,(y_1,\ldots,y_{m-1}),(z_1,\ldots,z_{n-1})\big)&\longmapsto \big(p(t,y_1,\ldots,y_{m-1}),q(t,z_1,\ldots,z_{n-1})\big),
	\end{aligned}
\end{equation}
where
$$
\left\{\begin{aligned}
	&p(t,y_1,\ldots,y_{m-1})=\cosh(\frac{t}{\sqrt{2}})
	(1+\frac{1}{2}\sum_{i=1}^{m-1}y_i^2,y_1,\ldots,y_{m-1},\frac{1}{2}\sum_{i=1}^{m-1}y_i^2)\\
	&~~~~~~~~~~~~~~~~~~~~~~~~~~~~~~~~~~~                      +\sinh(\frac{t}{\sqrt{2}})(-\frac{1}{2}
	\sum_{i=1}^{m-1}y_i^2,-y_1,\ldots,-y_{m-1},1-\frac{1}{2}\sum_{i=1}^{m-1}y_i^2),\\
	&q(t,z_1,\ldots,z_{n-1})=\cosh(\frac{t}{\sqrt{2}})
	(1+\frac{1}{2}\sum_{i=1}^{n-1}z_i^2,z_1,\ldots,z_{n-1},\frac{1}{2}\sum_{i=1}^{n-1}z_i^2)\\
	&~~~~~~~~~~~~~~~~~~~~~~~~~~~~~~~~~~~           +\sinh(\frac{t}{\sqrt{2}})(-\frac{1}{2}
	\sum_{i=1}^{n-1}z_i^2,-z_1,\ldots,-z_{n-1},1-\frac{1}{2}\sum_{i=1}^{n-1}z_i^2).
\end{aligned}\right.
$$

{\bf Case I-ii-3}. $C=-\frac{1}{\cosh(\sqrt{2}t)}$. 

In this subcase, from $(u'_{1}(t))^2+(u'_{2}(t))^2=1$, $(u'_{2}(t))^2-(u'_{1}(t))^2=C$ and $C=\frac{-1}{\cosh(\sqrt{2}t)}$, we solve the equation \eqref{eqn:APX4.mm1sx} and obtain 
\begin{equation}\label{eqn:APXpqs}
		u_1(t)={\rm arcsinh}\Big(-\sqrt{2}\sinh(\tfrac{t}{\sqrt{2}})\Big)-a, \ \ 
		u_2(t)=\ln\Big(\sqrt{2}\cosh(\tfrac{t}{\sqrt{2}})+\sqrt{\cosh(\sqrt{2}t)}\Big)-b,
\end{equation}
where $a,b$ are constants.  
From $\lambda_i=\frac{1}{\sqrt{2}}$, $\lambda_p=\frac{-1}{\sqrt{2}}$ for $1\leq i\leq m-1$ and $m\leq p\leq m+n-2$, by using \eqref{eqn:apx3.5ss} and \eqref{eqn:APXpqs}, it follows that 
$\lambda^{f_1}_i(x_1)=-\coth(a)$ and $\mu^{f_2}_p(x_2)=-\tanh(b)$. It means that 
$f_1(M_1)$ is a totally umbilical hypersurface of 
$\mathbb{H}^m$ with constant principal curvatures $-\coth(a)$, and $f_2(M_2)$ is a totally umbilical hypersurface of 
$\mathbb{H}^n$ with constant principal curvatures $-\tanh(b)$. 
Now, we have 
$$
\left\{
\begin{aligned}
	&p(t,x_1)=\sqrt{\cosh(\sqrt{2}t)}\big(\cosh(a)f_1(x_1)-\sinh(a)N^{f_1}(x_1)\big)\\
	&\ \ \ \ \ \ \ \ \ \ \ \ \ \ \ \ \ \ \ \ \ \ \
	-\sqrt{2}\sinh(\tfrac{t}{\sqrt{2}})\big(-\sinh(a)f_1(x_1)+\cosh(a)N^{f_1}(x_1)\big),\\
	&q(t,x_2)=\sqrt{2}\cosh(\tfrac{t}{\sqrt{2}})\big(\cosh(b)f_2(x_2)-\sinh(b)N^{f_2}(x_2)\big)\\
	&\ \ \ \ \ \ \ \ \ \ \ \ \ \ \ \ \ \ \ \ \ \ \
	+\sqrt{\cosh(\sqrt{2}t)}\big(-\sinh(b)f_2(x_2)+\cosh(b)N^{f_2}(x_2)\big).
\end{aligned}\right.
$$
One can check that $\cosh(a)f_1(x_1)-\sinh(a)N^{f_1}(x_1)$ is a point in $\mathbb{H}^{m}$, and $\cosh(b)f_2(x_2)-\sinh(b)N^{f_2}(x_2)$ is a totally geodesic hypersurface in $\mathbb{H}^{n}$.

Since the immersion $\Phi$ is constructed by using the parallel hypersurfaces of $f_1(M_1)$ and $f_2(M_2)$, 
up to an isometry of $\mathbb{H}^m\times \mathbb{H}^n$, we can rewrite $\Phi$ as follows: 
\begin{equation}\label{eqn:ap3.1ssss}
	\begin{aligned}
		\Phi_3:I\times \mathbb{S}^{m-1}\times \mathbb{H}^{n-1}&\longrightarrow \mathbb{H}^m\times \mathbb{H}^n,\\
		(t,x,y)&\longmapsto (p(t,x),q(t,y)),
	\end{aligned}
\end{equation}
where
$$
\left\{\begin{aligned}
	&p(t,x)=\Big(\sqrt{\cosh(\sqrt{2}t)},-\sqrt{2}\sinh(\tfrac{t}{\sqrt{2}})\ x\Big),\\
	&q(t,y)=\Big(\sqrt{2}\cosh(\tfrac{t}{\sqrt{2}})\ y,\sqrt{\cosh(\sqrt{2}t)}\Big).
\end{aligned}\right.
$$

\vskip 2mm

{\bf Case I-iii}. $\mathbb{S}^m\times\mathbb{H}^n$ for $m\geq2$ and $n\geq2$. 

In this case, $\epsilon_1=1$, $\epsilon_2=-1$.  
If $\kappa=0$, then by \eqref{eqn:4.8s}, it holds 
$\lambda_i=0$ or $\lambda_p=0$ which contradicts  \eqref{eqn:4.9s}, \eqref{eqn:4.9as} 
and $C\neq\pm1$. 
So $\kappa\neq0$ and $\lambda_i\lambda_p\neq0$. 
Then by  \eqref{eqn:4.8s}, we have $\lambda_i=\lambda_j$ and $\lambda_p=\lambda_q$ for  $1\leq i,j\leq m-1$ and $m\leq p,q\leq m+n-2$. 

In the following, according to the values of $m$ and $n$, we divide the discussion into three subcases. 

{\bf Case I-iii-1}. $m\geq 3$. 

In this subcase, by \eqref{eqn:4.7s} and
\eqref{eqn:4.8s}, we have $\lambda_p=\tfrac{1}{\lambda_i}+\lambda_i$ for any $m\leq p\leq m+n-2$. 
By \eqref{eqn:4.7s} and \eqref{eqn:4.9s}, we have 
$\lambda_{m+n-1}=\frac{1+2\lambda_i^2+C}{2\lambda_i}$.  By substituting 
$\lambda_p=\tfrac{1}{\lambda_i}+\lambda_i$ and  $\lambda_{m+n-1}=\frac{1+2\lambda_i^2+C}{2\lambda_i}$ into \eqref{eqn:4.9as}, 
we get 
$\frac{1+2(\lambda_i^2+\lambda_i^4)+C}{2\lambda_i^2}=\kappa$. Combining this with 
\eqref{eqn:4.7s}, we deduce that $C=-1$. It is a contradiction. 

{\bf Case I-iii-2}. $m=2$ and $n\geq 3$. 

In this subcase, by \eqref{eqn:4.8s} and \eqref{eqn:4.7adss}, we have $\lambda_i=\tfrac{-1}{\lambda_p}+\lambda_p$. 
By \eqref{eqn:4.9as} and \eqref{eqn:4.7adss}, we have 
$\lambda_{m+n-1}=\frac{C-1+2\lambda_p^2}{2\lambda_p}$.  We substitute $\lambda_i=\tfrac{-1}{\lambda_p}+\lambda_p$ and $\lambda_{m+n-1}=\frac{C-1+2\lambda_p^2}{2\lambda_p}$ into \eqref{eqn:4.9s}, and using $\kappa=-1+\lambda_p^2$, 
we can have $1-C=0$, which is a contradiction. 

{\bf Case I-iii-3}. $m=n=2$. 

In this subcase, if $\lambda_3=0$, then \eqref{eqn:4.9s} and \eqref{eqn:4.9as} become 
$\frac{1-C}{2}=\frac{-1-C}{2}=\kappa$, which is a contradiction. So $\lambda_3\neq0$. 
By \eqref{eqn:4.9s} and \eqref{eqn:4.9as}, we have $\lambda_1=\frac{-1+C+2\kappa}{2\lambda_3}$ and $\lambda_2=\frac{1+C+2\kappa}{2\lambda_3}$. Combining these 
with \eqref{eqn:4.8s}, we obtain  $\frac{(-1+C+2\kappa)(1+C+2\kappa)}{4\lambda_3^2}=\kappa$. 
It follows from $C\neq\pm1$ that $\kappa\neq0$ and $\kappa=\frac{1}{2}(\lambda_3^2-C-\sqrt{1+\lambda_3^4-2\lambda_3^2C})$ or $\kappa=\frac{1}{2}(\lambda_3^2-C+\sqrt{1+\lambda_3^4-2\lambda_3^2C})$. 

If $\kappa=\frac{1}{2}(\lambda_3^2-C-\sqrt{1+\lambda_3^4-2\lambda_3^2C})$, then by the fact that $E_3\kappa=0$ and $E_3C=-2\lambda_3\sqrt{1-C^2}$, we have 
$$
E_3\lambda_3=\frac{\sqrt{1-C^2}(\lambda_3^2-\sqrt{1+\lambda_3^4-2\lambda_3^2C})}{-\lambda_3^2+C+\sqrt{1+\lambda_3^4-2\lambda_3^2C}}. 
$$
By $\lambda_1=\frac{-1+C+2\kappa}{2\lambda_3}$,  $\lambda_2=\frac{1+C+2\kappa}{2\lambda_3}$ and $E_3C=-2\lambda_3\sqrt{1-C^2}$, 
we have 
\begin{equation}\label{eqn:E312}
E_3\lambda_1=\frac{\sqrt{1-C^2}(1-\lambda_3^2+\sqrt{1+\lambda_3^4-2\lambda_3^2C})}
{2\lambda_3^2(\lambda_3^2-C-\sqrt{1+\lambda_3^4-2\lambda_3^2C})},\ \ 
E_3\lambda_2=\frac{\sqrt{1-C^2}(1+\lambda_3^2-\sqrt{1+\lambda_3^4-2\lambda_3^2C})}
{2\lambda_3^2(\lambda_3^2-C-\sqrt{1+\lambda_3^4-2\lambda_3^2C})}.
\end{equation}

On the other hand, we calculate 
$g\big((\nabla_{E_3}A)E_{1}-(\nabla_{E_1}A)E_{3},E_1\big)$ and $g\big((\nabla_{E_3}A)E_{2}-(\nabla_{E_2}A)E_{3},E_2\big)$. 
By direct calculation, with the use of \eqref{eqn:APX2.9} and Codazzi equation \eqref{eqn:APX2.7}, it holds that  
\begin{equation}\label{eqn:apxe11}
E_3\lambda_1=\frac{(1-C)(\lambda_1^2-\lambda_1\lambda_3)}{\sqrt{1-C^2}}+\frac{\sqrt{1-C^2}}{2},\ 
E_3\lambda_2=\frac{(1+C)(-\lambda_2^2+\lambda_2\lambda_3)}{\sqrt{1-C^2}}+\frac{\sqrt{1-C^2}}{2}.
\end{equation}
By using $\lambda_1=\frac{-1+C+2\kappa}{2\lambda_3}$ and $\lambda_2=\frac{1+C+2\kappa}{2\lambda_3}$, we  compare \eqref{eqn:apxe11} with \eqref{eqn:E312}, it follows that 
$$
\frac{-\sqrt{1-C^2}(1-\lambda_3^2+\sqrt{1+\lambda_3^4-2\lambda_3^2C})}
{\lambda_3^2(-\lambda_3^2+C+\sqrt{1+\lambda_3^4-2\lambda_3^2C})}=0,\ \ 
\frac{\sqrt{1-C^2}(-1-\lambda_3^2+\sqrt{1+\lambda_3^4-2\lambda_3^2C})}
{\lambda_3^2(-\lambda_3^2+C+\sqrt{1+\lambda_3^4-2\lambda_3^2C})}=0. 
$$
Since $C\neq\pm1$, we get a contradiction.

If $\kappa=\frac{1}{2}(\lambda_3^2-C+\sqrt{1+\lambda_3^4-2\lambda_3^2C})$, then by the fact that $E_3\kappa=0$ and $E_3C=-2\lambda_3\sqrt{1-C^2}$, we have 
$$
E_3\lambda_3=-\frac{\sqrt{1-C^2}(\lambda_3^2+\sqrt{1+\lambda_3^4-2\lambda_3^2C})}{\lambda_3^2-C+\sqrt{1+\lambda_3^4-2\lambda_3^2C}}. 
$$
By $\lambda_1=\frac{-1+C+2\kappa}{2\lambda_3}$ and $\lambda_2=\frac{1+C+2\kappa}{2\lambda_3}$, we have 
\begin{equation}\label{eqn:E312s}
	E_3\lambda_1=-\frac{\sqrt{1-C^2}(-1+\lambda_3^2+\sqrt{1+\lambda_3^4-2\lambda_3^2C})}
	{2\lambda_3^2(\lambda_3^2-C+\sqrt{1+\lambda_3^4-2\lambda_3^2C})},\ \ 
	E_3\lambda_2=\frac{\sqrt{1-C^2}(1+\lambda_3^2+\sqrt{1+\lambda_3^4-2\lambda_3^2C})}
	{2\lambda_3^2(\lambda_3^2-C+\sqrt{1+\lambda_3^4-2\lambda_3^2C})}.
\end{equation}

On the other hand, we calculate 
$g\big((\nabla_{E_3}A)E_{1}-(\nabla_{E_1}A)E_{3},E_1\big)$ and $g\big((\nabla_{E_3}A)E_{2}-(\nabla_{E_2}A)E_{3},E_2\big)$. 
By direct calculation, with the use of \eqref{eqn:APX2.9} and Codazzi equation \eqref{eqn:APX2.7}, we get 
\begin{equation}\label{eqn:apxE312s}
E_3\lambda_1=\frac{(1-C)(\lambda_1^2-\lambda_1\lambda_3)}{\sqrt{1-C^2}}+\frac{\sqrt{1-C^2}}{2},\  
E_3\lambda_2=\frac{(1+C)(-\lambda_2^2+\lambda_2\lambda_3)}{\sqrt{1-C^2}}+\frac{\sqrt{1-C^2}}{2}.
\end{equation}

By substituting $\lambda_1=\frac{-1+C+2\kappa}{2\lambda_3}$ and $\lambda_2=\frac{1+C+2\kappa}{2\lambda_3}$ into the right-hand sides of \eqref{eqn:apxE312s}, 
we compare \eqref{eqn:apxE312s} with \eqref{eqn:E312s}, it follows that  
$$
\frac{-\sqrt{1-C^2}(-1+\lambda_3^2+\sqrt{1+\lambda_3^4-2\lambda_3^2C})}
{\lambda_3^2(\lambda_3^2-C+\sqrt{1+\lambda_3^4-2\lambda_3^2C})}=0,\ \ 
\frac{\sqrt{1-C^2}(1+\lambda_3^2+\sqrt{1+\lambda_3^4-2\lambda_3^2C})}
{\lambda_3^2(\lambda_3^2-C+\sqrt{1+\lambda_3^4-2\lambda_3^2C})}=0. 
$$
Since $C\neq\pm1$, we get a contradiction. 
Thus {\bf Case I-iii} does not occur.

\vskip 2mm

{\bf Case II}. $C^2=1$.

In this case, we know that $M$ is an open part of the Riemannian product manifold 
$\Gamma_1\times M^n(\epsilon_2)$  or $M^m(\epsilon_1)\times\Gamma_2$, where $\Gamma_1$ 
(resp. $\Gamma_2$) is a hypersurface of $M^m(\epsilon_1)$ (resp. $M^n(\epsilon_2)$). 
Since $\epsilon_1\epsilon_2\neq0$, it follows that the Riemannian product manifolds 
$\Gamma_1\times M^n(\epsilon_2)$ and $M^m(\epsilon_1)\times\Gamma_2$ 
cannot have constant sectional curvature. 

We have completed the proof of Theorem \ref{thm:1.4}. 
\end{proof}

\begin{remark}
For $\Phi_1$ and $\Phi_3$, restrict the polar parametrizations to $t\ne0$.
Their spherical variables collapse at $t=0$. A smooth extension of $\Phi_1$ is
$$
(x,z)\longmapsto
\big((\sqrt{2+|z|^2}\,x,\sqrt{1+|z|^2}),
(\sqrt{1+|z|^2},z)\big),
\quad x\in\mathbb H^{m-1},\ z\in\mathbb R^n.
$$
Its metric is $(2+|z|^2)g_{\mathbb H^{m-1}}+|dz|^2-
(z\cdot dz)^2/(2+|z|^2)$, isometric to the complete radius-$\sqrt2$
hyperboloid by $(x,z)\mapsto(\sqrt{2+|z|^2}x,z)$.
Interchanging the roles of the factors gives the analogous extension
of $\Phi_3$. These are complete examples. 
\end{remark}

\begin{remark}
The hypersurface $\Phi_1(I\times \mathbb{H}^{m-1}\times \mathbb{S}^{n-1})$ is the high-dimensional version of  Example \ref{exam:4.1ss1ss11ab1}. 
The hypersurface $\Phi_2(I\times \mathbb{R}^{m-1}\times \mathbb{R}^{n-1})$ is the high-dimensional version of  hypersurface $M_{1,1}^{1/2}$, which appears in Example \ref{exam:3.4fg1}. 
When $m=n\geq3$, up to the isometry map
$$
\mathbb{F}:\mathbb{H}^m\times \mathbb{H}^m\rightarrow\mathbb{H}^m\times \mathbb{H}^m,\ \
\mathbb{F}(p,q)=(q,p),
$$
the hypersurface $\Phi_1(I\times \mathbb{H}^{m-1}\times \mathbb{S}^{m-1})$ is congruent to 
the hypersurface $\Phi_3(I\times \mathbb{S}^{m-1}\times \mathbb{H}^{m-1})$. 
\end{remark}

\begin{remark}
The classification of hypersurfaces with constant sectional curvature in products of space forms has been studied in several settings. When the ambient space is a Riemannian product with a Euclidean factor, we refer to \cite{AEG,CT2025,M-T} and the references therein. 
When the ambient space is a Riemannian product with no Euclidean factors, a complete
classification of hypersurfaces with constant sectional curvature in $\mathbb S^2\times\mathbb S^2$ was obtained by the authors in our previous paper \cite{LVWY}.
After the first version of the present paper, which was submitted to arXiv in June 2026, we noticed that Kim, Nikolayevsky and Tojeiro \cite{KNT-1} recently gave a  classification of constant sectional curvature hypersurfaces of dimension
at least three in products of two space forms with arbitrary factor curvatures $\kappa_1,\kappa_2$ (at most one of which is zero).
Some results for the $3$-dimensional case in \cite{KNT-1} are based on our classification results for
$\mathbb S^2\times\mathbb S^2$ in  \cite{LVWY} and for $\mathbb H^2\times\mathbb H^2$ in the present paper.  
Some other related results on hypersurfaces in class $\mathcal A$ and on intrinsically locally symmetric hypersurfaces can be found in
\cite{CT} and \cite{KNT}, respectively. 
\end{remark}

\vskip 0.1cm
\noindent\textbf{Data availability}
Data sharing is not applicable to this article as no datasets were generated or analysed during the current study.

\vskip 0.1cm

\noindent\textbf{Declarations}

\vskip 0.1cm

\noindent\textbf{Conflict of interest} On behalf of all authors, the corresponding author states that there is no conflict of interest.


\vskip 10mm

\begin{flushleft}
Haizhong Li\\
{\sc Department of Mathematical Sciences, Tsinghua University,\\
Beijing, 100084, P.R. China}\\
E-mail: lihz@tsinghua.edu.cn

\vskip 1mm

Luc Vrancken\\
{\sc
Univ. Polytechnique Hauts-de-France,
CERAMATHS-Laboratoire de Mat\'eriaux C\'{e}ramiques et de Math\'{e}matiques,
F-59313 Valenciennes, France;\\
INSA Hauts-de-France, CERAMATHS, F-59313 Valenciennes, France;\\
Department of Mathematics, KU Leuven, Celestijnenlaan 200B, Box 2400, BE-3001 Leuven, Belgium}\\
E-mail: luc.vrancken@uphf.fr

\vskip 1mm

Xianfeng Wang\\
{\sc School of Mathematical Sciences and LPMC, Nankai University,\\
Tianjin 300071, P.R. China}\\
E-mail: wangxianfeng@nankai.edu.cn

\vskip 1mm

Zeke Yao\\
{\sc School of Mathematical Sciences, South China Normal University,\\
Guangzhou 510631, P.R. China}\\
E-mail: yaozk.2021@tsinghua.org.cn 

\end{flushleft}

\end{document}